\documentclass[11pt,reqno,tbtags]{amsart}

\usepackage{amsthm,amssymb,amsfonts,amsmath,wasysym,mathtools,bm,tikz,physics}
\usepackage{siunitx} 
\usepackage{booktabs} 
\usepackage[numbers, sort&compress]{natbib} 
\usepackage{subfigure, graphicx}
\usepackage{vmargin}
\usepackage{setspace}
\usepackage{paralist}
\usepackage[pagebackref=true]{hyperref}
\usepackage{color}
\usepackage[normalem]{ulem}
\usepackage{stmaryrd} 
\usepackage[refpage,noprefix,intoc]{nomencl} 

\usepackage{tcolorbox}
\usetikzlibrary{graphs,graphs.standard,quotes}
\tcbuselibrary{theorems}

\providecommand{\N}{}

\renewcommand{\N}{{\mathbb N}}

\renewcommand{\d}[1]{\ensuremath{\operatorname{d}\!{#1}}}

\newcommand{\E}[1]{{\mathbb E}\left[#1\right]}

\newcommand{\p}[1]{{\mathbb P}\left\{#1\right\}}

\newcommand{\I}[1]{{\mathbf 1}_{[#1]}}

\newcommand\cA{\mathcal A}
\newcommand\cB{\mathcal B}
\newcommand\cC{\mathcal C}

\newcommand\cF{\mathcal F}
\newcommand\cG{\mathcal G}
\newcommand\cH{\mathcal H}

\newcommand\cP{\mathcal P}

\newcommand\cR{{\mathcal R}}
\newcommand\cS{{\mathcal S}}
\newcommand\cT{{\mathcal T}}

\newcommand\cV{\mathcal V}

\newcommand\cX{{\mathcal X}}
\newcommand\cY{{\mathcal Y}}

\newcommand{\bA}{\mathbf{A}} 
\newcommand{\bB}{\mathbf{B}} 
\newcommand{\bC}{\mathbf{C}}

\newcommand{\bF}{\mathbf{F}} 
\newcommand{\bG}{\mathbf{G}} 
\newcommand{\bH}{\mathbf{H}}

\newcommand{\bK}{\mathbf{K}}

\newcommand{\bP}{\mathbf{P}} 
\newcommand{\bQ}{\mathbf{Q}} 
\newcommand{\bR}{\mathbf{R}} 
\newcommand{\bS}{\mathbf{S}} 
\newcommand{\bT}{\mathbf{T}} 
\newcommand{\bU}{\mathbf{U}} 
\newcommand{\bV}{\mathbf{V}} 
\newcommand{\bW}{\mathbf{W}} 
\newcommand{\bX}{\mathbf{X}} 
\newcommand{\bY}{\mathbf{Y}}

\newcommand{\bt}{\mathbf{t}}

\newcommand{\be}{\mathbf{e}}

\newcommand{\bh}{\mathbf{h}}

\newcommand{\boldsigma}{\bm{\sigma}}

\newcommand{\eqdist}{\ensuremath{\stackrel{\mathrm{d}}{=}}}

\newcommand{\st}{\ensuremath{\preceq_{\operatorname{st}}}}

\newcommand{\h}{\operatorname{ht}}

\newcommand{\mult}{\operatorname{mult}}

\providecommand{\eps}{}
\renewcommand{\eps}{\epsilon}
\providecommand{\ora}[1]{}
\renewcommand{\ora}[1]{\overrightarrow{#1}}
\DeclareRobustCommand{\SkipTocEntry}[5]{}

\reversemarginpar
\definecolor{clou}{rgb}{0.8,0.25,0.5125}

\newtheorem{thm}{Theorem}
\newtheorem{lem}[thm]{Lemma}
\newtheorem{obs}[thm]{Observation}
\newtheorem{prop}[thm]{Proposition}
\newtheorem{cor}[thm]{Corollary}
\newtheorem{dfn}[thm]{Definition}

\newtheorem{conj}{Conjecture}

\newtcbtheorem[number within=section]{boxedthm}{}%
{colback=gray!15,colframe=gray!40!black,fonttitle=\bfseries}{th}
\numberwithin{equation}{section}
\numberwithin{thm}{section}

\graphicspath{ {./images/} }
\usepackage{wrapfig}

\def\rest#1{\raise-.2ex\hbox{\ensuremath|}_{#1}}

\newcommand{\dseq}{\mathrm{d}}

\newcommand{\bseq}{\mathrm{b}}
\newcommand{\cseq}{\mathrm{c}}

\newcommand{\diam}{\operatorname{diam}}
\newcommand{\rad}{\operatorname{rad}}

\newcommand{\dist}{\operatorname{dist}}
\newcommand{\len}{\operatorname{len}}
\newcommand{\height}
{\operatorname{ht}}

\newcommand{\grd}{\ensuremath{\mathcal{G}_{\mathrm{d}}}}
\newcommand{\gmd}{\ensuremath{\mathcal{G}^-_{\mathrm{d}}}}
\newcommand{\connd}{\ensuremath{\mathcal{C}_{\mathrm{d}}}}

\newcommand{\cyc}{\mathrm{cyc}}

\renewcommand*{\backref}[1]{}
\renewcommand*{\backrefalt}[4]{%
    \ifcase #1%
          \or $\uparrow$#2%
          \else $\uparrow$#2%
    \fi%
    }
    
\begin{document}
\title{Universal diameter bounds for random graphs with given degrees}
\author{Louigi Addario-Berry}
\author{Gabriel Crudele}

\begin{abstract}
Given a graph $G$, let  $\diam(G)$ be the greatest  distance between any two vertices of $G$ which lie in the same connected component, and let $\diam^+(G)$ be the greatest distance between any two vertices of $G$; so $\diam^+(G)=\infty$ if $G$ is not connected.

Fix a sequence $(d_1,\ldots,d_n)$ of positive integers, and let $\bG$ be a uniformly random connected simple graph with $V(\bG)=[n]\coloneq\{1,\ldots,n\}$ such that $\deg_\bG(v)=d_v$ for all $v \in [n]$. We show that, unless a $1-o(1)$ proportion of vertices have degree $2$, then $\mathbb{E}[\diam(\bG)]=O(\sqrt{n})$. It is not hard to see that this bound is best possible for general degree sequences (and in particular in the case of trees, in which $\sum_{v\in[n]} d_v = 2(n-1)$). We also prove that this bound holds without the connectivity constraint. As a key input to the proofs, we show that graphs with minimum degree $3$ are with high probability connected and have logarithmic diameter: if  $\min(d_1,\ldots,d_n) \ge 3$ and $\bG$ is a uniformly random simple graph with $V(\bG)=[n]$ such that $\deg_\bG(v)=d_v$ for all $v \in [n]$, then $\diam^+(\bG)=O_{\mathbb{P}}(\log n)$; this bound is also best possible. 
\end{abstract}

\keywords{Random graphs, random trees, diameter, switching}
\subjclass[2020]{05C80,60C05}

\maketitle

\tableofcontents

\section{Introduction}
All graphs considered in this paper are finite and contain no isolated vertices. Given a connected graph $G$, we write $\diam(G)=\max(\dist_G(u,v):u,v \in V(G))$, where $\dist_G(u,v)$ is the graph distance from $u$ to $v$ in $G$. For a not-necessarily-connected graph $G$, we write $\diam(G)$ for the greatest diameter of any connected component of $G$ and $\diam^+(G)=\max(\dist_G(u,v):u,v \in V(G))$, so $\diam^+(G)=\infty$ if $G$ is not connected. 

 A {\em degree sequence} is a sequence $\dseq=(d_1,\ldots,d_n)$ of positive integers with even sum. We set $|\dseq|_0=n$ and $|\dseq|_1=\sum_{v\in[n]} d_v$. Write $\grd$ for the set of all simple graphs $G$ with $V(G)=[n]$ such that $\deg_G(v)=d_v$ for all $v \in [n]$, and $\connd=\{G \in \grd:G\mbox{ is connected}\}$. Also, for any $b\in\N$, we write $n_b(\dseq)\coloneq|\{v \in [n]:d_v=b\}|$ for the number of entries of $\dseq$ which equal $b$. We say $\dseq$ is {\em $b$-free} if $n_b(\dseq)=0$. 

For a finite set $\cX$, by $\bX \in_u \cX$ we mean that $\bX$ is a uniformly random element of~$\cX$. Whenever we write $\bX\in_u\cX$, we implicitly assume that $\cX$ is nonempty.
\begin{thm}\label{thm:main}
    There exists $C>0$ such that the following holds for all $\eps>0$.
    Let $\dseq=(d_1,\ldots,d_n)$ be a degree sequence such that $n_2(\dseq)<(1-\eps)n$, and let $\bG \in_u \connd$. Then $\E{\diam(\bG)} \le C(\sqrt{n/\eps}+\log(n)/\eps)$.
\end{thm}

If $\eps$ is bounded away from $0$, then the square root dominates the logarithm, so Theorem~\ref{thm:main} has the following immediate corollary.

\begin{cor}\label{cor:fixed_eps}
    For all $\eps>0$, there exists $C>0$ such that the following holds. Let $\dseq=(d_1,\dots,d_n)$ be a degree sequence such that $n_2(\dseq)<(1-\eps)n$, and let $\bG\in_u\connd$. Then $\E{\diam(\bG)}\leq C\sqrt{n}$.
\end{cor}

More generally, provided that $\eps\geq\log^2(n)/n$, then $\log(n)/\eps\leq \sqrt{n/\eps}$ and the conclusion of Theorem~\ref{thm:main} simplifies to $\E{\diam(\bG)}=O(\sqrt{n/\eps})$. This bound can not in general be improved, as the following example illustrates. Fix positive integers $k$ and $m$, and let $\dseq$ be the sequence 
\[
(\,\underbrace{3, \ldots, 3}_{\text{$m$ times}}\;\;,
  \underbrace{1, \ldots, 1}_{\text{$m+2$ times}},\;
  \underbrace{2, \ldots, 2}_{\text{$km$ times}}\, ).
\]
The elements of $\cC_{\dseq}$ have $n=(k+2)m+2$ vertices and $(k+2)m+1=n-1$ edges, and so are {\em sub-binary trees} --- trees in which the maximum degree is three. Let $\bT \in_u \cC_{\dseq}$, and let $\bT'$ be the {\em homeomorphic reduction} of $\bT$, obtained from $\bT$ by replacing by a single edge each maximal path of $\bT$ all of whose internal vertices have degree 2. Then $\bT'$ is a uniformly random labeled binary tree with internal vertices labeled $\{1,\ldots,m\}$ and leaves labeled $\{m+1,\ldots,2m+2\}$. Such a tree has diameter $\Theta(\sqrt{m})$ in expectation \cite{MR680517}; moreover, the path between two random leaves in $\bT'$ also has expected length $\Theta(\sqrt{m})$. Now, conditionally given $\bT'$, by symmetry, for each edge $uv$ of $\bT'$, the expected length of the path between $u$ and $v$ in $\bT$ is $1+km/(2m+1)=\Theta(k)$. Thus, the path between two random leaves in $\bT$ has expected length $\Theta(k\sqrt{m})=\Theta(\sqrt{kn})$; if we write $n_2(\dseq)=(1-\eps)n$, then $k=\Theta(1/\eps)$ so $\E{\diam(\bT)}=\Theta(\sqrt{n/\eps})$.

For a graph $G$, we write $V(G)$ and $E(G)$ respectively for the vertex set and edge set of $G$. We also use the shorthand $v(G)=|V(G)|$ and $e(G)=|E(G)|$. If $G$ is connected, then the {\em surplus} of $G$ is the quantity $s(G)=1+e(G)-v(G)$, which is the maximum number of edges that can be deleted from $G$ without disconnecting $G$. Note that for a degree sequence $\dseq=(d_1,\ldots,d_n)$, writing $s(\dseq) = 1+\sum_{v\in[n]} d_v/2-n=1+|\dseq|_1/2-|\dseq|_0$, then for any $G \in \cC_{\dseq}$ we have $s(G)= s(\dseq)$. 

The proof of Theorem~\ref{thm:main} actually shows that $\E{\diam(\bG)}=O(\sqrt{n/\eps}+\log(s+1)/\eps)$ where $s$ is the surplus of $\dseq$. Since $s\leq e(\bG)\leq {n\choose 2}$, we have $\log(s+1)\leq 2\log n$, and this is how the $\log n$ term in Theorem~\ref{thm:main} arises. We believe it should be possible to replace the $\log(s+1)/\eps$ term by $\log(\eps n+1)/\eps\leq \sqrt{n/\eps}$. Since $\log(\eps n+1)/\eps\leq \sqrt{n/\eps}$, this would strengthen the conclusion of the theorem, yielding that  $\E{\diam(\bG)}=O(\sqrt{n/\eps})$ whatever the value of $\eps$.

The preceding example shows that the bound stated in Theorem~\ref{thm:main} can be saturated by certain ensembles of random trees, which have surplus zero; for any fixed $s > 0$, it is also not very hard to see that degree sequences of the form 
\begin{equation}\label{eq:deg3seq}
(\underbrace{3, \ldots, 3}_{\text{$m+s$ times}}\;\;\;,
  \underbrace{1, \ldots, 1}_{\text{$m+2-s$ times}},\;\;\;\; \underbrace{2, \ldots, 2}_{\text{$km$ times}}\,\;)
\end{equation}
yield graphs of surplus $s$ and diameter $\Theta(k\sqrt{m/s}\log (s+1))=\Theta(\sqrt{kn/s}\log(s+1))$, which 
likewise saturate the bound. (Many other examples are possible. We write $\log(s+1)$ rather than $\log s$ in order to include the case $s=1$.) 

Our second result establishes the same bound as our first result, but for uniform samples from $\cG_{\dseq}$ rather than from $\cC_{\dseq}$.
\begin{thm}\label{thm:disco}
    There exists $C>0$ such that the following holds for all $\eps>0$.
    Let $\dseq=(d_1,\ldots,d_n)$ be a degree sequence such that $n_2(\dseq)<(1-\eps)n$, and let $\bG \in_u \grd$. Then $\E{\diam(\bG)} \le C(\sqrt{n/\eps}+\log(n)/\eps)$.
\end{thm}
By the same reasoning that yielded Corollary~\ref{cor:fixed_eps} from Theorem~\ref{thm:main}, we obtain the following corollary of Theorem~\ref{thm:disco}. 
\begin{cor}\label{cor:fixed_eps_disco}
    For all $\eps>0$, there exists $C>0$ such that the following holds. Let $\dseq=(d_1,\dots,d_n)$ be a degree sequence such that $n_2(\dseq)<(1-\eps)n$, and let $\bG\in_u\grd$. Then $\E{\diam(\bG)}\leq C\sqrt{n}$. 
\end{cor}

An important step in the proofs of the above results is a logarithmic bound on the diameters of random graphs with given degrees, when the minimum degree is at least three, which is of independent interest.
\begin{thm}\label{thm:ker_diam_main}
There exists an absolute constant $C>0$ such that the following holds. Let $\dseq=(d_1,\ldots,d_n)$ be a degree sequence such that $n_1(\dseq)=n_2(\dseq)=0$, and let $\bG \in_u \cG_{\dseq}$. Then $\p{\diam^+(\bG)\geq62\log n}\le  C/n$. 
\end{thm}

As an immediate corollary of Theorem~\ref{thm:ker_diam_main}, we obtain that the expected diameter of random connected graphs with minimum degree $3$ is at most logarithmic. \begin{cor}
There exists an absolute constant $C>0$ such that the following holds. Let $\dseq=(d_1,\ldots,d_n)$ be a degree sequence such that $n_1(\dseq)=n_2(\dseq)=0$, and let $\bG \in_u \cC_{\dseq}$. Then $\E{\diam(\bG)}\le C \log n$. 
\end{cor}
Remarkably, this fact does not appear to have been previously proven; it was highlighted as a question in \cite{gao}.

\subsection{Three conjectures}
\setlength{\leftmargini}{0em}

The examples sketched above show that for any fixed integer $s \ge 0$, there are degree sequences with surplus $s$ which saturate the bound of Corollary~\ref{cor:fixed_eps}, up to a multiplicative factor of order $O((s+1)^{-1/2}\log(s+2))$. However, it seems likely that when $s$ is large, the bound can always be improved.

\begin{conj}\label{conj:large_surplus}
Fix $\eps>0$ and let $(\dseq^n, n\ge1)$ be a sequence of degree sequences with $|\dseq^n|_0=n$ and  $n_2(\dseq^n)\leq(1-\eps)n$, and suppose that $s_n=s(\dseq^n) \to \infty$ as $n \to \infty$. For $n \ge 1$ let $\bG_n \in_u \cC_{\dseq^n}$. Then $\E{\diam(\bG_n)}=O((1+\sqrt{n/s_n})\log s_n)$ as $n \to \infty$.
\end{conj}
It would already be interesting to show that $\E{\diam(\bG_n)}=o\left(\sqrt{n}\right)$ as $n \to \infty$ under the assumptions of the above conjecture.

To motivate the second conjecture, it is useful to first present a result from \cite{addarioberry2024random}. 
A {\em rooted tree} is a tree $T$ with a distinguished node $\rho(T)$ called the root. In a rooted tree, each node aside from the root has a parent; the degree of the root is its number of children; and the degree of each non-root node is its number of children plus one. The {\em height} $\mathrm{ht}(T)$ is the greatest graph distance of any vertex from the root.

A {\em child sequence} is a finite sequence $\cseq=(c_v,v \in V)$ of non-negative integers with $\sum_{v \in V} c_v \le |V|-1$. For a nonnegative integer $b$, we write $n_b(\cseq)=|\{v \in V: c_v=b\}|$. Child sequence $\cseq$ is {\em sub-binary} if $n_b(\cseq)=0$ for all $b>2$. 
A rooted tree $T$ has child sequence $\cseq=(c_v,v\in V)$ if for all $v \in V$, vertex $v$ has $c_v$ children in $T$. (Note that for $T$ to have child sequence $\cseq$ it is necessary that $\sum_{v \in V} c_v = |V|-1$.)
The set of all rooted trees with child sequence $\cseq$ is denoted $\cT_\cseq$. 

The paper \cite{addarioberry2024random} showed that random sub-binary rooted trees have stochastically maximal height among trees with a given number of leaves and one-child vertices. More precisely, if $\bseq$ is a sub-binary child sequence and $\cseq$ is any child sequence with $n_0(\cseq)=n_0(\bseq)$ and $n_1(\cseq)=n_1(\bseq)$, and $\bT_\cseq \in_u \cT_\cseq$ and $\bT_\bseq \in_u \cT_\bseq$, then $\mathrm{ht}(\bT_{\cseq})\st \mathrm{ht}(\bT_{\bseq})$. We conjecture that a similar result holds for random graphs with a given degree sequence and surplus. 
\begin{conj}\label{conj:stoch} Fix degree sequences $\dseq$ and $\dseq'$ such that 
$s(\dseq)=s(\dseq')$. 
Suppose that $n_1(\dseq)=n_1(\dseq')$, that  $n_2(\dseq)=n_2(\dseq')$, and that $n_b(\dseq)=0$ for $b>3$. 
Let $\bG\in_u \cC_{\dseq}$ and let $\bG'\in_u \cC_{\dseq'}$. 
Then  $\diam(\bG')\st \diam(\bG)$. 
\end{conj}
In view of the discussion around \eqref{eq:deg3seq}, this conjecture would imply that random graphs with surplus $s>0$ always have expected diameter at most $O((1+\sqrt{n/s})\log(s+1))$ unless almost all vertices have degree $2$, and so would strengthen Corollary~\ref{cor:fixed_eps} for such ensembles. 

\medskip
In keeping with the heuristic that as the degrees of random graphs increase, their diameters decrease, we propose the following conjecture. Informally, it states that among random graphs with given degree sequences with minimum degree $d$, the random $d$-regular graphs have the largest diameter, at least to first order. For fixed $d \ge 3$, let $\mathcal{S}_{n,d}$ be the set of all degree sequences $(d_1,\ldots,d_n)$ with $\min(d_1,\dots,d_n) \ge d$.
\begin{conj}
    For all $d \ge 3$ and $\eps > 0$, it holds that 
    \[
    \limsup_{n \to \infty}\sup_{\mathrm{d} \in \mathcal{S}_{n,d}} \mathbb{P}_{\bG \in_u \mathcal{G}_{\mathrm{d}}}(\diam(\bG)>(1+\eps)\log_{d-1} n) =0.
    \]
\end{conj}
It is not hard to see that for all $d \ge 3$, a random $d$-regular graph with $n$ vertices has diameter $(1+o(1))\log_{d-1} n$ with high probability and in expectation, which justifies the informal statement of the conjecture just before its formal statement. The $d=3$ case of the conjecture is already interesting, and may be viewed as a slight weakening of Conjecture~\ref{conj:stoch}.


\subsection{Related work}

There is a large body of literature on diameters of random graphs. The first result on the subject we could find is due to Burtin \cite{MR369168}, who proved two-point concentration for the diameter of the Erd\H{o}s-R\'enyi random graph $G(n,p)$ when $p=p(n) \gg (\log n)/n$. (A slightly weaker result of the same flavour was independently proved by Bollob\'as \cite{MR621971}, and related results were proved around the same time by Klee and Larman \cite{MR627647}.) \L uczak showed that in the subcritical random graph, when $p=p(n)$ satisfies $n^{1/3}(1-np) \to \infty$, the connected component of $G(n,p)$ with maximal diameter is with high probability a tree, and found the asymptotic behaviour of this maximal diameter throughout the subcritical regime. Much later, Chung and Lu \cite{MR1826308}, and then Riordan and Wormald \cite{MR2726083}, substantially extended Burtin's result. The latter paper proved two-point concentration of the  diameter of $G(n,p)$ whenever $np \to \infty$, and additionally found tight estimates for the diameter throughout the supercritical regime  $n^{1/3}(np-1)\to \infty$. Independently and at roughly the same time, the diameter of $G(n,p)$ in the critical regime $n^{1/3}(np-1)=O(1)$ was studied in \cite{MR2892951,MR2435849}. Combined, the preceding results fully describe all possibilities for the asymptotic behaviour of the diameter of $G(n,p)$.

The diameters of many other random graph ensembles have also been studied, including the cases of random regular graphs \cite{MR685038,MR4245175}; scale-free random graphs \cite{MR2057681,MR3938775,MR3710797}; inhomogeneous random graphs \cite{MR3845099,MR2484340,MR2438903}; preferential attachment models \cite{MR2602984,MR4144085,MR4144077,log,asymp}; the configuration model under various degree assumptions \cite{MR2362640,MR3704860}; hyperbolic random graphs \cite{MR3813231}; and in directed settings \cite{MR4533728,MR4583664,MR4761372}. 

The setting we consider in this paper -- the uniform distribution over simple graphs with a given degree sequence -- is an extremely natural reference model that, to date, appears not to have been studied in generality. Past work on the diameters of random graphs with given degree sequences has focussed on regimes where the degree sequences are tame enough to be studied via the configuration model \cite[Chapter 7]{MR3617364}, and has largely focussed on settings exhibiting small-world phenomena (and in particular logarithmic diameters). It bears emphasis that much of the previous work on random graphs with given degree sequences also focusses on {\em sequences} of degree sequences which are in some sense convergent. For example, a frequent assumption is that the empirical distribution of the degree of a uniformly random vertex converges in distribution to a limiting random variable with (some number of) finite moments, and that the empirical moments likewise converge to the moments of the limiting random variable. By contrast, the results of this work are entirely non-asymptotic. As such, while they can be applied to obtain bounds in ``asymptotic'' settings, they apply far more generally.

While we believe that Theorem~\ref{thm:ker_diam_main} is an important contribution to the understanding of the logarithmic-diameter regime, we also hope our paper will spur interest in developing a finer understanding of typical and extreme distances in random graphs with ``sparser'' degree sequences, where it is necessary to study both the cycle structure (core) of the graph and the structure of the trees which attach to the core. The study of other graph properties, such as mixing times and cutoff phenomena, would also be interesting for such graphs.

\subsection{Notation}
We now introduce some notation that will be used throughout the paper, including in the proof overview which immediately follows. 

The {\em height} of a rooted tree $T$, denoted $\height(T)$, is the greatest distance of any vertex from the root of $T$. A {\em rooted forest} is a set $F=\{T_1,\ldots,T_k\}$ of rooted trees with disjoint vertex sets; the {\em root set} of $F$ is $\{\rho(T_i):i \in [k]\}$. The height of $F$, denoted $\height(F)$, is the greatest height of any of its constituent trees. 

Rooted forest $F=\{T_1,\ldots,T_k\}$ has child sequence $\cseq=(c_v,v\in V)$ if $\bigcup_{i \in [k]} V(T_i)= V$ and for all $v \in V$, vertex $v$ has $c_v$ children in its tree in $F$. 
 
The set of all forests with child sequence $\cseq$ is denoted $\cF_{\cseq}$. Given a subset $A$ of the vertex set of a child sequence $\cseq$, we write $\cF_{\cseq,A}$ for the set of rooted forests with child sequence $\cseq$ and root set $A$, and $\cseq_{A}$ for the restriction of $\cseq$ to $A$. We use the analogous notation for restrictions of degree sequences. 

For positive integers $m$ and $h$, an {\em $m$-composition of $h$} is a sequence of $m$ nonnegative integers summing to $h$. We write $\cP_{m,h}$ for the set of all $m$-compositions of $h$. 

We shall encode multigraphs as triples $(V,E,\iota)$, where $V$ is the vertex set, $E$ is the edge set, and $\iota:E \to {V\choose 2}\cup {V \choose 1}$ is the endpoint map. For all multigraphs considered in this work, the vertex set $V$ is totally ordered, and for $e \in E$ we define $e^-=\min \iota(e)$ and $e^+=\max \iota(e)$. We say that edges $e,f\in E$ are \emph{parallel} if $\iota(e)=\iota(f)$, and write $\mult(e)$ for the number of edges parallel to $e$. 

Finally, we note that the notation $\cG_\dseq$, $\connd$ and the like makes sense even if $\dseq=(d_v,v \in V)$ is a degree sequence indexed by a finite set $V$ which is not an initial segment of $\N$; we will require this at a few points in the proof. Relatedly, for a graph $G=(V,E)$ we refer to $\dseq(G)=(\deg_G(v),v \in V)$ as the degree sequence of $G$, where for concreteness we list the vertex degrees according to the total order of $V$.

\subsection{Proof outline}\label{sec:pf_outline}
Given a graph $G$, write $C(G)$ for the maximum subgraph of $G$ of minimum degree at least $2$, and call $C(G)$ the {\em core} of $G$. Thus, if $G$ is a forest, then $C(G)$ is empty. We say a graph $C$ is a core if it has minimum degree $2$. For a vertex $v \in V(G)$ not in a tree component of $G$, we write $\alpha(v)$ for the (unique) core vertex at minimum graph distance from $v$. If $v$ lies in a tree component of $G$, then we take $\alpha(v)$ to be the least labeled leaf in this tree. 

We let $F(G)$ be the rooted forest obtained from $G$ as follows. For each edge $rs$ with $r \not\in V(C(G))$ and $s \in V(C(G))$, remove edge $rs$ and let the resulting tree, which contains $r$, have root $r$. Second, for each tree component $T'$ of $G$, delete the leaf of $T'$ with minimum label and root the resulting tree $T$ at the unique neighbour of that leaf. (All the graphs we consider in this paper have $V(G)\subset\N$, so ``minimum label'' makes sense.) Finally, remove all vertices of $C(G)$ and their remaining incident edges. Setting $c_v=\mathrm{deg}_G(v)-1$ for $v \in V(F(G))$, then $F(G)$ has child sequence $\cseq=(c_v,v \in V(F(G)))$. 

Assuming that $G$ is connected, the {\em kernel} of $G$ is the multigraph $K(G)$ constructed from $G$ as follows. If $s(G)\geq 2$, then $K(G)$ is obtained from the core $C=C(G)$ by replacing each maximal path all of whose internal vertices have degree 2 by a single edge. 
If $s(G)=1$ and $G$ has a leaf, then we let $\ell$ be the leaf of minimum label, and define $K(G)$ to be a single loop edge at the core vertex $\alpha(\ell)$ nearest to $\ell$. In the remaining cases, where $G$ is either a tree or a cycle, $K(G)$ is empty. Note that if $s(G)\geq 2$, then $K(G)=K(C)$ and $V(K)=\{v\in V(C):\deg_C(v)\geq 3\}$. If $G$ is not connected but has connected components $G_1,\dots,G_k$, then $K=K(G)$ is the disjoint union of $K(G_1),\dots,K(G_k)$. For $e \in E(K)$ we write $C(e)$ for the path in $C$ corresponding to edge $e$, and we write $|C(e)|$ for the number of internal vertices on $C(e)$.

Observe that $C(G)$ can be constructed from $G$ by repeatedly deleting leaves, and $K(G)$ can be constructed from $C(G)$ by repeatedly replacing a path with two edges whose internal vertex has degree 2 by a single edge. Deleting leaves and replacing such paths of length 2 by edges are both operations that do not change surplus, so $s(G)=s(C(G))=s(K(G))$.

Now write $C=C(G)$, $F=F(G)$, and $K=K(G)$. For $u,v \in V(G)$ in the same component of $G$, a shortest path between $u$ and $v$ in $G$ may be decomposed into three parts: the (unique) shortest paths from $u$ to $\alpha(u)$ and from $v$ to $\alpha(v)$, and a shortest path between $\alpha(u)$ and $\alpha(v)$. Unless $\alpha(u)=\alpha(v)$, this shortest path between $\alpha(u)$ and $\alpha(v)$ lies in $C$ and can be further decomposed into three parts: a path from $\alpha(u)$ to a vertex $u'\in V(K)$, a path from $\alpha(v)$ to a vertex $v'\in V(K)$, and a shortest path between $u'$ and $v'$ in $C$. It follows that 
\begin{align}\label{eq:diamg_bd}
\begin{split}
   \diam(G) & \le 2(\height(F)+1)+\diam(C)\\
& \le  2(\height(F)+1)+(\diam(K)+2)\max(|C(e)|+1:e \in E(K)). 
\end{split}
\end{align}
Therefore, in order to bound $\diam(G)$, it suffices to control $\h(F)$, $\max(|C(e)|:e \in E(K))$, and $\diam(K)$. In the next three subsections, we  explain our techniques for each of these, in order.

\subsubsection{The height of the forest.}\label{sec:forest height} 
Fix a degree sequence $\dseq=(d_1,\ldots,d_n)$, let $\bG\in_u \grd$, and write $\bC=C(\bG)$, $\bF=F(\bG)$, and  $\bK=K(\bG)$. 
Writing $\bR=\{v \in [n]: d_v=1,v \not\in V(\bF)\}$, then $V(\bC)=[n]\setminus(\bR\cup V(\bF))$, and $\bF$ has $\mathbf{s}$ roots, where 
\[
\mathbf{s}=|\bR|+\sum_{v \in V(\bC)} (d_v-\deg_{\bC}(v))=\sum_{v \in [n]\setminus  
V(\bF)} (d_v-\deg_{\bC}(v))\, .
\] 

To reconstruct $\bG$ from $\bF$ and $\bC$, it suffices to specify the partition $(\bF^{v},v \in \bR\cup V(\bC))$ of the trees of $\bF$, where $\bF^{v}$ is the set of trees of $\bF$ whose root is incident to $v$ in $\bG$. 
The number of possible partitions is 
\[
\frac{\mathbf{s}!}{\prod_{v \in V(\bC)} (d_v-\deg_{\bC}(v))!}=\frac{\mathbf{s}!}{\prod_{[n]\setminus V(\bF)} (d_v-\deg_{\bC}(v))!}. 
\]
Since this number only depends on $\bF$ through its vertex set $V(\bF)$, it follows that conditionally given $V(\bF)$, $\bF$ is a uniformly random forest with child sequence $\cseq=(c_v,v \in V(\bF))$. 

Next, let $\bF'$ be the homeomorphic reduction of $\bF$, obtained by replacing by a single edge each maximal path of $\bF$ all of whose internal vertices have exactly one child. Write $h=|V(\bF)\setminus V(\bF')|$ and $m=e(\bF')$. To recover $\bF$ from $\bF'$ it suffices to specify, for each edge $uv\in E(\bF')$, the (possibly empty) ordered sequence of one-child vertices appearing on the $u-v$ path in $\bF$. The number of possibilities for this data is 
\[
h! |\cP_{m,h}| = h! {h+m-1 \choose m-1}.
\]
This number only depends on $\bF$ and $\bF'$ through their vertex sets (since the degrees are specified, $e(\bF')$ is determined by $V(\bF')$), so conditionally given $V(\bF')$, $\bF'$ is a uniformly random forest with child sequence $\cseq'=(c_v,v \in V(\bF'))$. 

We now appeal to one of the main results of \cite{addarioberry2024random}, which yields\footnote{See (2.8) in the proof of Theorem~6 in~\cite{addarioberry2024random}. That result is stated for random trees with given child sequences, not random forests, but the extension to random forests is an immediate consequence of Lemma~\ref{lem:subgaussian}.} the following theorem.
\begin{thm}\label{thm:ad24}
    Fix a 1-free child sequence $\cseq'=(c_v,v \in V)$, and let $\bF'\in_u\cF_{\cseq'}$. Then for all $x\geq0$,
    \begin{equation*}
        \p{\height(\bF')>x\sqrt{|V|}}\le 4\exp\left(-\frac{x^2}{2^8}\right).
    \end{equation*}
\end{thm}
If $\dseq$ is $2$-free, then $\cseq=(c_v,v \in V(\bF))$ is $1$-free, so $\bF'=\bF$; since $\diam(\bG)-\diam(\bC) \le 2(\height(\bF)+1)$ and $|V(\bF)|\le n$, in this case the above theorem immediately implies that $\E{\diam(\bG)}\le \E{\diam(\bC)}+O(\sqrt{n})$.

If we assume only that $n_2(\dseq) < (1-\eps)n$, then to bound $\diam(\bG)$ as above, we also need to control $\height(\bF)-\height(\bF')$. For this, we will essentially show that, on average, each edge in $\bF'$ is subdivided $O(1/\eps)$ times by one-child vertices in $\bF$ and, moreover, the number of times that a long path in $\bF'$ is subdivided by one-child vertices is concentrated enough that $\E{\height(\bF)\mid\bF'} = O((\height(\bF')+\log n)/\eps)$. (In fact, we will also need to control the effect of subdivisions on distances in the core, and we will do this at the same time as we handle their effect on $\bF$, so our eventual argument is slightly different from the sketch in this paragraph; but the idea is the same.)

\subsubsection{The greatest length of a core path.}
We next explain our approach to bounding $\max(|\bC(e)|: e \in E(\bK))$. To get the key ideas across, it is useful to assume that $\dseq$ is $2$-free, and also to condition on the events that $\bK$ is a simple graph and that for all $v \in V(\bK)$, $\deg_{\bG}(v)=\deg_\bK(v)$, so that $v$ has no neighbours in $\bF$. Write $m=e(\bK)$, and list the edges of $K$ in lexicographic order as $e_1,\ldots,e_m$, where $e_i=u_iv_i$ and $u_i < v_i$. If $|\bC(e_i)|>0$ then let $\bT_i$ be the subgraph of $\bG$ induced by the set of vertices $\{v \in V(\bG):\alpha(v) \in V(\bC(e_i)) \setminus\{u_i,v_i\}\}$. Then let $u_i'$ and $v_i'$ be the vertices of $\bT_i$ nearest to $u_i$ and $v_i$, respectively; it is possible that $u_i'=v_i'$. 
Now define a tree $\bT$ as follows: for each $i \in [m]$ with $\bT_i$ non-empty let $j>i$ be minimal such that $\bT_j$ is non-empty, and add an edge from $v_i'$ to $u_j'$; if no such $j$ exists then instead add an edge from $v_i'$ to vertex $0$. Root the resulting tree at $u_{i_0}'$, where $i_0=\min(i \in [m]: \bT_i\mbox{ is non-empty})$. An example appears in Figure~\ref{fig:biasedtree}. 
\begin{figure}[htb]
    \centering
    \includegraphics[width=0.4\linewidth,page=2]{kernel2.pdf}
    \includegraphics[width=0.4\linewidth,page=1]{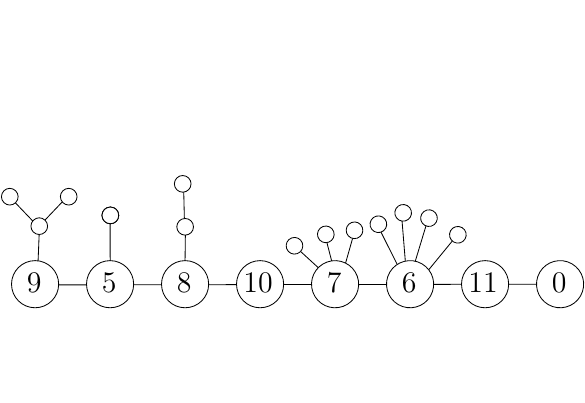}
    \caption{Left, a connected graph $\bG$, with its kernel vertices circled in bold. Only labels of vertices in the core are shown, for readability. Right: The tree $\bT$ constructed from (the trees hanging from the core paths of) $\bG$, with root $9$. Listing the kernel edges as $e_1,\ldots,e_6$ in lexicographic order, then the numbers of internal vertices of the paths $C(e_1),\ldots,C(e_6)$ are $(2,3,0,2,0,0)$, which is a partition of $7=\height_{\bT}(0)$ with $6=e(K(\bG))$ parts. The graph $\bG$ can be recovered from $\bT$ together with this partition.}
    \label{fig:biasedtree}
\end{figure}

Write $B=[n]\setminus V(\bK)$, and let $\cseq=(c_v,v \in B\cup \{0\})$ be the child sequence with $c_0=0$ and $c_v=d_v-1$ for $v \in B$. Then $\bT$ has child sequence $\cseq$. Moreover, the graph $\bG$ can be recovered from $\bT$ together with the data $(|\bC(e_i)|,i \in [m])$. The set of internal vertices of $\bC(e_1),\ldots,\bC(e_m)$ are precisely the vertices on the path in $\bT$ from its root to the leaf $0$ (excluding $0$); thus, conditionally, given $\bT$, the number of possibilities for this data is $|\cP_{m,\h_\bT(0)}|$. It follows that conditionally given (simple) $\bK$ and given that $(\deg_\bG(v),v \in V(\bK))=(\deg_{\bK}(v),v \in V(\bK))$, 
\[
(\bT,(|\bC(e_i)|,i \in [m])) \in_u \{(T,P):T \in \cT_{\cseq}, P \in \cP_{m,\height_T(0)}\}\, .
\]
Thus, conditionally given $m$ and given its child sequence $\cseq$, $\bT$ is distributed as a random tree with child sequence $\cseq$ chosen with probability proportional to the number of $m$-compositions of the root-to-$0$ path. We build on the techniques of \cite{addarioberry2024random} in order to  establish that, up to a parity constraint, among all such composition-biased random rooted trees with a given 1-free child sequence, the height of a fixed leaf is stochastically maximized when the child sequence is binary. More precisely, we prove the following theorem. The \emph{parity} of an integer $n$ is the remainder of $n$ modulo 2.
\begin{thm}\label{thm:st_dom}
    Fix positive integers $m$ and $n$, and let $\pi$ be the parity of $n$. Let $\cseq=(c_0,\dots,c_n)$ be a $1$-free child sequence with $c_0=0$ and let $\bseq=(b_0,\dots,b_{n+\pi})$ be a binary child sequence with $b_0=0$. Then with $(\bT,\bP)\in_u\{(T,P): T\in \cT_\cseq, P\in\cP_{m,\h_T(0)}\}$ and $(\bS,\bQ)\in_u\{(S,Q): S\in \cT_\bseq, Q\in\cP_{m,\h_S(0)}\}$,
    \begin{equation*}
        \h_{\bT}(0) \st \h_{\bS}(0).
    \end{equation*}
\end{thm}
The proof of this theorem, which appears in Section~\ref{sec:st_dom} is surprisingly challenging, and requires an auxiliary stochastic domination result which is proved by studying the asymptotic behaviour of a Markov chain reminiscent of importance sampling. Once the stochastic domination is proved, however, we are able to establish the following sub-Gaussian tail bound by proving a corresponding bound for the height of a fixed node in a composition-biased binary tree. 
\begin{thm}\label{thm:tail_bound}
    Fix positive integers $m$ and $n$ with $n\geq 64m$. Let $\cseq=(c_0,\dots,c_n)$ be a 1-free child sequence with $c_0=0$ and let $\bh=\h_\bT(0)$, where $(\bT,\bP)\in_u\{(T,P):T\in\cT_\cseq,P\in\cP_{m,\h_T(0)}\}$. Then for all $x\geq 0$,
    \begin{equation*}
        \p{\bh\geq 2\sqrt{mn}+x\sqrt{n}} \leq \exp\left(-\frac{x^2}{3}+4\right).
    \end{equation*}
\end{thm}
The proof of Theorem~\ref{thm:tail_bound} appears in Section~\ref{sec:tailbd}. 
It then follows from fairly straightforward arguments for the sizes of entries in a uniformly random composition that for any fixed kernel edge $e_i$, $|\bC(e_i)|/(1+\sqrt{n/m})$ has sub-exponential tails, and so in particular $\E{\max(|\bC(e_i)|:i\in[m])} = O((1+\sqrt{n/m})\log m)$ by a union bound. (The arguments become somewhat more complicated when we remove the assumptions that $\bK$ is simple and that $\dseq$ is $2$-free, but the key ideas are present here.) 

\subsubsection{The diameter of the kernel.} 
Given the above, the remaining step in the proof of Theorem~\ref{thm:main} is to bound the diameter of the kernel. At this point, we stop assuming that the kernel is a simple graph and that $\dseq$ is $2$-free.

Write $\cG_{\dseq}^-$ for the set of all (not necessarily connected) graphs with degree sequence $\dseq$ and no cycle components, and let $\bC^-$ be obtained from $\bC$ by deleting all cycle components. Note that, conditionally given the degree sequence $\dseq^-=(\deg_{\bC^-}(v),v \in V(\bC^-))$ of $\bC^-$, then $\bC^- \in_u \cG_{\dseq^-}^-$, and $\dseq^-$ is a degree sequence with minimum degree $2$. Moreover, $\bK=K(\bC^-)$; so, to prove uniform tail bounds on the diameter of the kernel, it suffices to prove corresponding bounds for kernels $K(\bC)$ of random graphs $\bC \in_u \cG_{\dseq}^-$ that hold uniformly over degree sequences with minimum degree $2$. We do so in Section~\ref{sec:ker_diam}, where we prove the following theorem. 

\begin{thm}\label{thm:ker_diam}
    Fix positive integers $n\leq \ell$ and a 1-free degree sequence $\dseq$ of the form $\dseq=(d_1,\dots,d_\ell)=(d_1,\dots,d_n,2,\dots,2)$ with $d_v\geq 3$ for all $v\in[n]$. 
    Take $\bC\in_u\gmd$ and let $\bK=K(\bC)$. Write $2m=\sum_{v\in[n]}d_v$ so that if $C\in\cG_\dseq^-$, then $K(C)$ has $m$ edges. Then
    \begin{equation*}
        \p{\diam^+(\bK)\geq31\log m}= O(1/m).
    \end{equation*}
\end{thm}

Note that because $\dseq$ is 1-free, graphs $C\in\cG_\dseq^-$ are cores, each of whose components have surplus at least 2. 
Also, observe that Theorem~\ref{thm:ker_diam_main} follows immediately from the case $\ell=n$ of Theorem~\ref{thm:ker_diam}, since in this case $\cG_{\dseq}=\cG_{\dseq}^-$, $\bK=\bC$, and $3n\leq 2m\leq n^2$.

To prove Theorem~\ref{thm:ker_diam}, we consider a breadth-first exploration of the kernel of the random core. At each time, there will be a queue of vertices which have been discovered but whose neighbourhoods still need to be explored. We will show that with high probability, until the queue is extremely large, its size grows exponentially as a function of the radius of the explored subgraph. We then show that if the queues in two explorations started from different vertices become large enough, they must link up with high probability. We prove that all of these events occur with high enough probability that a union bound over pairs of starting locations for the exploration allows us to conclude.

To implement the above argument, we use the technique of {\em switching}, independently introduced in \cite{MR42678,Havel1955} (see also \cite{MR681916,MR384612} for some other early works using the technique), which has found substantial use in the combinatorial analysis of graphs with given degrees, as well as as in developing Markov chains for sampling from random graphs with given degrees (see the ICM article \cite{MR2827981} for a fairly recent overview of the uses of the method). In its simplest incarnation, the technique works as follows. Let $G$ be a graph and let $u,v,x,y \in V(G)$ be such that $\{u,v\},\{x,y\} \in E(G)$ and $\{u,x\},\{v,y\}\not\in E(G)$. Then {\em switching on $uv$ and $xy$} consists of removing the edges $\{u,v\},\{x,y\}$ and adding the edges $\{u,x\},\{v,y\}$. Note that in the resulting graph $G'$, all vertices have the same degrees as they do in $G$. 

As a very simple example of how switching can be used to prove probability bounds, consider the all-$3$'s degree sequence $\dseq=(3,\ldots,3)$ of length $n$. Let $\cA=\{G \in \cG_{\dseq}: G|_{[4]}\mbox{ is a clique}\}$, and let $\cB\subset \cG_{\dseq}\setminus \cA$ be the set of graphs $G'$ that can be obtained from $A$ by a switching that uses the edge $12$. For each graph in $G\in \cA$, there are exactly $3n/2-6$ such switchings. On the other hand, for each graph $G' \in \cB$, there is exactly one switching which results in a graph in $\cA$; so $|\cB| \ge (3n/2-6)|\cA|$. It follows that if $\bG \in_u \cG_{\dseq}$ then 
\begin{align*}
\p{\bG|_{[4]}\mbox{ is a clique}}
 =
\p{\bG\in \cA} 
 = \frac{|\cA|}{|\cG_\dseq|} 
 \le \frac{|\cA|}{|\cA\cup\cB|} 
\le \frac{1}{3n/2-5}\, .
\end{align*}
Similarly, let $\cB'$ be the set of graphs $G''$ that can be obtained from a graph $G$ in $\cA$ by first switching on $12$ and $ab$, then on $13$ and $cd$, where $a,b,c,d \in \{5,\ldots,n\}$ and $\{a,b\},\{c,d\} \in E(G)$ are non-incident edges. Then the same logic as above shows that $\cB' \ge (3n/2-6)(3n/2-11)|\cA|$, so in fact 
$\p{\bG|_{[4]}\mbox{ is a clique}} = O(1/n^2)$. 

As the above example illustrates, switching allows us to compare the probabilities of two events by counting switchings between the corresponding sets of graphs. This is powerful because switching is a local operation, so counting switchings between sets of graphs is frequently tractable in settings where directly counting the sizes of the sets is not. 

The version of switching we use is called \emph{path switching}, and was introduced with a slightly different formalism in \cite{Joos_2017}. Path switching reroutes pairs of core paths between kernel vertices, rather than just pairs of edges. We postpone further discussion to Section~\ref{sec:ker_diam}, where the method is set up in detail.

To summarize, we have the following bounds:  the forest $\bF=F(\bG)$, which is essentially spanned by the set of vertices not lying in cycles, has height $O(\sqrt{n})$ in expectation and has sub-Gaussian tails on this order; the greatest length of a path in the core $\bC=C(\bG)$ between adjacent kernel vertices is $O((1+\sqrt{n/m})\log m)=O(\sqrt{n})$ in expectation, where $m\leq {n\choose 2}$ is the number of kernel edges, and has sub-exponential tails on this order; and the diameter of the kernel $\bK=K(\bG)$ is $O(\log m)$ except with probability $O(1/m)$. Combining these bounds with the inequality \eqref{eq:diamg_bd}, and being careful with the contributions to the expectation from small-probability events, is enough to conclude. 

At this point, it seems like a good idea to turn to the details.

\section{The simple kernel}\label{sec:simple_kernel}

Loops and multi-edges complicate the analysis of kernels because they have symmetries which must be accounted for (permuting parallel edges and reversing loops) and because there are constraints on the number of core vertices which must lie on a given kernel edge (loops must be subdivided at least twice and at most one copy of any edge can be non-subdivided). To deal with both problems, we augment the kernel with some additional data by ``pre-subdividing" some of the edges, which breaks these symmetries and accounts for these constraints.

Fix a graph $G$, let $C=C(G)$ and let $K=K(G)$. 
The \emph{simple kernel} of $G$, denoted $S(G)$, is the subdivision of $K$ constructed from $C$ by replacing $C(e)$ for each $e\in E(K)$ with a short path or cycle according to the following rules.
If $e$ is a loop, then replace the subpath of $C(e)$ between the two neighbours of $e^-=e^+$ with a single edge $e^*$. 
If $e$ has distinct endpoints and $\mult(e)=1$, then replace $C(e)$ with a single edge $e^*$. 
If $e$ has distinct endpoints, $\mult(e)\geq 2$, and $|C(e)|\geq 1$, then replace the subpath of $C(e)$ between $e^-$ and the neighbour of $e^+$ on $C(e)$ with a single edge $e^*$. 
In the only remaining case, which is that $e$ has distinct endpoints, $\mult(e)\geq 2$, and $|C(e)|=0$, we leave $C(e)$ unchanged; in this case $e^*$ is not defined. Note that if $G$ has connected components $G_1,\dots,G_k$, then $S(G)$ has connected components $S(G_1),\dots,S(G_k)$.

We call the edges $e^*$ \emph{mutable} and the remaining edges \emph{immutable}. Let $M(S)$ be the set of all mutable edges, and set $m(S)=|M(S)|$. For $e\in M(S)$, we write $C(e)$ for the path in $C$ replaced by $e$, and write $|C(e)|$ for the number of internal vertices on $C(e)$. When reconstructing $C$ from $S$, the mutable edges may be subdivided, but the immutable edges must not be. (Formally, we might take $S=(V(S),E(S),M(S))$, in order to encode which edges are mutable as part of the data of $S$; but we never need this notation in the proofs themselves.)

We call $S$ the simple kernel because it is simple and homeomorphic to $K$ with each edge subdivided at most twice. In particular, $e(K)\leq e(S)\leq 3e(K)$. Notice that for all $e\in E(K)$, there is exactly one mutable edge on $S(e)$ unless $\mult(e)\geq 2$ and $|C(e)|=0$. Because only one edge $f\in E(K)$ parallel to $e$ can have $|C(f)|=0$, this implies $e(K)/2\leq m(S)\leq e(K)$ and $e(S)/3\leq m(S)\leq e(S)$. We also have $\diam(K)\leq\diam(S)\leq 2\diam(K)+2$ because in $S$, every vertex has a neighbour in $K$, and for every pair of vertices $u,v\in V(K)$ that are adjacent in $K$, there is a path of length at most $2$ between $u$ and $v$ in $S$.

Fix a degree sequence $\dseq$, and denote by $\grd^=$ the set of graphs in $\grd$ with no cycle components and no tree components (the superscript should be read as ``minus-minus" rather than ``equals"). Let $S$ be a simple kernel with $m\geq 1$ mutable edges and with $V(S) \subset [n]$. We now describe a bijective encoding of the set $\grd^=(S)=\{G\in\grd^=:S(G)=S\}$. Let $\cseq=(c_0,\dots,c_n)$ be the child sequence given by 
\begin{equation*}
    c_v=
    \begin{cases}
        0 & v=0\\
        d_v-1 & v\in[n]\setminus V(S)\\
        d_v-\deg_S(v) & v\in V(S).
    \end{cases}
\end{equation*}
Recall that $\cseq_{A}$ denotes the restriction of $c$ to set $A$. We shall describe an injection from $\grd^=(S)$ into the set
\begin{equation*}
    \cX_\dseq^+(S) \coloneq \bigcup_{(A,B)}\cF_{\cseq_{A},V(S)}\times\{(T,P):T\in\cT_{\cseq_{B}},P\in\cP_{m,\h_T(0)}\},
\end{equation*}
where the union is over partitions $(A,B)$ of $[0,n]\coloneq\{0,\dots,n\}$ with $V(S)\subset A$ and $0\in B$. 

For $G \in \grd^=(S)$, construct $(F,T,P)\in\cX_\dseq^+(S)$ as follows. Deleting all edges of $C=C(G)$ yields a forest in which each tree contains at most one vertex in $V(S)$. The rooted forest $F$ is obtained by retaining only those resulting trees with a vertex in $V(S)$, and rooting each such tree at its unique vertex in $V(S)$. The $m$-composition $P=(P_1,\dots,P_m)$ is given by $P_i=|C(e_i)|$ for $i\in[m]$, where $e_1,\dots,e_m$ are the mutable edges of $S$ listed in lexicographic order. (Recall that $|C(e_i)|$ denotes the number of internal vertices of $C(e_i)$.) The tree $T$ is built from smaller trees $T_1,\dots,T_m$ which we describe first. For $i\in[m]$, deleting the edges in $C(e_i)$ yields $|C(e_i)|$ new tree components, each of which intersects $C(e_i)$ at a different internal vertex. Then $T_i$ is the union of the path $C(e_i)$ with all of these trees, so $e_i^-$ and $e_i^+$ are leaves in $T_i$. For $i\in[m-1]$, join the neighbour of $e_i^+$ in $T_i$ to the neighbour of $e_{i+1}^-$ in $T_{i+1}$ by an edge, then delete $e_i^+$ and $e_{i+1}^-$. In the resulting tree, $e_1^-$ and $e_m^+$ are leaves. By rooting this tree at the neighbour of $e_1^-$ and deleting $e_1^-$, then relabeling $e_m^+$ by $0$, we obtain $T$.

The preceding construction defines an injection $\grd^=(S)\to \cX_\dseq^+(S)$, and thus a bijection $\grd^=(S)\to \cX_\dseq(S)$, where $\cX_\dseq(S)\subset \cX_\dseq^+(S)$ is the image of this injection. The image $\cX_\dseq(S)$ can be explicitly described as the set of all $(F,T,P)\in\cX_\dseq^+(S)$ satisfying the following constraint for each triangle component of $S$, say with vertices $u,v,w$ and mutable edge $\{u,v\}$: writing $T_u,T_v,T_w$ for the trees in $F$ rooted at $u,v,w$ respectively, the least labeled leaf in $V(T_u)\cup V(T_v)\cup V(T_w)$ must lie in $V(T_w)$.

Due to this bijection, a random graph $\bG\in_u\grd^=(S)$, can be encoded by a random triple $(\bF,\bT,\bP)\in_u\cX_\dseq(S)$. Under this bijection, writing $\bC=C(\bG)$, the number of internal vertices on $\bC(e_i)$ is given by $\bP_i$. In particular, the total number of vertices in $\bC$ but outside the simple kernel is equal to $\sum_{i=1}^m\bP_i=\h_\bT(0)$. This implies the following proposition.

\begin{prop}\label{prop:bij}
    Let $\dseq=(d_1,\dots,d_n)$ be a degree sequence and let $S$ be a simple kernel with $m\geq 1$ mutable edges $e_1,\dots,e_m$ listed in lexicographic order. Let $(\bF,\bT,\bP)\in_u\cX_\dseq(S)$, let $\bG\in_u\grd^=(S)$, and write $\bC=C(\bG)$. Then
    \begin{align*}
        (|\bC(e_1)|,\dots,|\bC(e_m)|)&\eqdist(\bP_1,\dots,\bP_m)\text{, and hence}\\
        v(\bC)-v(S)&\eqdist\h_\bT(0).
    \end{align*}
\end{prop}

For any $B\subset[0,n]\setminus V(S)$ with $0\in B$, conditional on $V(\bT)=B$, we have $(\bT,\bP)\in_u\{(T,P)\mid T\in\cT_{\cseq_{B}}, P\in \cP_{m,\h_T(0)}\}$. It follows that conditional on $\bT=T$ for a tree $T$ with $\h_T(0)=h$, then $\bP\in_u\cP_{m,h}$. Moreover $0$ is a leaf in $\cseq_{B}$, and if $\dseq_{[n]\setminus V(S)}$ is 2-free then $\cseq_{B}$ is 1-free. Therefore if we can prove a tail bound on $\h_\bS(0)$ where $(\bS,\bQ)\in_u\{(S,Q)\mid S\in\cT_\cseq, Q\in\cP_{m,\h_S(0)}\}$ which is uniform over all $1$-free child sequences $\cseq=(c_0,\dots,c_n)$ with $c_0=0$ and which is increasing in $n$, then it will apply without change to $v(\bC)-v(S)$ provided that $d_v\neq 2$ for all $v\in[n]\setminus V(S)$. In order to understand $v(\bC)$, like in the proof outline, it therefore suffices to study the height of $0$ in a random rooted tree with a given 1-free child sequence $\cseq=(c_0,\dots,c_n)$ such that $c_0=0$, where the probability of selecting a given tree $T\in\cT_\cseq$ is proportional to the number of $m$-compositions of $\h_T(0)$.

\section{A tail bound for composition-biased trees}

\subsection{Stochastic domination}\label{sec:st_dom}
This subsection is devoted to the proof of Theorem~\ref{thm:st_dom}. 
For a random variable $\bX$ and an event $E$ with $\p{E}>0$, we use the notation $\bY\stackrel{d}{=}(\bX\mid E)$ to mean that $\p{\bY\in B}=\p{\bX\in B\mid E}$ for every measurable set $B$. Whenever we refer to a random variable of the form $(\bX\mid E)$, we implicitly assume that $\p{E}>0$. The intermediate results we need for proving Theorem~\ref{thm:st_dom} make repeated use of the following easily verified pair of stochastic inequalities.

\begin{lem}\label{lem:easy}
    For any random variable $\bX$ and any $x\leq y$, we have both $(\bX\mid\bX\geq x)\st(\bX\mid \bX\geq y)$ and $(\bX\mid\bX\leq x)\st(\bX\mid \bX\leq y)$.
\end{lem}

We start by proving the Lemma~\ref{lem:sameX}, below, using the Markov chain argument mentioned in the proof outline.

\begin{lem}\label{lem:sameX}
    Let $\bU$, $\bV$, and $\bW$ be discrete random variables such that $\bW$ is independent of $\bU$ and $\bV$. Suppose that for all $w$, $(\bU\mid\bU\geq w)\st(\bV\mid\bV\geq w)$. Then
    \begin{equation*}
        (\bU\mid\bU\geq\bW)\st(\bV\mid\bV\geq\bW).
    \end{equation*}
\end{lem}

\begin{proof}
    Choose $w_0\leq u_0\leq v_0$ in the supports of $\bW,\bU,\bV$ respectively. Consider the Markov chain $((\bU_t,\bW_t),t\geq 0)$ started from $(\bU_0,\bW_0)=(u_0,w_0)$ with transitions given by: for $t\geq 1$, and for $u\geq w$,
    \begin{align*}
        (\bU_t\mid\bW_{t-1}=w) &\eqdist (\bU\mid\bU\geq w)\text{, and}\\
        (\bW_t\mid\bU_t=u) &\eqdist (\bW\mid\bW\leq u).
    \end{align*}
    To fully specify the transition, let $\bU_t$ be conditionally independent of $\bU_{t-1}$ given that $\bW_{t-1}=w$, and let $\bW_t$ be conditionally independent of $\bW_{t-1}$ given that $\bU_t=u$. This chain is irreducible and aperiodic. A routine calculation using the independence of $\bU$ and $\bW$ shows that its stationary measure, and therefore its limit in total variation distance, is $((\bU,\bW)\mid\bU\geq\bW)$. The analogous chain $((\bV_t,\bW_t'),t\geq 0)$ started at $(\bV_0,\bW_0')=(v_0,w_0)$ converges in total variation distance to its stationary distribution $((\bV,\bW)\mid\bV\geq\bW)$. 

    We will inductively construct a coupling of the two chains such that almost surely $\bU_t\leq\bV_t$ and $\bW_t\leq\bW_t'$ for all $t\geq 0$. The base case is automatic since $\bU_0=u_0\leq v_0=\bV_0$ and $\bW_0=w_0=\bW_0'$. Suppose then that $\bW_{t-1}$ and $\bW_{t-1}'$ are coupled so that $\bW_{t-1}\leq\bW_{t-1}'$ almost surely. For any $w\leq w'$, by Lemma~\ref{lem:easy} and our assumption,
    \begin{align*}
        (\bU_t\mid\bW_{t-1}=w) \eqdist (\bU\mid\bU\geq w)
        \st (\bU\mid\bU\geq w')
        \st (\bV\mid\bV\geq w')
        \eqdist (\bV_t\mid\bW_{t-1}'=w').
    \end{align*}
    Since $\bW_{t-1}\leq\bW_{t-1}'$ almost surely, it follows that $\bU_t\st\bV_t$, so there is a coupling $(\bU_t,\bV_t)$ with $\bU_t\leq\bV_t$ almost surely. It remains to couple $\bW_t$ and $\bW_t'$. For any $u\leq v$, using Lemma~\ref{lem:easy},
    \begin{align*}
        (\bW_t\mid\bU_t=u) \eqdist (\bW\mid\bW\leq u)
        \st (\bW\mid\bW\leq v)
        \eqdist (\bW_t'\mid\bV_t=v).
    \end{align*}
    Since $\bV_t\leq\bU_t$ almost surely, it follows that there is a coupling $(\bW_t,\bW_t')$ with $\bW_t\leq \bW_t'$ almost surely. 

    Let $\widehat\bU\eqdist(\bU\mid\bU\geq\bW)$ and $\widehat\bV\eqdist(\bV\mid\bV\geq\bW)$, so $\bU_t\to\widehat\bU$ and $\bV_t\to\widehat\bV$ in total variation distance. Then the coupled sequence $((\bU_t,\bV_t),t\geq 0)$ yields a coupling $(\widehat\bU,\widehat\bV)$ of the limits, and because $\bU_t\leq\bV_t$ for all $t\geq 0$ almost surely, $\widehat\bU\leq\widehat\bV$ almost surely.
\end{proof}

With two more applications of Lemma~\ref{lem:easy}, we can promote Lemma~\ref{lem:sameX} to the following corollary, which contains Lemma~\ref{lem:sameX} as the special case where $\bX_1=\bX_2$.

\begin{cor}\label{cor:general}
    For $i\in\{1,2\}$, let $\bX_i$ and $\bY_i$ be independent discrete random variables. Suppose that $(\bX_{1}\mid \bX_{1}\leq x) \st (\bX_{2}\mid \bX_{2}\leq x)$ and $(\bY_{1}\mid\bY_{1}\geq y) \st (\bY_{2}\mid\bY_{2}\geq y)$ for all $x$ and $y$. Then
    \begin{equation*}
    (\bY_{1}\mid\bY_{1}\geq\bX_{1})\st(\bY_{2}\mid\bY_{2}\geq \bX_{2}).
    \end{equation*}
\end{cor}

\begin{proof}
We may assume that $\bX_1$, $\bY_1$, $\bX_2$, $\bY_2$ share a common probability space in which $\bX_1$ and $\bY_2$ are independent. By Lemma~\ref{lem:sameX} with $\bU=\bY_1$, $\bV=\bY_2$, and $\bW=\bX_1$,
\begin{align*}
    (\bY_1\mid\bY_1\geq\bX_1)\st(\bY_2\mid\bY_2\geq \bX_1).
\end{align*}
By Lemma~\ref{lem:sameX} with $\bU=-\bX_2$, $\bV=-\bX_1$, and $\bW=-\bY_2$, there is a coupling $(\widehat\bX_1,\widetilde\bX_2)$ of $(\bX_1\mid\bX_1\leq\bY_2)$ and $(\bX_2\mid\bX_2\leq\bY_2)$ such that $\widehat\bX_1\leq\widetilde\bX_2$ almost surely. Now for each $x\leq y$, by Lemma~\ref{lem:easy}, there is a coupling of $(\bY_2\mid\bY_2\geq x)$ and $(\bY_2\mid\bY_2\geq y)$ such that the former is at most the latter almost surely. Obtain $(\widehat\bY_2,\widetilde\bY_2)$ by first sampling $(\widehat\bX_1,\widetilde\bX_2)$, and then given that $(\widehat\bX_1,\widetilde\bX_2)=(x,y)$, sampling $(\widehat\bY_2,\widetilde\bY_2)$ according the coupling of $(\bY_2\mid\bY_2\geq x)$ and $(\bY_2\mid\bY_2\geq y)$. Then $\widehat\bY_2\leq \widetilde\bY_2$ almost surely and it is straightforward to check that $\widehat\bY_2\eqdist(\bY_2\mid\bY_2\geq\bX_1)$ and $\widetilde\bY_2\eqdist(\bY_2\mid\bY_2\geq\bX_2)$. We have thus shown
\begin{equation*}
    (\bY_1\mid\bY_1\geq\bX_1)\st(\bY_2\mid\bY_2\geq \bX_1)\st (\bY_2\mid\bY_2\geq \bX_2).\qedhere
\end{equation*}
\end{proof}

We now turn to realizing $\h_\bT(0)$ and $\h_\bS(0)$ as random variables of the form $(\bY_{1}\mid \bY_{1}\geq \bX_{1})$ and $(\bY_{2}\mid \bY_{2}\geq \bX_{2})$ for some choice of $\bX_{1},\bX_{2},\bY_{1},\bY_{2}$ satisfying the hypotheses of Corollary~\ref{cor:general}. For this we will require a bijection due to Foata and Fuchs first introduced in \cite{MR260623}, which we refer to as the (discrete) line-breaking construction and describe shortly. The present work adds to the list of probabilistic applications of the discrete line-breaking construction, also called the Foata-Fuchs bijection. For example, the proof of Theorem~\ref{thm:ad24} (\cite[Theorem 6]{addarioberry2024random}) relies heavily on it, and the article \cite{addarioberry2022foatafuchsproofcayleysformula} gives several variations and consequences of it. Given a child sequence $\cseq=(c_0,\dots,c_n)$ with $\sum_{v\in[0,n]}c_v=n$, let $\cV_\cseq$ denote the set of all sequences $V=(V_1,\dots,V_n)$ with alphabet $[0,n]$ such that $|\{i\in[n]:V_i=v\}|=c_v$ for all $v\in[0,n]$. 

As it is somewhat easier to describe, we will actually provide the inverse of the line breaking construction, which is the map $\cT_\cseq\to \cV_\cseq$ that constructs a sequence $V=(V_1,\dots,V_n)\in\cV_\cseq$ from a rooted tree $T\in\cT_\cseq$ in the following way. List the leaves of $T$ in increasing order of label as $\ell_1<\dots<\ell_k$. Define a list $(S_i,i\in[0,k])$ of growing subsets of $[0,n]$ and a list of sequences of vertices $(P_i,i\in[k])$ as follows, starting with $S_0=\{\rho(T)\}$. For $i\in[k]$, let $P_i$ be the sequence of vertices on the shortest $S_{i-1}$-to-$\ell_i$ path in $T$, excluding the final vertex $\ell_i$ but including the starting vertex in $S_{i-1}$. Let $S_i$ be the union of $S_{i-1}$ and the vertices in $P_i$. Let $V$ be the concatenation of $P_1,\dots,P_k$.

It is not hard to check that the resulting sequence $V$ lies in $\cV_\cseq$, nor is it hard to construct the inverse map $\cV_\cseq\to\cT_\cseq$ (both of which are done in \cite{addarioberry2022foatafuchsproofcayleysformula}). Thus, the line-breaking construction is a well-defined bijection which encodes trees with a given child sequence as sequences with specified multiplicities. For us, the salient property of this bijection is that if $T\in\cT_\cseq$ maps to $V=(V_1,\dots,V_n)\in\cT_\cseq$, then the first repetition in $V$ encodes the height of the leaf $\ell_1$ of minimum label. Specifically, $\h_T(\ell_1)=r(V)-1$, where $r(V)=\min(i\in[n]:V_i\in\{V_1,\dots,V_{i-1}\})$. For the rest of this section, we will work with child sequences $\cseq=(c_0,\dots,c_n)$ such that $c_0=0$, which implies that $\ell_1=0$. Thus, if $\bT\in_u\cT_\cseq$ and $\bV\in_u\cV_\cseq$ then $\h_\bT(0)\eqdist r(\bV)-1$.

It is an easy exercise to deduce from this identity in law the following formulae for the probability mass function of $\h_\bT(0)$ when $\bT$ is either a uniform or composition-biased binary tree. We use the falling factorial notation $(x)_{k}=x(x-1)\cdots(x-k+1)$, so for fixed $m$, as $h$ varies, $|\cP_{m,h}|={h+m-1\choose m-1}\propto(h+m-1)_{m-1}$.

\begin{lem}\label{lem:pmf}
    Let $m$ be a positive integer, and let $\bseq=(b_0,\dots,b_n)$ be a binary child sequence with $b_0=0$. Let $\bS\in_u\cT_\bseq$ and let $(\bT,\bP)\in_u\{(T,P):T\in\cT_\bseq,P\in\cP_{m,\h_T(0)}\}$. Then for all $h\in [n/2]$, 
    \begin{align*}
        \p{\h_\bS(0)=h} &= \frac{h}{n-h}\prod_{i=1}^{h-1}\left(1-\frac{i}{n-i}\right)\text{, and hence}\\
        \p{\h_\bT(0)=h} &\propto (h+m-1)_{m-1}\frac{h}{n-h}\prod_{i=1}^{h-1}\left(1-\frac{i}{n-i}\right).
    \end{align*}
\end{lem}

We will take $\bY_1=r(\bV)$ and $\bY_2=r(\bW)$ where $\bV\in_u\cV_\cseq$ and $\bW\in_u\cV_\bseq$ for some 1-free child sequence $\cseq$ and some binary child sequence $\bseq$ such that $c_0=b_0=0$. This way, for $i\in\{1,2\}$, $\bY_i$ encodes the height of $0$ in a random tree. Soon we will see how to choose $\bX_i$ so that conditioning on $\bY_i\geq\bX_i$ precisely accounts for the fact that the trees in Theorem~\ref{thm:st_dom} are composition-biased. Before describing $\bX_i$ though, we verify that the hypotheses of Corollary~\ref{cor:general} hold with these choices of $\bY_i$. Recall that by the parity of $n$ we mean the remainder of $n$ modulo $2$.

\begin{lem}\label{lem:pointwise}
    Fix a positive integer $n$ and let $\pi$ be the parity of $n$. Let $\cseq=(c_0,\dots,c_n)$ be a $1$-free child sequence and let $\bseq=(b_0,\dots,b_{n+\pi})$ be a binary child sequence. Take $\bV\in_u\cV_\cseq$ and $\bW\in_u\cV_\bseq$. Then for all $y$,
    \begin{equation*}
        (r(\bV)\mid r(\bV)\geq y)\st(r(\bW)\mid r(\bW)\geq y).
    \end{equation*}
\end{lem}

\begin{proof}
    We need to prove that for all positive integers $y$ and $h$,
    \begin{equation}\label{eq:need}
        \p{r(\bV)\geq h\mid r(\bV)\geq y}\leq \p{r(\bW)\geq h\mid r(\bW)\geq y}.
    \end{equation}
    When $h\leq y$, both probabilities are equal to 1, so we can assume $h\geq y+1$. Consider a sequence $(V_1,\dots,V_{i-1})$ with no repetitions such that $\p{(\bV_1,\dots,\bV_{i-1})=(V_1,\dots,V_{i-1})}>0$. Conditionally given that $(\bV_1,\dots,\bV_{i-1})=(V_1,\dots,V_{i-1})$, $\bV_{i}$ is uniform on the multiset containing $c_v-1$ copies of each $v\in \{V_1,\dots,V_{i-1}\}$ and $c_v$ copies of each $v\in[0,n]\setminus\{V_1,\dots,V_{i-1}\}$. Therefore the conditional probability that $\bV_{i}$ is a repetition is given by
    \begin{equation*}
        \p{r(\bV)=i\mid (\bV_1,\dots,\bV_{i-1})=(V_1,\dots,V_{i-1})}=\frac{\sum_{j=1}^{i-1}(c_{V_j}-1)}{n-i+1}.
    \end{equation*}
    Since $\cseq$ is 1-free, this is bounded below by $(i-1)/(n-i+1)$. Averaging over sequences $(V_1,\dots,V_{i-1})$ containing no repetitions, we find that 
    \begin{equation*}
        \p{r(\bV)>i\mid r(\bV)>i-1} \leq 1-\frac{i-1}{n-i+1}.
    \end{equation*}
    Therefore
    \begin{equation}\label{eq:V}
        \p{r(\bV)\geq h\mid r(\bV)\geq y}
        = \prod_{i=y}^{h-1}\p{r(\bV)> i\mid r(\bV)>i-1}
        \leq \prod_{i=y}^{h-1}\left(1-\frac{i-1}{n-i+1}\right).
    \end{equation}
    Identical reasoning applied to $\bW$ yields
    \begin{equation}\label{eq:W}
        \p{r(\bW)\geq h\mid r(\bW)\geq y} = \prod_{i=y}^{h-1}\left(1-\frac{i-1}{n+\pi-i+1}\right).
    \end{equation}
    Equality holds in (\ref{eq:W}) instead of an inequality like in (\ref{eq:V}) because $\bseq$ is binary, which implies $\sum_{j=1}^{i-1}(b_{W_j}-1)=j-1$ for any sequence $(W_1,\dots,W_{i-1})$ with no repetitions and $\p{(\bW_1,\dots,\bW_{i-1})=(W_1,\dots,W_{i-1})}>0$. Comparing the bound (\ref{eq:V}) with the formula (\ref{eq:W}), the desired inequality (\ref{eq:need}) is clear.
\end{proof}

For $i\in\{1,2\}$, we will let $\bX_i$ be (a translate of) the maximum of a random set of indices. The following lemma will be used to ensure that this choice of $\bX_i$ satisfies the assumption of Corollary~\ref{cor:general}.

\begin{lem}\label{lem:max}
    Fix $1\leq j\leq k\leq \ell$ and let $\bA\in_u{[k]\choose j}$ and $\bB\in_u{[\ell]\choose j}$. Then for all $x$,
    \begin{equation*}
        (\max\bA\mid\max\bA\leq x)\st(\max\bB\mid\max\bB\leq x).
    \end{equation*}
\end{lem}

\begin{proof}
    We may assume that $x$ is a positive integer. Let $k'=\min(k,x)$ and $\ell'=\min(\ell,x)$, and take $\bA'\in_u{[k']\choose j}$ and $\bB'\in_u{[\ell']\choose j}$. Then $\bA'\stackrel{d}{=}(\bA\mid\max \bA\leq x)$ and $\bB'\stackrel{d}{=}(\bB\mid\max \bB\leq x)$, so we just need $\max\bA'\st\max\bB'$. This holds because for $x'\leq k'$,
    \begin{align*}
        \p{\max\bA'\leq x'} = \frac{{x'\choose j}}{{k'\choose j}}\geq \frac{{x'\choose j}}{{\ell'\choose j}} = \p{\max\bB'\leq x'},
    \end{align*}
    and for $x'>k'$,
    \begin{equation*}
        \p{\max\bA'\leq x'}=1\geq \p{\max\bB'\leq x'}.\qedhere
    \end{equation*}
\end{proof}

For a child sequence $\cseq=(c_0,\dots,c_n)$ with $c_0=0$, let $\cV_\cseq^*$ be the set of sequences $V\in\cV_\cseq$ such that each multiplicity-one element $v\in[0,n]$ (meaning $c_v=1$) appears in $V$ before the first repetition. Let $\cT_\cseq^*$ be the image of $\cV_\cseq^*$ under the line breaking construction. Explicitly, $\cT_\cseq^*$ consists of the trees $T\in\cT_\cseq$ such that each one-child vertex $v\in[0,n]$ (meaning $c_v=1$) lies on the root-to-0 path in $T$. The lemma below follows from a natural correspondence between pairs $(T,P)$ where $T\in\cT_\cseq$ and $P\in\cP_{m,\h_T(0)}$ and trees $T^*\in\cT_{\cseq^+}^*$, where $\cseq^+$ is obtained from $\cseq$ by appending $m-1$ entries equal to 1. Namely, we let $T^*$ by obtained from $T$ by breaking up the root-to-0 path with one-child vertices in such a way that the pieces of this path have lengths given by the parts of $P$.

\begin{lem}\label{lem:T_to_Tplus}
    Let $\cseq=(c_0,\dots,c_n)$ be a $1$-free child sequence with $c_0=0$ and write $\cseq^+$ for the child sequence on $[0,n+m-1]$ of the form $\cseq^+=(c_0,\dots,c_n,1,\dots,1)$. Let $(\bT,\bP)\in_u\{(T,P)\mid T\in\cT_\cseq, P\in\cP_{m,\h_T(0)}\}$ and let $\bT^*\in_u\cT_{\cseq^+}^*$. Then
    \begin{equation*}
        \h_\bT(0)\stackrel{d}{=}\h_{\bT^*}(0)-m+1.
    \end{equation*}
\end{lem}

\begin{proof}
    We will exhibit an $(m-1)!$-to-1 function which maps trees $T^*\in\cT_{\cseq^+}^*$ to pairs $(T,P)$ with $T\in\cT_\cseq$ and $P\in\cP_{m,\h_T(0)}$ in such a way that $\h_T(0)=\h_{T^*}(0)-(m-1)$. Let $T$ be the homeomorphic reduction of $T^*$, and define $P=(P_1,\dots,P_m)$ as follows. List the one-child vertices in the order they appear on the root-to-0 path in $T^*$ as $(v_1,\dots,v_{m-1})$; also define $v_m=0$. Set $P_1=\h_{T^*}(v_1)$, and for $i\in[2,m]$, set $P_i=\h_{T^*}(v_i)-\h_{T^*}(v_{i-1})-1$. It is clear that $T\in \cT_\cseq$. Furthermore, $\h_{T^*}(v_i)\geq \h_{T^*}(v_{i-1})+1$ and so $P_i\geq 0$ for all $i\in[m]$. Finally, the root-to-$0$ path in $T$ is obtained by suppressing the $m-1$ one-child vertices on the root-to-0 path in $T^*$, so $\h_T(0)=\h_{T^*}(0)-(m-1)$. Hence
    \begin{align*}
        \sum_{i=1}^{m}P_i &= \h_{T^*}(v_1)+\sum_{i=2}^{m}(\h_{T^*}(v_i)-\h_{T^*}(v_{i-1})-1)\\
        &= \h_{T^*}(v_m)-(m-1)\\
        &= \h_{T^*}(0)-(m-1)\\
        &= \h_{T}(0),
    \end{align*}
    and so $P\in\cP_{m,\h_T(0)}$ and our map is well-defined. The $(m-1)!$ preimages of $(T,P)$ are the trees obtained from $T^*$ by permuting the labels of the $m-1$ one-child vertices.
\end{proof}

Given $V=(V_1,\dots,V_n)\in\cV_\cseq$ for some child sequence $\cseq=(c_0,\dots,c_n)$ containing a 1, let $f(V)=\max(i\in[n]: c_{V_i}=1)$ denote the final index of an entry in $V$ which appears exactly once. Then $r(V)\geq f(V)$ if and only if $V\in\cV_\cseq^*$. We can now prove Theorem~\ref{thm:st_dom}. 

\begin{proof}[Proof of Theorem~\ref{thm:st_dom}]
    Let $\bV\in_u\cV_\cseq$, and let $\cseq^+$ be the child sequence on $[0,n+m-1]$ of the form $\cseq^+=(c_0,\dots,c_n,1,\dots,1)$. Let $\bA\in_u{[n+m-1]\choose m-1}$ be independent of $\bV$, and generate $\bV^+\in_u\cV_{\cseq^+}$ from $\bV$ and $\bA$ as follows. Let the entries of $\bV^+$ at the indices in $\bA$ be the elements of $[n+1,n+m-1]$ ordered by a uniformly random permutation independent of $\bA$ and $\bV$. Let the entries of $\bV^+$ at the indices outside $\bA$ be those of $\bV$, with the same order as in $\bV$. Take $\bV^*\in_u\cV_{\cseq^+}^*$ and $\bT^*\in_u\cT_{\cseq^+}^*$. 

    By Lemma~\ref{lem:T_to_Tplus}, $\h_\bT(0)\stackrel{d}{=}\h_{\bT^*}(0)-m+1$, and since the height of $0$ corresponds to the index preceding the first repetition under the line-breaking construction, $\h_{\bT^*}(0)\stackrel{d}{=}r(\bV^*)-1$. Now $\bV^+\in\cV_{\cseq^+}^*$ if and only if $r(\bV^+)\geq f(\bV^+)$, so $\bV^*\stackrel{d}{=}(\bV^+\mid r(\bV^+)\geq f(\bV^+))$. This establishes
    \begin{equation*}
        \h_\bT(0) \stackrel{d}{=} \h_{\bT^*}(0)-(m-1)
        \stackrel{d}{=} r(\bV^*)-m
        \stackrel{d}{=} (r(\bV^+)\mid r(\bV^+)\geq f(\bV^+))-m.
    \end{equation*}

    Observe that $r(\bV^+)\geq f(\bV^+)$ if and only if $r(\bV^+)=r(\bV)+m-1$, because $\cseq$ is $1$-free and no repetition in $\bV^+$ occurs at an index in $\bA$. Therefore $r(\bV^+)\geq f(\bV^+)$ if and only if $r(\bV)+m-1\geq f(\bV^+)$, and on this event, $r(\bV^+)=r(\bV)+m-1$. Hence
    
    \begin{align*}
        (r(\bV^+)\mid r(\bV^+)\geq f(\bV^+)) &= (r(\bV)+m-1 \mid r(\bV)+m-1\geq f(\bV^+))\\
        &= (r(\bV) \mid r(\bV)\geq f(\bV^+)-m+1)+m-1\\
        &= (r(\bV) \mid r(\bV)\geq \max\bA-m+1)+m-1.
    \end{align*}
    
    The same reasoning applies to $\bW\in_u\cV_\bseq$ and an independent $\bB\in_u{[n+\pi+m-1]\choose m-1}$, so 
    \begin{align*}
        \h_\bT(0) &\stackrel{d}{=} (r(\bV) \mid r(\bV)\geq \max\bA-m+1)-1\text{, and}\\
        \h_\bS(0) &\stackrel{d}{=} (r(\bW) \mid r(\bW)\geq \max\bB-m+1)-1.
    \end{align*}
    
    Now $(r(\bV)\mid r(\bV)\geq y)\st(r(\bW)\mid r(\bW)\geq y)$ for all $y$ by Lemma~\ref{lem:pointwise} and it follows from Lemma~\ref{lem:max} that $(\max\bA-m+1\mid\max\bA-m+1\leq x)\st(\max\bB-m+1\mid\max\bB-m+1\leq x)$ for all $x$. Since $\bA$ and $\bV$ are independent and $\bB$ and $\bW$ are independent, Corollary~\ref{cor:general} yields
    \begin{align*}
        \h_\bT(0) &\stackrel{d}{=} (r(\bV) \mid r(\bV)\geq \max\bA-m+1)-1\\
        &\st (r(\bW) \mid r(\bW)\geq \max\bB-m+1)-1\\
        &\stackrel{d}{=} \h_\bS(0).\qedhere
    \end{align*}
    \end{proof}

\subsection{Composition-biased binary trees}\label{sec:tailbd}

In this subsection we prove Theorem~\ref{thm:tail_bound}, which establishes a sub-Gaussian tail bound on the centered and rescaled height of a fixed leaf in a composition-biased random tree. Using the stochastic domination result Theorem~\ref{thm:st_dom}, we can focus on the case of a binary child sequence. 

Before proceeding with the proof, we note that the following considerations suggest that the form of the bound in Theorem~\ref{thm:tail_bound} cannot be improved. For positive integers $m$ and $n$ with $n$ even, let $\bh_{m,n}\eqdist\h_\bT(0)$, where $(\bT,\bP)\in_u\{(T,P):T\in\cT_\bseq,P\in\cP_{m,\h_T(0)}\}$ for some binary child sequence $\bseq=(b_0,\dots,b_n)$ with $b_0=0$. By Lemma~\ref{lem:pmf}, the probability mass function of $\bh_{m,n}$ is proportional to the function $p(h)$ defined below. By analyzing the asymptotics of $p(h)$ one can show that $\bh_{m,n}/\sqrt{n}\to\bh_m$ in distribution as $n\to \infty$ for a continuous random variable $\bh_m$ with density proportional to $x^me^{-x^2/2}\I{x\geq0}$. Furthermore, $\bh_m-\sqrt{m}$ converges in distribution as $m\to\infty$ to a normal random variable with mean $0$ and variance $1/2$, suggesting that $(\bh_{m,n}-\sqrt{mn})/\sqrt{n/2}$ is approximately standard Gaussian when $1\ll m\ll n$.

For the rest of this section, fix positive integers $m$ and $n$ with $n$ even. We prove Theorem~\ref{thm:tail_bound} by analyzing $\bh_{m,n}$ and then applying Theorem~\ref{thm:st_dom}. Define $p:[n/2]\to\mathbb{R}$ by
\begin{equation*}
    p(h)=(h+m-1)_{m-1}\frac{h}{n-h}\prod_{i=1}^{h-1}\left(1-\frac{i}{n-i}\right),
\end{equation*}
so $\p{\bh_{m,n}=h}\propto p(h)$ by Lemma~\ref{lem:pmf}. Then for all $h\in[n/2]$,
\begin{equation}
    \p{\bh_{m,n}\geq h} = \frac{\sum_{k\in[h,n/2]}p(k)}{\sum_{k\in[n/2]}p(k)}.
\end{equation}
We will upper-bound the numerator and lower-bound the denominator by controlling the ratios $p(h+1)/p(h)$. The regime $m\ll h\ll n$ is most important, so the following should be read as asserting that $p(h+1)/p(h)\approx \exp(m/h-h/n)$.

\begin{lem}\label{lem:ratio}
    For all $h\in[1,n/2)$, 
    \begin{equation*}
        \exp\left(\frac{m}{h+m}-\frac{h}{n-2h}\right) \leq \frac{p(h+1)}{p(h)} \leq \exp\left(\frac{m}{h}-\frac{h-2}{n}\right).
    \end{equation*}
\end{lem}

\begin{proof}
    For the upper bound, 
    \begin{align*}
        \frac{p(h+1)}{p(h)} &= \left(1+\frac{m-1}{h+1}\right)\left(1+\frac{1}{h}\right)\left(1+\frac{1}{n-h-1}\right)\left(1-\frac{h}{n-h}\right)\\
        &\leq \exp\left(\frac{m-1}{h+1}+\frac{1}{h}+\frac{1}{n-h-1}-\frac{h}{n-h}\right)\\
        &\leq \exp\left(\frac{m-1}{h}+\frac{1}{h}+\frac{2}{n-h}-\frac{h}{n-h}\right)\\
        &= \exp\left(\frac{m}{h}-\frac{h-2}{n-h}\right)\\
        &\leq \exp\left(\frac{m}{h}-\frac{h-2}{n}\right).
    \end{align*}
    To see the second inequality, note that $n-h>2h-h\geq 1$ since $h\in[1,n/2)$. Therefore $n-h\geq 2$, which implies $1/(n-h-1)\leq 2/(n-h)$. 

    For the lower bound,\belowdisplayskip=-11pt
    \begin{align*}
        \frac{p(h)}{p(h+1)} &=  \left(1-\frac{m-1}{h+m}\right)\left(1-\frac{1}{h+1}\right)\left(\frac{n-h-1}{n-h}\right)\left(1+\frac{h}{n-2h}\right)\\
        &\leq\exp\left(-\frac{m-1}{h+m}-\frac{1}{h+1}+\frac{h}{n-2h}\right)\\
        &\leq\exp\left(-\frac{m-1}{h+m}-\frac{1}{h+m}+\frac{h}{n-2h}\right)\\
        &= \exp\left(-\frac{m}{h+m}+\frac{h}{n-2h}\right).
    \end{align*}
\end{proof}

\begin{lem}\label{lem:num}
    For $h\geq 5$ and $t\geq 1$ such that $t+h\leq n/2$,
    \begin{equation*}
        p(h+t) \leq p(h)\exp\left(\frac{tm}{h}-\frac{th}{2n}-\frac{t^2}{2n}\right).
    \end{equation*}
\end{lem}

\begin{proof}
    By the upper-bound in Lemma~\ref{lem:ratio},
    \begin{align*}
        p(h+t) &= p(h)\prod_{i=h}^{h+t-1}\frac{p(i+1)}{p(i)}\\
        &\leq p(h)\exp\left(\sum_{i=h}^{h+t-1}\left(\frac{m}{i}-\frac{i-2}{n}\right)\right)\\
        &\leq p(h)\exp\left(\frac{tm}{h}-\frac{1}{n}\sum_{i=h}^{h+t-1}(i-2)\right)\\
        &= p(h)\exp\left(\frac{tm}{h}-\frac{1}{n}\left(t\left(h-\frac{5}{2}\right)+\frac{t^2}{2}\right)\right)\\
        &\leq p(h)\exp\left(\frac{tm}{h}-\frac{th}{2n}-\frac{t^2}{2n}\right),
    \end{align*}
    where the final inequality holds because $h\geq 5$.
\end{proof}

\begin{lem}\label{lem:denom}
    For $1\leq t\leq h\leq n/8$, 
    \begin{equation*}
        p(h+t) \geq p(h)\exp\left(-\frac{4th}{n}\right).
    \end{equation*}
\end{lem}

\begin{proof}
    Note that $t+h-1<n/2$ since $t\leq h\leq n/8$. Hence, by the lower-bound in Lemma~\ref{lem:ratio},
    \begin{align*}
        p(h+t) &= p(h)\prod_{i=h}^{h+t-1}\frac{p(i+1)}{p(i)}\\
        &\geq p(h)\exp\left(\sum_{i=h}^{h+t-1}\left(\frac{m}{i+m}-\frac{i}{n-2i}\right)\right)\\
        &\geq p(h)\exp\left(-\frac{t(h+t-1)}{n-2(h+t-1)}\right)\\
        &\geq p(h)\exp\left(-\frac{2th}{n-4h}\right)\\
        &= p(h)\exp\left(-\frac{4th}{n}\right),
    \end{align*}
    where the last inequality holds because $h\leq n/8$. 
\end{proof}

\begin{proof}[Proof of Theorem~\ref{thm:tail_bound}] 
    For now suppose that $\cseq$ is binary, so $\p{\bh=h}\propto p(h)$. Note that $1.9\sqrt{mn}\geq \lceil1.8\sqrt{mn}\rceil\geq 5$ since $m\geq 1$ and $n\geq 64m$, so we can apply Lemma~\ref{lem:num} with $h=\lceil1.8\sqrt{mn}\rceil$ to obtain
    \begin{align*}
        &\sum_{1.9\sqrt{mn}+x\sqrt{n}\leq k\leq n/2}p(k)\\
        &\leq \sum_{x\sqrt{n}\leq t\leq n/2-1.9\sqrt{mn}}p(\lceil1.8\sqrt{mn}\rceil+t)\\
        &\leq \sum_{x\sqrt{n}\leq t\leq n/2-1.9\sqrt{mn}}p(\lceil1.8\sqrt{mn}\rceil)\exp\left(\frac{tm}{\lceil1.8\sqrt{mn}\rceil}-\frac{t\lceil1.8\sqrt{mn}\rceil}{2n}-\frac{t^2}{2n}\right)\\
        &\leq p(\lceil1.8\sqrt{mn}\rceil)\sum_{t\geq x\sqrt{n}}\exp\left(\frac{tm}{1.8\sqrt{mn}}-\frac{1.8t\sqrt{mn}}{2n}-\frac{t^2}{2n}\right).
    \end{align*}
    Using that for $t \ge x\sqrt{n}$,
    \[
    \exp\left(\frac{tm}{1.8\sqrt{mn}}-\frac{1.8t\sqrt{mn}}{2n}
    \right)
    \le 
    \exp\left(-\frac{t\sqrt{m}}{3\sqrt{n}}
    \right)
    \le \exp\left(-\frac{x\sqrt{m}}{3}\right)\, ,
    \]
     it follows that 
    \begin{equation}\label{eq:pk1.9bd}
    \sum_{1.9\sqrt{mn}+x\sqrt{n}\leq k\leq n/2}p(k)
    \le \exp\left(-\frac{x\sqrt{m}}{3}\right)p(\lceil1.8\sqrt{mn}\rceil)\sum_{t\geq x\sqrt{n}}\exp\left(-\frac{t^2}{2n}\right).
    \end{equation}
    Next, since $n\geq 64m$, we have $1\leq \lfloor\sqrt{n/m}\rfloor\leq \lfloor\sqrt{mn}\rfloor\leq n/8$ and hence $1\leq \lfloor\sqrt{n/m}\rfloor+\lfloor\sqrt{mn}\rfloor\leq n/2$, so by Lemma~\ref{lem:denom} with $h=\lfloor\sqrt{mn}\rfloor$ and $1\leq t\leq \lfloor \sqrt{n/m}\rfloor$,
    \[
        \sum_{k\in[n/2]}p(k) \geq \sum_{1\leq t\leq \sqrt{n/m}}p(\lfloor\sqrt{mn}\rfloor+t)
        \geq \sum_{1\leq t\leq \sqrt{n/m}}p(\lfloor\sqrt{mn}\rfloor)\exp\left(-\frac{4t\lfloor\sqrt{mn}\rfloor}{n}\right),
    \]
    which since 
    \[
    \sum_{1\leq t\leq \sqrt{n/m}}\exp\left(-\frac{4t\lfloor\sqrt{mn}\rfloor}{n}\right)
    \ge 
    \lfloor\sqrt{n/m}\rfloor \exp\left(-\frac{4\lfloor\sqrt{n/m}\rfloor\lfloor\sqrt{mn}\rfloor}{n}\right)
    \ge e^{-4-3/4}\sqrt{n/m},
    \]
    yields
    $\sum_{k\in[n/2]}p(k) \ge e^{-4-3/4}\sqrt{n/m}\,p(\lfloor\sqrt{mn}\rfloor)$. 
    Applying Lemma~\ref{lem:num} with $h=\lfloor\sqrt{mn}\rfloor$ and $t=\lceil1.8\sqrt{mn}\rceil-h$, we find that $p(\lceil1.8\sqrt{mn}\rceil)\leq p(\lfloor\sqrt{mn}\rfloor)e^{1/4}$. Combining these two bounds with \eqref{eq:pk1.9bd}, we thus have
    \begin{align*}
        \p{\bh\geq 1.9\sqrt{mn}+x\sqrt{n}} &= \frac{\sum_{1.9\sqrt{mn}+x\sqrt{n/m}\leq h\leq n/2}p(h)}{\sum_{h\in[n/2]}p(h)}\\
        &\leq \frac{\exp\left(-\frac{x\sqrt{m}}{3}\right)p(\lceil1.8\sqrt{mn}\rceil)\sum_{t\geq x\sqrt{n}}\exp\left(-\frac{t^2}{2n}\right)}{e^{-4-3/4}\sqrt{n/m}\,p(\lfloor\sqrt{mn}\rfloor)}\\
        &\leq \frac{\exp(-\frac{x\sqrt{m}}{3})e^{1/4}\sum_{t\geq x\sqrt{n}}\exp(-\frac{t^2}{2n})}{e^{-4-3/4}\sqrt{n/m}}.
    \end{align*}
    Now because $n \ge 64 m \ge 64$, we have
    \[
    \sum_{t\geq x\sqrt{n}}e^{-\frac{t^2}{2n}}\frac{1}{\sqrt{n}}
     \le \int_{x-\frac{1}{\sqrt{n}}}^{\infty}e^{-t^2/2}\d t
    \le \int_{x-\frac{1}{8}}^{\infty}e^{-t^2/2}\d t.
    \]
    Together with the preceding bound this yields that 
    \[
    \p{\bh\geq 1.9\sqrt{mn}+x\sqrt{n}} 
    \le
    \exp\left(5-\frac{x\sqrt{m}}{3}+\frac{\log m}{2}\right)\int_{x-\frac{1}{8}}^{\infty}e^{-t^2/2}\d t.
    \]
    It is easily checked that $-\frac{2.5\sqrt{m}}{3}+\frac{\log m}{2}\leq -0.8$ for all $m\geq 1$. So, fixing $x\geq 2.5$, we have
    \begin{equation*}
        \exp\left(5-\frac{x\sqrt{m}}{3}+\frac{\log m}{2}\right)\leq \exp(4.2).
    \end{equation*}
    Furthermore $x-1/8\geq e^{1/4}$, and so
    \begin{equation*}
        \int_{x-\frac{1}{8}}^{\infty}e^{-t^2/2}\d t \leq \frac{\exp\left(-\frac{1}{2}\left(x-\frac{1}{8}\right)\right)^2}{x-1/8} \leq \exp\left(-\frac{x^2}{2}+\frac{x}{8}-\frac{1}{4}\right).
    \end{equation*}
    Thus, when $x\geq2.5$,
    \begin{equation}\label{eq:bin}
        \p{\bh\geq 1.9\sqrt{mn}+x\sqrt{n}} \leq \exp\left(-\frac{x^2}{2}+\frac{x}{8}+3.95\right).
    \end{equation}
    When $0\leq x<2.5$, this bound is greater than 1, and therefore (\ref{eq:bin}) holds for all $x\geq 0$. Now (\ref{eq:bin}) is slightly stronger than the bound claimed in the theorem, so this completes the proof in the case where $\cseq$ is binary.

    If $\cseq$ is not binary, then consider a binary child sequence $\bseq=(b_0,\dots,b_{n+\pi})$ with $b_0=0$, where $\pi$ is the parity of $n$. Take $(\bS,\bQ)\in_u\{(S,Q): S\in\cT_\bseq,Q\in\cP_{m,\h_S(0)}\}$, so by Theorem~\ref{thm:st_dom}, $\bh\st\h_\bS(0)$. Note that $2\sqrt{mn}\geq 1.9\sqrt{m(n+1)}$ and $\sqrt{n}\geq\sqrt{n+1}/1.1$ as $n\geq 64m\geq 64$. Applying (\ref{eq:bin}) to $\h_\bS(0)$ yields the desired tail bound:
    \begin{align*}
        \p{\bh\geq2\sqrt{mn}+x\sqrt{n}} &\leq \p{\bh\geq1.9\sqrt{m(n+\pi)}+\frac{x}{1.1}\sqrt{n+\pi}}\\
        &\leq \p{\h_\bS(0)\geq1.9\sqrt{m(n+\pi)}+\frac{x}{1.1}\sqrt{n+\pi}}\\
        &\leq \exp\left(-\frac{(x/1.1)^2}{2}+\frac{x/1.1}{8}+3.95\right)\\
        &\leq \exp\left(-\frac{x^2}{3}+4\right).\qedhere
    \end{align*}
\end{proof}

\section{The diameter of the kernel of a core with no cycle components}\label{sec:ker_diam}

The goal of this section is to prove Theorem~\ref{thm:ker_diam}, so throughout, we fix positive integers $n\leq \ell$ and a 1-free degree sequence $\dseq$ of the form $\dseq=(d_1,\dots,d_\ell)=(d_1,\dots,d_n,2,\dots,2)$ with $d_v\geq 3$ for all $v\in[n]$. We take $\bC\in_u\cG_\dseq^-$ and aim to prove that, writing $2m=\sum_{v \in [n]}d_v$, then 
\begin{equation*}
        \p{\diam^+(K(\bC))\geq31\log m}=O(1/m).
\end{equation*}

Before turning to the proof, we note that this bound is tight for some degree sequences. Consider the case $d_1=\cdots=d_n=4$, so $m=2n$. If $\ell$ is sufficiently large compared to $n$, then $K(\bC)$ can be approximated by a sample from the 4-regular configuration model with an arbitrarily small error in total variation distance. In the 4-regular configuration model, each vertex $v\in[n]$ forms an isolated double-loop with probability $\Omega(1/m^2)$ and these events are asymptotically independent for different vertices. Using the second moment method, and more specifically the fact that for a non-negative random variable $\bX$ it holds that $\p{\bX > 0} \ge \E{\bX}^2/\E{\bX^2}$, we conclude that the probability that there exists $v\in[n]$ forming a double-loop is $\Omega(n/m^2)=\Omega(1/m)$. Thus, $\p{\diam^+(K(\bC))=\infty}=\Omega(1/m)$ for this choice of $\dseq$.

As mentioned in the proof overview, we will prove the theorem by studying a breadth-first exploration of the kernel. It is useful to augment the core with the additional data of half-edge labels at each kernel vertex. The set of \emph{half-edges} is $\bigcup_{v\in[n]}\{v1,\dots,vd_v\}$, where for $v\in[n]$ and $i\in[d_v]$ the notation $vi$ is shorthand for the ordered pair $(v,i)$. In this section we view kernels as perfect matchings of the set of half-edges. We say that a half-edge of the form $vi$ is a half-edge \emph{at} $v$, or that $vi$ is \emph{incident} to $v$. 

An \emph{augmented core} $A$ with degree sequence $\dseq$ is a set of pairs $(\{ui,vj\},(p_1,\dots,p_a))$ where $ui$ and $vj$ are distinct half-edges and $(p_1,\dots,p_a)$ is a sequence of distinct degree 2 vertices (elements of $[\ell]\setminus [n]$) such that the following hold. First, the \emph{kernel} of $A$, which we denote by $K(A)$ and consists of the pairs $\{ui,vj\}$ such that $(\{ui,vj\},(p_1,\dots,p_a))\in A$ for some sequence $(p_1,\dots,p_a)$, forms a perfect matching of the set of half-edges. We allow $a=0$ in which case $(p_1,\dots,p_a)=\varnothing$. Second, $[\ell]\setminus[n]$ is partitioned by the sequences $(p_1,\dots,p_a)$ such that $(\{ui,vj\},(p_1,\dots,p_a))\in A$ for some $\{ui,vj\}\in K(A)$. Given such a set of pairs $A$, one can construct a multigraph $C(A)$ with degree sequence $\dseq$ by introducing, for each $(\{ui,vj\},(p_1,\dots,p_a))\in A$ where $ui\prec_{\operatorname{lex}}vj$, a path from $u$ to $v$ whose sequence of internal vertices is $(p_1,\dots,p_a)$. For example, for the graph $G$ in Figure~\ref{fig:biasedtree}, one augmented core $A$ with $C(A)=C(G)$ is given by 
\begin{align*}
A& =\{(\{11,21\},(9,5)),
    (\{12,31\},(8,10,7)),
    (\{13,41\},\varnothing),\\
&\qquad  (\{22,32\},(6,11)),
    (\{23,42\},\varnothing),
    (\{33,43\},\varnothing)\}\, .
\end{align*}
Note that $C(A)$ will have no cycle components because $d_v\geq 3$ for all $v\in[n]$. The elements of $K(A)$ are called {\em kernel edges}, and for $e\in K(A)$, we write $\len_A(e)$ for the length of of the corresponding path in $C(A)$; for $(e,(p_1,\dots,p_a))\in A$, we set $\len_A(e)=a+1$. Kernel edge $e$ is {\em subdivided} in $A$ if $\len_A(e)\geq 2$ and {\em non-subdivided} otherwise. 

For $A$ to be an augmented core, we finally require that $C(A)$ be a simple graph. In other words, we insist that $\len_A(\{ui,vj\})\geq 3$ for all $\{ui,vj\}\in K(A)$ such that $u=v$, and that $\len_A(\{ui,vj\})\geq 2$ for all but at most one of the kernel edges $\{ui,vj\}\in K(A)$ whose pair of half-edges are at a fixed pair of distinct vertices $u$ and $v$. We write $\cA_\dseq$ for the set of all augmented cores with degree sequence $\dseq$. 

For each $C\in\cG_\dseq^-$, and $v\in[n]$, the number of ways to assign the half-edge labels $v1,\dots,vd_v$ to the paths of degree 2 vertices in $C$ emanating from $v$ is $d_v!$. These assignments are made independently at each $v\in[n]$, so $|\{A\in\cA_\dseq:C(A)=C\}|=\prod_{v\in[n]}d_v!$. That this does not depend on $C$ implies the following observation, which allows us to study $K(\bA)$ where $\bA\in\cA_\dseq$ rather than $K(\bC)$ where $\bC\in_u\grd^-$.

\begin{obs}\label{obs:A=C}
    If $\bA\in_u\cA_\dseq$, then $C(\bA)\in_u\cG_\dseq^-$.
\end{obs}

For a kernel represented by a perfect matching $K$ of half-edges, we identify subgraphs with subsets $L\subset K$, with one exception: for $v\in[n]$, by $L=\{v\}$ we mean the edgeless subgraph whose only vertex is $v$. We write $V(L)$ for the \emph{vertex set} of $L$, which contains the \emph{endpoints} $u$ and $v$ of all edges $\{ui,vj\}\in L$. We write $\partial L$ for the \emph{boundary} of $L$, which consists of all half-edges $ui$ such that $u\in V(L)$ but $\{ui,vj\}\not\in L$ for any half-edge $vj$.

For $A\in\cA_\dseq$ and $v\in[n]$, the {\em breadth-first kernel exploration} $(K_t,t\geq 0)=(K_t(A,v),t\geq0)$ of $A$ started from $v$ is constructed as follows. Let $K_0=\{v\}$ and for $t\geq0$, if $\partial K_t=\varnothing$, then let $K_{t+1}=K_t$. Otherwise, choose $ui\in\partial K_t$ to be lexicographically minimal among all half-edges in $\partial K_t$ nearest to $v$. Let $e\in K(A)$ be such that $ui\in e$ and let $K_{t+1}=K_t\cup\{e\}$.

This process explores the component of $K(A)$ containing $v$ one edge at a time in breadth-first order. So, if $s$ is the number of edges in the component of $K(A)$ containing $v$ (alternatively $s=\min(t : \partial K_t=\varnothing)$), then $|K_t|=t$ for all $t\leq s$ and $|K_t|=s$ for all $t\geq s$. We write $\rad(K_t)=\max(\dist_{K_t}(v,u):u\in V(K_t))$, and for $t<s$, we denote the unique edge in $K_{t+1}\setminus K_t$ by $e_{t+1}=e_{t+1}(A,v)$. We say that $e_{t+1}$ is a \emph{back-edge} if both of its endpoints are in $K_t$, and that it is a \emph{loop} if its endpoints are equal. If $e_{t+1}$ is a back-edge, then $|\partial K_{t+1}|=|\partial K_t|-2$, but if not, then $e_{t+1}$ has an endpoint $w$ not in $K_t$, and $|\partial K_{t+1}|=|\partial K_t|+d_w-2\geq |\partial K_t|+1$. We call $\partial K_t$ the \emph{queue} at time $t$, and say that an edge $e$ is explored by time $t$ if $e\in K_t$.

\begin{figure}[htb]
    \centering
    \includegraphics[width=0.3\linewidth,page=2]{pathswitch.pdf}
    \qquad\qquad
    \includegraphics[width=0.3\linewidth,page=1]{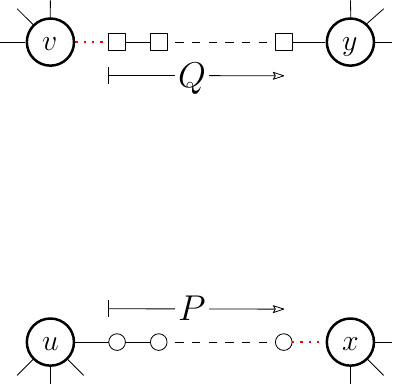}
    \caption{Left, a portion of $C(A)$ for an augmented core $A$, corresponding to pairs $(\{ui,vj\},(p_1,\ldots,p_a))$ and $(\{xk,y\ell\},(q_1,\ldots,q_b))$. If $ui\prec_{\operatorname{lex}}vj$ then $P$ is $(p_1,\ldots,p_a)$, and otherwise $P$ is $(p_a,\ldots,p_1)$, and likewise $Q$ is either $(q_1,\ldots,q_b)$ or $(q_b,\ldots,q_1)$ depending on whether or not $xk\prec_{\operatorname{lex}}y\ell$. Right, the same portion of $C(A)$ after switching on $e=(ui,vj)$ and $f=(xk,y\ell)$. 
    The curved arrows in the left figure indicate how the paths $P$ and $Q$ ``pivot'' as a result of the switching operation. The red dotted edge attached to $v$ corresponds to half-edge $vj$, which is ``detached from the head of $P$ and reattached to the tail of $Q$''; likewise, the edge corresponding to half-edge $xk$ is detached from the tail of $Q$ and reattached to the head of~$P$.}
    \label{fig:pathswitch}
\end{figure}

The principal tool we use for analyzing this exploration of a random augmented core is the notion of {\em path switching}, mentioned in the proof overview, and introduced with a different formalism in \cite{Joos_2017}. Fix $A\in\cA_\dseq$. By an oriented edge in $K(A)$, we mean an ordered pair $(ui,vj)$ such that $\{ui,vj\}\in K(A)$. We write $\vec{K}(A)$ for the set of all oriented edges in $K(A)$, so $|\vec{K}(A)|=2m$. For $e=(ui,vj)\in\vec{K}(A)$, we use the notation $e^-=u$, $e^+=v$, and $\bar{e}=(vj,ui)$. Given a pair of oriented edges $(e,f)\in \vec{K}(A)^2$, \emph{switching on $(e,f)$ in $A$} is the following operation, depicted in Figure~\ref{fig:pathswitch}. Write $e=(ui,vj)$ and $f=(xk,y\ell)$, and let $(p_1,\dots,p_a)$ and $(q_1,\dots,q_b)$ be such that $(\{ui,vj\},(p_1,\dots,p_a)),(\{xk,y\ell\},(q_1,\dots,q_b))\in A$. Delete $(\{ui,vj\},(p_1,\dots,p_a))$ and $(\{xk,y\ell\},(q_1,\dots,q_b))$ from $A$, then introduce either $(\{ui,xk\},(p_1,\dots,p_a))$ or $(\{ui,xk\},(p_a,\dots,p_1))$ and either $(\{vj,y\ell\},(q_1,\dots,q_b))$ or $(\{vj,y\ell\},(q_b,\dots,q_1))$ to form $A'$ according to the following rule: the sequences of internal vertices on the paths in $C(A)$ from $u$ to $v$ and from $x$ to $y$ must become the sequences of internal vertices on the paths in $C(A')$ from $u$ to $x$ and from $v$ to $y$ respectively. Then $C(A')$ may be a multigraph, but we say that $(e,f)$ is a \emph{valid pair} or that $(e,f)$ is a \emph{switching} in $A$ if $C(A')$ is simple (in other words, if $A'\in\cA_\dseq$). 

An important observation is that switching is reversible. By this we mean that writing $e'=(ui,xk)$ and $f'=(vj,y\ell)$, then $(e',f')\in\vec{K}(A')^2$ is a valid pair in $A'$ and switching on $(e',f')$ in $A'$ recovers $A$. We call $(e',f')$ the \emph{reversal} of $(e,f)$. Note that $(\bar{f},\bar{e})$ is a valid pair in $A$ and switching on $(\bar{f},\bar{e})$ results in the same augmented core $A'$ as switching on $(e,f)$. We therefore consider $(e,f)$ and $(\bar{f},\bar{e})$ \emph{equivalent}, so there are only four \emph{distinct} switchings involving the same unordered pair of unoriented edges as $(e,f)$; all are equivalent to one of $(e,f)$, $(\bar{e},f)$, $(e,\bar{f})$, $(\bar{e},\bar{f})$. Given $\cA\subset\cA_\dseq$, we say that a valid pair $(e,f)\in\vec{K}(A)^2$ is a switching \emph{from $A$ to $\cA$} if switching on $(e,f)$ in $A$ yields an agumented core in $\cA$. Similarly, a switching \emph{from $\cA$ to $A$} is a valid pair $(e,f)\in\vec{K}(A')$ for some $A'\in\cA$ such that switching on $(e,f)$ in $A'$ yields $A$.

\begin{lem}\label{lem:switch}
    Fix a proposition $\varphi(e,f,A)$ and say that $(e,f)$ is a good switching in $A$ if $\varphi(e,f,A)$ holds and $(e,f)$ is a valid pair in $A$. Also fix $\cA,\cB\subset\cC\subset\cA_\dseq$ and $a,b\geq0$. Suppose that for each $A\in\cA$, there are at least $a$ good switchings from $A$ to $\cB$, and for each $B\in\cB$, there are at most $b$ good switchings from $\cA$ to $B$. Then with $\bA\in_u\cA_\dseq$, 
    \begin{equation*}
        a\p{\bA\in\cA\mid\bA\in\cC}\leq b\p{\bA\in\cB\mid\bA\in\cC}.
    \end{equation*}
\end{lem}

\begin{proof}
    Consider the bipartite multigraph $\cH$ in which one part is a copy of $\cA$ and the other part is a copy of $\cB$, and in which there is an edge between $A\in\cA$ and $B\in\cB$ for each good switching $(e,f)$ from $A$ to $B$. Then each $A\in\cA$ has degree at least $a$ in $\cH$ and each $B\in\cB$ has degree at most $b$ in $\cH$. Therefore $\cH$ has at least $a|\cA|$ and at most $b|\cB|$ edges, so $a|\cA|/|\cC|\leq b|\cB|/|\cC|$.
\end{proof}

Often $\varphi(e,f,A)$ will be the trivial proposition in which case any valid pair $(e,f)$ in $A$ is good. Applying Lemma~\ref{lem:switch} then amounts to lower-bounding, for each $A\in\cA$, the number of switchings from $A$ to $\cB$, and upper-bounding the number of switchings to each $B\in\cB$ from $\cA$. Because switching is reversible, the number of good switchings to $B$ from $\cA$ is equal to the number of switchings $(e,f)\in\vec{K}(B)^2$ which are the reversal of a good switching to $B$ from some $\cA$.

To apply Lemma~\ref{lem:switch}, it will be useful to have sufficient conditions for a given pair $(e,f)$ of oriented edges in $K(A)$ to be valid in $A$. First note that $(e,f)$ is invalid in $A$ if and only if at least one of the following holds:
\begin{itemize}
    \item $e^-=f^-$ and $\len_A(e)\leq 2$, or
    \item $e^+=f^+$ and $\len_A(f)\leq 2$, or
    \item $\len_A(e)=1$ and $\{e^-,f^-\}\in E(C(A))$, or 
    \item $\len_A(f)=1$ and $\{e^+,f^+\}\in E(C(A))$.
\end{itemize}
In particular, if $e$ and $f$ are vertex-disjoint and neither endpoint of $e$ is adjacent to an endpoint of $f$ in $K(A)$, then $(e,f)$ is valid in $A$. More subtly, we have the following observation, which follows from the above criteria.

\begin{obs}\label{obs:suff_for_two}
    Fix $A\in\cA_\dseq$ and vertex-disjoint $e,f\in \vec{K}(A)$. If both $e$ and $f$ are either subdivided or have an endpoint not adjacent to either endpoint of the other, then at least two of the pairs $(e,f)$, $(e,\overline{f})$, $(\overline{e},f)$, $(\overline{e},\overline{f})$ are valid in $A$.
\end{obs}

Recall that $\dseq=(d_1,\dots,d_n,2,\dots,2)$ is a 1-free degree sequence such that $d_v\geq 3$ for all $v\in[n]$, and that $2m=\sum_{v\in[n]}d_v$. By adjusting the implicit constant in Theorem~\ref{thm:ker_diam}, we can and will assume that $m$ is sufficiently large to make various inequalities hold which are scattered throughout this section. Fix, for the remainder of this section, some $v\in[n]$ and $A^*\in\cA_\dseq$, and take $\bA\in_u\cA_\dseq$. Let $K=K(A^*)$ and let $\bK=K(\bA)$. Let $(\bK_t,t\geq0) = (K_t(\bA,v),t\geq0)$ and $(K_t,t\geq0) = (K_t(A^*,v),t\geq0)$. For $t\geq1$, let $\be_t=e_t(\bA,v)$ and let $e_t=e_t(A^*,v)$ if they exist. 

Our proof of Theorem~\ref{thm:ker_diam} is rather technical, so we take a moment to give a heuristic argument before getting into the details. In the process described above, at each time $t\geq0$ we explore an edge $\be_{t+1}$. If $\be_{t+1}$ is a back-edge, then $|\partial \bK_{t+1}|=|\partial \bK_t|-2$, and if not, then $|\partial \bK_{t+1}|\geq|\partial \bK_t|+1$. Thus, assuming that the half-edge of $\be_{t+1}$ not in $\partial \bK_t$ is roughly uniform over all unexplored half-edges, then at the early stages of the exploration when most half-edges are unexplored and lie outside the queue, one should expect
\begin{align*}
    &\E{|\partial \bK_{t+1}|-|\partial \bK_t|\mid \bK_t}\\
    &\geq \p{\text{$\be_{t+1}$ is not a back-edge}\mid \bK_t}-2\p{\text{$\be_{t+1}$ is a back-edge}\mid \bK_t}\\
    &= 1-3\p{\text{$\be_{t+1}$ is a back-edge}\mid \bK_t}\\
    &\approx 1-\frac{3|\partial \bK_t|}{2m}.
\end{align*}

Until a significant fraction of the $2m$ half-edges in $\bK$ are either explored or in $\partial \bK_t$, we should therefore expect an average net gain of at least some $\eps>0$ from each step of the exploration. Then letting $s$ be the number of time steps until all half-edges in $\partial \bK_t$ are explored, we should expect $|\partial \bK_{t+s}|\geq (1+\eps)|\partial \bK_t|$. Because our exploration is \emph{breadth-first}, for all $s\geq 0$ such that $\be_{t+s}$ has a half-edge in $\partial \bK_t$, $\rad(\bK_{t+s})\leq\rad(\bK_t)+1$. We thus expect that with high probability the queue size grows exponentially as a function of the radius until the queue has order $m$ half-edges. Equivalently, after only exploring to distance $O(\log m)$ from $v$, we should have explored a significant fraction $\bK$. Assuming again that the match of a given half-edge is roughly uniformly distributed over the other half-edges, if two such explorations (started from different vertices but run on the same kernel) have order $m$ half-edges in their queues, then with high probability some half-edge in the queue of one exploration should be matched with a half-edge in the queue of the other, joining the explorations. If all starting vertices have linear queue size after exploring to radius $O(\log m)$, and any two linear-size queues meet, then $\diam^+(\bK)=O(\log m)$. 

The only part of this argument which resists formalization is the assumption that the match of a given half-edge is approximately uniformly distributed over the other half-edges. This assumption is essentially equivalent to the distribution of $\bK$ being well approximated by a sample from the configuration model, which is false when $\ell$ is not large enough compared to $n$ and $\dseq$ contains sufficiently large degrees. We can however prove the following two lemmas using the switching method, both of which are weak forms of the assertion that $\p{\text{$\be_{t+1}$ back-edge}\mid \bK_t}\approx |\partial \bK_t|/2m$. For $L\subset K$ and $r\geq 0$, we use the notation $B_K(L,r)$ for the set of edges in $K$ both of whose endpoints have distance at most $r$ from some vertex in $V(L)$. 

\begin{lem}\label{lem:back_and_small_B2}
    For all $t\geq0$,
    \begin{equation*}
        \p{\text{$\be_{t+1}$ is a back-edge}, |B_\bK(K_t,2)|\leq m/2\mid\bK_t=K_t} \leq \frac{|\partial K_t|}{m}.
    \end{equation*}
\end{lem}

\begin{lem}\label{lem:loop}
For all $t\geq 0$ such that $t+|\partial K_t|+{|\partial K_t|\choose 2}\leq m/2$,
    \begin{equation*}
        \p{\text{$\be_{t+1}$ is a loop}\mid \bK_t=K_t}\leq \frac{2|\partial K_t|}{m}.
    \end{equation*}
\end{lem}

Unfortunately, not all back-edges are loops, and the event $\{|B_\bK(\bK_t,2)|\leq m/2\}$ is not measurable with respect to the $\sigma$-algebra generated by the process up to time $t+1$. This means that Lemma~\ref{lem:back_and_small_B2} is insufficient as an input to Azuma's inequality (or related inequalities). We avoid this problem by using an analogue of Lemma~\ref{lem:back_and_small_B2} (namely Lemma~\ref{lem:batch_of_b}) where we explore many edges at a time. This way we can obtain failure probabilities that are small enough to be integrated with a union bound. Lemma~\ref{lem:back_and_small_B2} and Lemma~\ref{lem:loop} will still be used for analyzing the first few steps of the exploration.

\begin{proof}[Proof of Lemma~\ref{lem:back_and_small_B2}]
    Let
    \begin{align*}
        \cC &= \{A\in\cA_\dseq:K_t(A,v)=K_t\},\\
        \cA &= \{A\in\cC:\text{$e_{t+1}(A,v)$ is a back-edge}, |B_{K(A)}(K_t,2)|\leq m/2\}\text{, and}\\
        \cB &= \{A\in\cC:\text{$e_{t+1}(A,v)$ is not a back-edge}\}.
    \end{align*}
    For $A\in\cA$ and $B\in\cB$, we will lower-bound the number of switchings from $A$ to an element of $\cB$ and upper-bound the number of switchings from $B$ to an element of $\cA$.

    Fix $A\in\cA$ and choose an orientation $e$ of $e_{t+1}(A,v)$. There are 2 choices for $e$. Choose $f\in\vec{K}(A)\setminus \vec{B}_{K(A)}(K_t,2)$. Since $A\in\cA$, we have $|K(A)\setminus B_{K(A)}(K_t,2)|\geq m/2$, so there are at least $2(m/2)=m$ choices for $f$. Since $A\in\cA$, $e_{t+1}(A,v)$ is a back-edge, and hence $e\in\vec{B}_{K(A)}(K_t,0)$. It follows that $(e,f)$ is a valid pair. Writing $B$ for the augmented core obtained from $A$ by switching on $(e,f)$, the edge $e_{t+1}(B,v)$ shares one half-edge with $e_{t+1}(A,v)$ which is in $\partial K_t$, and the other half-edge of $e_{t+1}(B,v)$ is incident to an endpoint of $f$. Since both endpoints of $f$ are outside $K_t$, it follows that $B\in\cB$. This proves that there are at least $2m$ switchings from $A$ to $\cB$.

    Fix $B\in\cB$ and observe that any switching from $B$ to an element of $\cA$ is equivalent to a pair $(e,f)$ where $e$ is an orientation of $e_{t+1}(B,v)$ and $f$ is an orientation of an edge in $K(B)$ with a half-edge in $\partial K_t$. Moreover if $e^-\in V(K_t)$ then $f^-\in V(K_t)$ and if $e^+\in V(K_t)$ then $f^+\in V(K_t)$. There are at most $2|\partial K_t|$ such pairs. Hence there are at most $2|\partial K_t|$ switchings from $B$ to $\cA$, and so there are at most $2|\partial K_t|$ switchings from $\cA$ to $B$.

    Applying Lemma~\ref{lem:switch} finishes the proof:
    \begin{align*}
        &\p{\text{$\be_{t+1}$ back-edge}, |B_\bK(K_t,2)|\leq m/2\mid\bK_t=K_t}\\
        &= \p{\bA\in\cA\mid\bA\in\cC}\\
        &\leq \frac{2|\partial K_t|}{2m}\p{\bA\in\cB\mid\bA\in\cC}\\
        &\leq \frac{|\partial K_t|}{m}.\qedhere
    \end{align*}
\end{proof}

\begin{proof}[Proof of Lemma~\ref{lem:loop}]
    Let 
    \begin{align*}
        \cC&=\{A\in\cA_\dseq:K_t(A,v)=K_t\},\\
        \cA&=\{A\in\cC:\text{$e_{t+1}(A,v)$ is a loop}\}\text{, and}\\
        \cB&=\{A\in\cC:\text{$e_{t+1}(A,v)$ is not a loop}\}
    \end{align*}

    Fix $A\in\cA$ and consider the set $E$ of all edges in $K(A)$ not incident to $K_t$ that are either subdivided or have an endpoint not adjacent to $K_t$. The number of edges in $K(A)$ with an endpoint in $K_t$ is at most $|K_t|+|\partial K_t|\leq t+|\partial K_t|$. There are at most $|\partial K_t|$ vertices outside $K_t$ that are adjacent to $K_t$. Because the non-subdivided edges span a simple subgraph of $K(A)$, it follows that there are at most ${|\partial K_t|\choose 2}$ non-subdivided edges both of whose endpoints are outside $K_t$ but adjacent to $K_t$. This implies
    \begin{equation*}
        |E|\geq m-|\partial K_t|-t-{|\partial K_t|\choose 2} \geq m/2.
    \end{equation*}
    Choose an orientation $e$ of $e_{t+1}(A,v)$. There are 2 choices for $e$. Now choose $f\in\vec{E}$. There are at least $2(m/2)=m$ choices for $f$. Now $e$ is a loop since $A\in\cA$, so $e$ is subdivided. Since $f\in\vec{E}$, $f$ is either subdivided or has an endpoint not adjacent to $e^-=e^+$. Therefore by Observation~\ref{obs:suff_for_two}, at least two of $(e,f)$, $(e,\bar{f})$, $(\bar{e},f)$, $(\bar{e},\bar{f})$ are valid pairs. Moreover switching on whichever of these pairs is valid yields an element of $\cB$. It follows that there are at least $2m/2=m$ switchings from $A$ to $\cB$.

    Now fix $B\in\cB$. Any switching from $B$ to some $A\in\cA$ is equivalent to a pair $(e,f)$ such that $e$ is an orientation of $e_{t+1}(B,v)$ and $f$ is an orientation of some edge with a half-edge in $\partial K_t$ (in fact one of the endpoints of $f$ must be the endpoint of $e_{t+1}(B,v)$ which is in $K_t$). There are at most $2|\partial K_t|$ such pairs. This proves that there are at most $2|\partial K_t|$ switchings from $B$ to $\cA$.

    Applying Lemma~\ref{lem:switch} finishes the proof:
    \begin{align*}
        &\p{\text{$\be_{t+1}$ is a loop}\mid\bK_t=K_t}\\
        &=\p{\bA\in\cA\mid\bA\in\cC}\\
        &\leq \frac{2|\partial K_t|}{m}\p{\bA\in\cB\mid\bA\in\cC}\\
        &\leq \frac{2|\partial K_t|}{m}.\qedhere
    \end{align*}
\end{proof}

\subsection{Getting the queue going}

The goal of this subsection is to show that with probability $1-O(1/m^2)$, the queue reaches size 5 within the first few time steps. This will be important to have later in the argument when we analyze the queue growth from exploring a batch of half-edges. For the failure probabilities there to come out small enough, the batch size, and hence the queue size, will have to be at least 5.

\begin{lem}\label{lem:reach_5}
    \begin{equation*}
        \p{\text{$|\partial \bK_t|<5$, $|B_\bK(\bK_t,2)|\leq m/2$ for all $t\leq 5$}} \leq \frac{416}{m^2}.
    \end{equation*}
\end{lem}

Lemma~\ref{lem:reach_5} relies on Lemma~\ref{lem:back_and_small_B2}, Lemma~\ref{lem:loop} and the following additional preliminary.

\begin{lem}\label{lem:e_2}
    For all $q\geq 0$ such that $1+q+{q\choose 2}\leq m/2$, 
    \begin{equation*}
        \p{\text{$\be_2$ is a back-edge}, |\partial \bK_1|\leq q} \leq \frac{6q}{m}.
    \end{equation*}
\end{lem}

\begin{proof}
    Let $\cA$ be the set of all $A\in\cA_\dseq$ such that $|\partial K_1(A,v)|\leq q$ and for some $u\neq v$, $e_1(A,v)$ and $e_2(A,v)$ both have endpoints $u$ and $v$. If $e_2(A,v)$ is a back-edge and $|\partial K_1(A,v)|\leq q$, then either $e_2(A,v)$ is a loop or $A\in\cA$. By Lemma~\ref{lem:loop},
    \begin{equation}\label{eq:loop+A}
        \p{\text{$\be_2$ back, $|\partial\bK_1|\leq q$}} \leq \p{\text{$\be_2$ loop}\mid|\partial\bK_1|\leq q}+\p{\bA\in\cA}
        \leq \frac{2q}{m}+\p{\bA\in\cA}.
    \end{equation}
    Let $\cB=\{A\in\cA_\dseq:\text{$e_2(A,v)$ is not back a back-edge}\}$.

    Fix $A\in\cA$ and consider the set $E$ of all edges in $K(A)$ not incident to $K_1(A,v)$ that are either subdivided or have an endpoint not adjacent to $K_1(A,v)$. As in the proof of Lemma~\ref{lem:loop},
    \begin{equation*}
        |E|\geq m-|\partial K_1(A,v)|-1-{|\partial K_1(A,v)|\choose 2} \geq m-q-1-{q\choose 2} \geq m/2.
    \end{equation*}
    Since $A\in\cA$, $e_1(A,v)$ and $e_2(A,v)$ share both endpoints, so at least one of them is subdivided. Let $e$ be an orientation of whichever of $e_1(A,v)$ or $e_2(A,v)$ is subdivided, so there are at least two choices for $e$. Choose $f\in\vec{E}$, so there are at least $2(m/2)=m$ choices for $f$. Both endpoints of $e$ are in $K_1(A,v)$, so since $f\in\vec{E}$, either $f$ is subdivided or it has an endpoint not adjacent to either endpoint of $e$. By Observation~\ref{obs:suff_for_two}, at least two of $(e,f)$, $(\bar{e},f)$, $(e,\bar{f})$, $(\bar{e},\bar{f})$ are valid in $A$, and switching on whichever of these pairs are valid yields some $B\in\cB$. It follows that there are at least $2m/2=m$ switchings from $A$ to $\cB$. 

    Now fix $B\in\cB$ and suppose that switching on $(e,f)$ in $B$ yields some $A\in\cA$. Let $u\neq v$ be the common endpoint of $e_1(A,v)$ and $e_2(A,v)$ other than $v$. Letting $(e',f')$ be the reversal of $(e,f)$, up to equivalence of switchings, $e'$ is an orientation of $e_i(A,v)$ for some $i\in\{1,2\}$. If $(e')^-=v$, then $f^-=u$ and $e$ is the orientation of $e_i(B,v)$ with $e^-=v$. If $(e')^-=u$, then $e^-=u$ and $f$ is the orientation of $e_i(B,v)$ with $f^-=v$. Note that $d_u+d_v-2=|\partial K_1(A,v)|\leq q$, so $d_u\leq q-1$ since $d_v\geq 3$. Therefore all possibilities for $(e,f)$ can be generated as follows. First choose $i\in\{1,2\}$ and let $j\in\{1,2\}\setminus\{i\}$. There are two choices for $(i,j)$. Let $g$ be the orientation of $e_i(B,v)$ with $g^-=v$. There is one choice for $g$. Choose an oriented edge $h$ with $h^-=u$, where $u\neq v$ is the other endpoint of $e_j(B,v)$. There are at most $d_u\leq q-1$ choices for $h$. Finally, choose $(e,f)\in\{(g,h),(h,g)\}$ in one of two ways. This shows there are at most $2(q-1)2\leq 4q$ switchings from $B$ to $\cA$. 

    By Lemma~\ref{lem:switch}, and by (\ref{eq:loop+A})
    \begin{align*}
        \p{\text{$\be_2$ back},|\partial\bK_1|\leq q} &\leq \frac{2q}{m}+\p{\bA\in\cA}
        \leq \frac{2q}{m}+\frac{4q}{m}\p{\bA\in\cB}
        \leq \frac{6q}{m}.\qedhere
    \end{align*}
\end{proof}

\begin{proof}[Proof of Lemma~\ref{lem:reach_5}]
    If neither $\be_1$ nor $\be_2$ is a back-edge, then $|\partial \bK_2|\geq 5$. If one of $\be_1$ or $\be_2$ is a back-edge but none of $\be_3$, $\be_4$, $\be_5$ are back-edges, then $|\partial \bK_5|\geq 5$. Therefore
    \begin{align*}
        &\p{\text{$|\partial\bK_t|<5,|B_\bK(\bK_t,2)|\leq m/2$ for all $t\leq 5$}}\\
        &\leq \p{\text{$\be_1$ back, $|\partial\bK_t|\leq 4$, $|B_\bK(\bK_t,2)|\leq m/2$, $\be_{t+1}$ back for some $1\leq t\leq 4$}}\\
        &+ \p{\text{$\be_2$ back, $|\partial\bK_1|\leq 4$, $|\partial\bK_t|\leq 4$, $|B_\bK(\bK_t,2)|\leq m/2$, $\be_{t+1}$ back for some $2\leq t\leq 4$}}
    \end{align*}
    We assume that $d_v\leq 4$ since if $|\partial\bK_0|=d_v\geq 5$, then there is nothing to show. If $\be_1$ is a back-edge, then it is a loop, so by Lemma~\ref{lem:loop} and Lemma~\ref{lem:back_and_small_B2},
    \begin{align*}
        &\p{\text{$\be_1$ back, $|\partial\bK_t|\leq 4$, $|B_\bK(\bK_t,2)|\leq m/2$, $\be_{t+1}$ back for some $1\leq t\leq 4$}}\\
        &\leq \p{\text{$\be_1$ loop}}\sum_{t=1}^4\p{\text{$\be_{t+1}$ back},|B_\bK(\bK_t,2)|\leq m/2\mid \text{$\be_1$ loop},|\partial \bK_t|\leq 4}\\
        &\leq \frac{2 d_v}{m}4\left(\frac{4}{m}\right)\leq \frac{128}{m^2}.
    \end{align*}
    By Lemma~\ref{lem:e_2} and Lemma~\ref{lem:back_and_small_B2},
    \begin{align*}
        &\p{\text{$\be_2$ back, $|\partial\bK_1|\leq 4$, $|\partial\bK_t|\leq 4$, $|B_\bK(\bK_t,2)|\leq m/2$, $\be_{t+1}$ back for some $2\leq t\leq 4$}}\\
        &\leq \p{\text{$\be_2$ back},|\partial\bK_1|\leq 4}\\
        &\cdot \sum_{t=2}^4\p{\text{$\be_{t+1}$ back},|B_\bK(\bK_t,2)|\leq m/2\mid \text{$\be_2$ back},|\partial\bK_1|\leq 4,|\partial\bK_t|\leq 4}\\
        &\leq \frac{6\cdot 4}{m}3\left(\frac{4}{m}\right) = \frac{288}{m^2}.
    \end{align*}
    We conclude from our initial calculation that 
    \begin{equation*}
        \p{\text{$|\partial\bK_t|<5$,$|B_\bK(\bK_t,2)|\leq m/2$ for all $t\leq 5$}}\leq \frac{128}{m^2}+\frac{288}{m^2} = \frac{416}{m^2}.\qedhere
    \end{equation*}
\end{proof}

\subsection{Growing the queue}

In this subsection we show that with high probability, the exploration reaches queue length of order $m/\log^2m$ within logarithmic radius.
\begin{lem}\label{lem:reach_m/log^2m}
There is an absolute constant $c>0$ such that
    \begin{equation*}
        \p{\text{$\rad(\bK_t)\leq 15\log m$, $|\partial \bK_t|\geq cm/\log^2 m$ for some $t$}} = 1-O(1/m^2).
    \end{equation*}
\end{lem}
As previously mentioned, we will analyze the queue growth not one edge at a time but instead in batches of size $b$, where $b$ will depend on the current queue size. Lemma~\ref{lem:batch_of_b} below allows us bound the probability that after exploring a batch of half-edges, the queue growth is an insignificant fraction of the batch size.

\begin{lem}\label{lem:batch_of_b}
Fix $b\geq 5$ and $\varepsilon\in(0,1/4)$. Then for all $t\geq 0$ such that $|\partial K_t|\geq b$, 
    \begin{align*}
        &\p{\text{$|\partial \bK_{t+s}|<|\partial K_t|+b-4\lceil\eps b\rceil$ for all $s\in[b]$},|B_\bK(K_t,3)|\leq m/2 \mid\bK_t=K_t}\\
        &\leq \frac{\eps b}{8}\left(\frac{32b|\partial K_t|}{\eps m}\right)^{\max(2,\sqrt{\eps b}/4)}.
    \end{align*}
\end{lem}

\begin{proof}
    Let $H$ consist of the first $b$ half-edges in $\partial K_t$, so no half-edge in $(\partial K_t)\setminus H$ is closer to $v$ than a half-edge in $H$, and subject to this, $H$ consists of the $b$ lexicographically least half-edges. Let $\cC=\{A\in\cA_\dseq:K_t(A,v)=K_t\}$, and for $A\in\cC$, let $s(A)$ be maximal such that $e_{t+s(A)}(A,v)$ has a half-edge in $H$. Let $E(A)=\{e_{t+1}(A,v),\dots,e_{t+s(A)}(A,v)\}$ be the set of edges explored with a half-edge in $H$. Note that $b/2\leq s(A)\leq b$, since each $e\in E(A)$ has either 1 or 2 half-edges in $H$.
    
    Let $\cA$ consist of all $A\in\cC$ such that $|B_{K(A)}(K_t,3)|\leq m/2$, no $e\in E(A)$ has an endpoint outside $K_t$ of degree at least $3b$, and greater than $\lceil\eps b\rceil$ of the edges in $E(A)$ are back-edges. If at most $\lceil\eps b\rceil$ of the edges in $E(A)$ are back-edges, then $s(A)\geq b-\lceil\eps b\rceil$, so in the time interval $[t+1,t+s(A)]$, the queue size drops by 2 at most $\lceil\eps b\rceil$ times, and increases by at least 1 at least $b-2\lceil\eps b\rceil$ times, for a net gain of at least $b-4\lceil\eps b\rceil$. If some $e\in E(A)$ had an endpoint outside $K_t$ of degree at least $3b$, then even if all $b-1$ of the other edges in $E(A)$ were back-edges, the net change in queue size after exploring the edges in $E(A)$ would be at least $3b-2-2(b-1)=b\geq b-4\lceil\eps b\rceil$. Thus $|\partial K_{t+s(A)}(A,v)|\geq |\partial K_t|+b-4\lceil\eps b\rceil$ for all $A\in\cC\setminus\cA$.

    Let $s=\max(2,\lceil\sqrt{\eps b}/4\rceil)$, set $\cA_0=\cA$ and for $i\in[0,s-1]$, construct $\cA_{i+1}$ from $\cA_i$ as follows. For $A\in\cA_{i-1}$, call a valid pair $(e,f)$ in $A$ good if $e$ is a back-edge in $\vec{E}(A)$ that does not have an endpoint outside $V(K_t)$ of degree at least $3b$, and if $f^-,f^+\not\in V(K_t)\cup\{e^-,e^+\}$. If $i=0$ and there is a non-subdivided edge in $E(A)$, then for $(e,f)$ to be good, we also insist that $e$ is non-subdivided. Let $\cA_{i+1}$ consist of all $B\in\cA_\dseq$ which can be obtained from some $A\in\cA_i$ by a good switching. Note that $\cA_i\subset \cC$ for all $i\in[0,s]$ because good switchings do not involve edges in $K_t$. 

    Fix $A\in\cA_i$. Let $A_0\in\cA_0$ be such that $A$ can be obtained from $A_0$ by a sequence of $i$ good switchings. Let $P$ be the set of back-edge in $E(A_0)\cap E(A)$. Because $P$ has at least $\lceil\eps b\rceil+1$ back-edges in $E(A_0)$ and each good switching removes at most 1 such edge, $|P|\geq \lceil\eps b \rceil+1-i$. 
    
    When $i=0$, let $P^-$ consist of the non-subdivided edges in $P$ if any exist and otherwise set $P^-=P$, so $|P^-|\geq 1$. For $i\geq 1$, obtain $P^-$ from $P$ by removing the non-subdivided edges both of whose endpoints are incident to an edge in $K(A)\setminus K(A_0)$. Since $|K(A)\setminus K(A_0)|\leq 2i$, there are at most $4i$ vertices incident to an edge in $K(A)\setminus K(A_0)$, and therefore at most ${4i\choose 2}$ non-subdivided edges in $K(A)$ both of whose endpoints are incident to an edge in $K(A)\setminus K(A_0)$. Hence, $|P^-|\geq |P|-{4i\choose 2}\geq \lceil \eps b \rceil+1-i-{4i\choose 2}$. 
    
    When $i=1$, we in fact have $|P^-|\geq \lceil \eps b \rceil$. Otherwise, after switching away a non-subdivided back-edge in $K(A_0)$ to obtain $K(A)$, there is still a non-subdivided edge in $K(A)$ with the same endpoints. This implies that there were 2 non-subdivided edges in $K(A_0)$ with the same endpoints, which is impossible. 
    
    Let $Q=K(A_0)\cap (K(A)\setminus B_{K(A)}(K_t,3))$, so since $|B_{K(A)}(K_t,3)|\leq m/2$, $|Q|\geq m/2-2i$. Now obtain $Q^-$ by removing from $Q$ each non-subdivided edge both of whose endpoints are incident to a vertex in $V(K_t)\cup V(P^-)$. Any edge joining $V(K_t)\cup V(P^-)$ to $V(Q)$ must be in $K(A)\setminus K(A_0)$ since $K_t\cup P\subset B_{K(A)}(K_t,1)$ and $Q\subset K(A)\setminus B_{K(A)}(K_t,3)$. It follows that $|Q^-|\geq |Q|-{4i\choose 2}\geq m/2-2i-{4i\choose 2}$. If $e\in\vec{P}^-$ and $f\in\vec{Q}^-$, then $f$ is either subdivided or has an endpoint not adjacent to either endpoint of $e$, and $e$ is either subdivided or has an endpoint not incident to any edge in $K(A_0)\setminus K(A)$. Since $e,f\in\vec{K}(A_0)$, $e\in\vec{B}_{K(A_0)}(K_t,1)$, and $f\not\in \vec{B}_{K(A_0)}(K_t,3)$, any edge joining an endpoint of $e$ to an endpoint of $f$ would have to lie in $K(A)\setminus K(A_0)$. By Observation~\ref{obs:suff_for_two}, at least 2 of $(e,f),(\bar{e},f),(e,\bar{f}),(\bar{e},\bar{f})$ are valid in $K(A)$. By construction of $P^-$ and $Q^-$, any valid pair in $\vec{P}^-\times\vec{Q}^-$ is also good. Therefore the number of good switchings from $A$ to $\cA_{i+1}$ is at least
    \begin{align*}
        \frac{|\vec{P}^-\times\vec{Q}^-|}{2} &\geq \frac{2\left(\lceil \eps b \rceil+1-i-{4i\choose 2}\I{i\geq 2}\right)^{\I{i\geq 1}}2\left(m/2-2i-{4i\choose 2}\right)}{2}\\
        &= 2\left(\eps b-8i^2\I{i\geq 2}\right)^{\I{i\geq 1}}\left(m-16i^2\right)
    \end{align*}

    Now fix $B\in\cA_{i+1}$, and let $(e,f)$ be the reversal of a good switching $(e',f')$ from some $A\in\cA_i$ to $B$. Up to equivalence of switchings, $e$ has a half-edge in $\partial K_t$, and $f$ either has a half-edge in $H$ or $f$ is incident to an endpoint of degree less than $(3-7\eps/2)b$ of some edge in $E(B)$. The number of oriented edges in $K(B)$ with a half-edge in $\partial K_t$ is at most $2|\partial K_t|$. The number of oriented edges in $K(B)$ with a half-edge in $H$ is at most $2b$. The number of oriented edges in $K(B)$ with an endpoint of degree less than $(3-7\eps/2)b$ that is incident to an edge in $E(B)$ is at most $2(3-7\eps/2)b^2$. Therefore the number of good switchings to $B$ from $\cA$ is at most 
    \begin{equation*}
        2|\partial K_t|(2b+2(3-7\eps/2)b^2) \leq 16b^2|\partial K_t|.
    \end{equation*}
    By Lemma~\ref{lem:switch},
    \begin{align*}
        \p{\bA\in\cA_i\mid\bA\in\cC} &\leq \frac{16b^2|\partial K_t|}{2\left(\eps b-8i^2\I{i\geq 2}\right)^{\I{i\geq 1}}\left(m-16i^2\right)}\p{\bA\in\cA_{i+1}\mid\bA\in\cC}
    \end{align*}
    When $i=0$, this reads
    \begin{equation*}
        \p{\bA\in\cA_0\mid\bA\in\cC} \leq \frac{8b^2|\partial K_t|}{m}\p{\bA\in\cA_1\mid\bA\in\cC}.
    \end{equation*}
    When $i=1$, since $16\leq 3m/4$, the above yields
    \begin{equation*}
        \p{\bA\in\cA_1\mid\bA\in\cC} \leq \frac{32b|\partial K_t|}{\eps m}\p{\bA\in\cA_2\mid\bA\in\cC}.
    \end{equation*}
    If $i\geq 2$, then $s=\lceil \sqrt{\eps b}/4\rceil$, so since $i\leq s-1$, we have $8i^2\leq 8s^2\leq \eps b/2$ and $16i^2\leq 16s^2\leq \eps b\leq  |\partial K_t|\leq m/2$. The above then simplifies to
    \begin{equation*}
        \p{\bA\in\cA_i\mid\bA\in\cC} \leq \frac{32b|\partial K_t|}{\eps m}\p{\bA\in\cA_{i+1}\mid\bA\in\cC}.
    \end{equation*}
    By induction, we then have 
    \begin{align*}
        &\p{\text{$|\partial \bK_{t+s}|<|\partial K_t|+(1-7\eps/2)b$ for all $s\in[b]$},|B_\bK(K_t,3)|\leq m/2\mid \bK_t=K_t}\\
        &\leq \p{\bA\in\cA_0\mid \bA\in\cA}\\
        &\leq \frac{8b^2|\partial K_t|}{m}\left(\frac{32b|\partial K_t|}{\eps m}\right)^{s-1}\p{\bA\in\cA_s\mid\bA\in\cC}\\
        &= \frac{\eps b}{8}\left(\frac{32b|\partial K_t|}{\eps m}\right)^s \\
        &\leq \frac{\eps b}{8}\left(\frac{32b|\partial K_t|}{\eps m}\right)^{\max(2,\sqrt{\eps b}/4)}.\qedhere
    \end{align*}
\end{proof}

Many of our bounds on ``bad" queue growth event probabilities only hold when intersected with the event that the ball of some small radius around the exploration contains a minority of the kernel edges. The following observation shows that there is almost no loss of generality in intersecting with such an event.

\begin{obs}\label{obs:big_ball}
    Fix integers $t_0,r_0,q\geq 0$ such that $m>2(r_0+4)(q-1)$. If $\rad(K_{t_0})\leq r_0$ and $|B_K(K_{t_0},3)|\geq m/2$, then there exists $t\leq t_0+4q-6$ such that $\rad(K_t)\leq r_0+3$ and $|\partial K_t|\geq q$.
\end{obs}

\begin{proof}
    To save a variable name, we shall assume that $r_0=\rad(K_{t_0})$. For $r\geq 0$, let $t(r)$ be the minimal $t$ such that the half-edge $ui\in e_{t+1}$ about to be explored has $\dist_K(v,u)=r$. Note that $\rad(K_{t(r)})=r$. Also notice that if $\{xj,yk\}\in K$ has $\dist(x,v)=r$ and $\dist(y,v)\geq r$, then $xj\in\partial K_{t(r)}$. This implies that $t(r+1)\leq t(r)+|\partial K_{t(r)}|$. Similarly $t(r_0)\leq t_0+|\partial K_{t_0}|-2$ since $\rad(K_{t_0})=r_0$. This also implies that for each $\{xj,yk\}\in B_K(K_t,3)$, if $r=\dist_K(x,v)\leq \dist_K(y,v)$ then $\dist_K(x,v)=r\leq r_0+3$ and $\dist_K(y,v)\geq r$, so $xi\in\partial K_{t(r)}$. In this way each $e\in B_K(K_t,3)$ yields a different element $xi\in\bigcup_{r=0}^{r_0+3}\partial K_{t(r)}$, so $|B_K(K_t,3)|\leq \sum_{r=0}^{r_0+3}|\partial K_{t(r)}|$.

    We can assume that $|\partial K_t|<q$ for all $t\leq t_0$, so $t(r_0)\leq t_0+|\partial K_{t_0}|-2\leq t_0+q-3$. Therefore if $|\partial K_{t(r_0)}|\geq q$, then since $\rad(K_{t(r_0)})=r_0$, we are done. Otherwise, 
    \begin{equation*}
        t(r_0+1)\leq t(r_0)+|\partial K_{t(r_0)}|\leq t_0+q-3+q-1 = t_0+2q-4.
    \end{equation*}
    Hence, if $|\partial K_{t(r_0+1)}|\geq q$, then we are done. Otherwise, $t(r_0+2)\leq 3q-5$. Continuing in this way, if there exists $i\in\{0,1,2,3\}$ with $|\partial K_{t(r_0+i)}|\geq q$, then we are done. Otherwise, since $|\partial K_{t(r)}|< q$ for all $r<r_0$, we obtain
    \begin{equation*}
        \frac{m}{2}\leq |B_K(K_t,3)|\leq\sum_{r=0}^{r_0+3}|\partial K_{t(r)}|\leq (r_0+4)(q-1).
    \end{equation*}
    This contradicts our assumption that $m>2(r_0+4)(q-1)$.
\end{proof}

To streamline the proof of Lemma~\ref{lem:reach_m/log^2m}, we isolate the following deterministic fact. The statement is technical but, roughly speaking, it says that if the queue is large enough compared to the batch size, then linear queue growth from each batch (which we should expect by Lemma~\ref{lem:batch_of_b}) implies exponential queue growth as a function of radius until the process crosses a queue size threshold. Once this threshold is crossed, we increase the batch size and apply Lemma~\ref{lem:from_p_to_q} again.

\begin{lem}\label{lem:from_p_to_q}
    Fix $\eps\in(0,1)$, $\delta\in(0,1]$ and integers $t_0\geq 0$ and $b,p,q\geq1$ such that $q\geq 2p$, $b\leq \delta p$, and $2q(\min(\log(q)/\log(1+\eps/2-\eps\delta)+1,t_0+(1/\eps+1)q)+4)\leq m$. Suppose that $|\partial K_{t_0}|\geq p$ and for all $t< (1/\eps+1)q+t_0$ such that $p\leq |\partial K_t|<q$, either $|B_K(K_t,3)|\geq m/2$ or $|\partial K_{t+s}|\geq |\partial K_t|+\eps b$ for some $s\in[b]$. Then there exists $t^*\leq t_0+(1/\eps+4.5)q-6$ such that $\rad(K_{t^*})\leq \rad(K_{t_0})+\log(q/p)\log(1+\eps/2-\eps\delta)+4$ and $|\partial K_{t^*}|\geq q$.  
\end{lem}

\begin{proof}
    First we show that under these assumptions, for all $t$ such that $p\leq |\partial K_t|<q$, $t+|\partial K_t|/2\leq (1/\eps+1)q+t_0$, and $|B_K(K_t,3)|<m/2$, there exists $s\geq 1$ such that $s\leq \min((|\partial K_{t+s}|-|\partial K_t|)/\eps,|\partial K_t|/2)$, $\rad(K_{t+s})\leq \rad(K_t)+1$, and either $|\partial K_{t+s}|\geq \min((1+\eps/2-\eps\delta)|\partial K_t|,q)$, or $|B_K(K_{t+s})|\geq m/2$.

    Define $(s_i,i\in[0,k])$ recursively as follows. First set $s_0=0$. For $i\geq 0$, if $t+s_i<(1/\eps+1)q+t_0$, $|\partial K_{t+s_i}|<\min((1+\eps/2-\eps\delta)|\partial K_t|,q)$, and $|B_K(K_{t+s_i})|<m/2$, then by the assumptions of the lemma, we can choose $r_{i+1}\in[b]$ such that $|\partial K_{t+s_i+r_{i+1}}|\geq |\partial K_{t+s_i}|+\eps b$ and we take $s_{i+1}=s_i+r_{i+1}$. If any of these three conditions fail, set $k=i$ and terminate. We take $s=s_k$. By assumption, these conditions hold when $i=0$, so $k\geq 1$, and hence $s\geq 1$. Furthermore, it follows by induction that $s_i\leq ib$ and $|\partial K_{t+s_i}|\geq |\partial K_t|+\eps ib$ for all $i\in[0,k]$. We claim that $s<|\partial K_t|/2$. Indeed $|\partial K_t|+\eps(k-1)b\leq |\partial K_{t+s_{k-1}}|<(1+\eps/2-\eps\delta)|\partial K_t|$, which combined with the fact that $|\partial K_t|\geq p\geq b/\delta $ yields $kb<|\partial K_t|/2$. Thus $s=s_k\leq kb<|\partial K_t|/2$ as claimed. This implies $t+s<t+|\partial K_t|/2\leq (1/\eps+1)q+t_0$, so by definition of $k$ and since $s=s_k$, either $|\partial K_{t+s}|\geq\min((1+\eps/2-\eps\delta)|\partial K_t|,q)$ or $|B_K(K_{t+s})|\geq m/2$. That $s<|\partial K_t|/2$ also implies $\rad(K_{t+s})\leq\rad(K_t)+1$. Indeed at most two half-edges are explored in each time step, and in order to increase the radius of the exploration by two, we must explore every half-edge in the current queue. It only remains to check that $s\leq (|\partial K_{t+s}|-|\partial K_t|)/\eps$, and this holds because $s=s_k\leq kb$, so $|\partial K_{t+s}|\geq |\partial K_t|+\eps kb\geq |\partial K_t|+\eps s$. Thus, $s$ has the desired properties. 

    Starting with $t_0$, recursively define $(t_i,i\in[0,i^*])$ as follows. If $p\leq |\partial K_{t_i}|<q$, $t_i+|\partial K_{t_i}|/2\leq (1/\eps+1)q+t_0$, and $|B(K_{t_i},3)|<m/2$, then choose $t_{i+1}\geq t_i+1$ such that $t_{i+1}-t_i\leq \min((|\partial K_{t_{i+1}}|-|\partial K_{t_i}|)/\eps,|\partial K_{t_i}|/2)$, $\rad(K_{t_{i+1}})\leq\rad(K_{t_i})+1$, and either $|\partial K_{t_{i+1}}|\geq \min((1+\eps/2-\eps\delta)|\partial K_{t_i}|,q)$ or $|B_K(K_{t_{i+1}},3)|\geq m/2$. Such a $t_{i+1}$ exists by the preceding argument. Otherwise, set $i^*=i$ and terminate.
    
    One of $|\partial K_{t_{i^*}}|\geq q$ or $t_{i^*}+|\partial K_{t_{i^*}}|/2>(1/\eps+1)q+t_0$ or $|B(K_{t_{i^*}},3)|\geq m/2$ must hold. Furthermore, it follows by induction that $|\partial K_{t_i}|\geq\min((1+\eps/2-\eps\delta)^ip,q)$ and $\rad(K_{t_i})\leq\rad(K_{t_0})+i$ for all $i\in[0,i^*]$. If $i^*\geq 1$ then $|\partial K_{t_{i^*-1}}|<q$, so $(1+\eps/2-\eps\delta)^{i^*-1}p<q$, and hence $i^*\leq \log(q/p)/\log(1+\eps/2-\eps\delta)+1$. If $i^*=0$ then this still holds. Therefore 
    \begin{equation*}
        \rad(K_{t_{i^*}}) \leq \rad(K_{t_0})+i^* \leq \rad(K_{t_0})+\frac{\log(q/p)}{\log(1+\eps/2-\eps\delta)}+1
    \end{equation*}
    Now observe that $t_{i^*}\leq (1/\eps+1/2)q+t_0$. Indeed this holds if $i^*=0$, and if $i^*\geq 1$, then
    \begin{align*}
        t_{i^*}-t_0 &= (t_{i^*}-t_{i^*-1}) +\sum_{i=1}^{i^*-1}(t_i-t_{i-1})\\
        &\leq |\partial K_{t_{i^*-1}}|/2+\sum_{i=1}^{i^*-1}(|\partial K_{t_i}|-|\partial K_{t_{i-1}}|)/\eps\\
        &= |\partial K_{t_{i^*-1}}|/2+(|\partial K_{t_{i^*-1}}|-|\partial K_{t_0}|)/\eps\\
        &\leq (1/\eps+1/2)q.
    \end{align*}
    Therefore $\rad(K_{t_{i^*}})\leq t_{i^*}\leq t_0+(1/\eps+1/2)q$. Moreover if $|\partial K_{t_{i^*}}|<q$, then
    \begin{equation*}
        t_{i^*}+|\partial K_{t_{i^*}}|/2<(1/\eps+1/2)q+q/2+t_0=(1/\eps+1)q+t_0.
    \end{equation*}
    This implies $|B(K_{t_{i^*}})|\geq m/2$. Since $2q(\min(\log(q)/\log(1+\eps/2-\eps\delta)+1,t_0+(1/\eps+1/2)q)+4)\leq m$ by assumption, we have $m>2(\rad(K_{t_{i^*}})+4)(q-1)$. By Observation~\ref{obs:big_ball}, there exists $t^*\leq t_{i^*}+4q-6\leq t_0+(1/\eps+4.5)q-6$ such that 
    \begin{align*}
        \rad(K_{t^*}) &\leq \rad(K_{t_{i^*}})+3 \leq \rad(K_{t_0})+\frac{\log(q/p)}{\log(1+\eps/2-\eps\delta)}+4\text{, and}\\
        |\partial K_{t^*}| &\geq q,
    \end{align*}
    as required. Otherwise $|\partial K_{t_{i^*}}|\geq q$, and with $t^*=t_{i^*}\leq t_0+(1/\eps+1/2)q\leq t_0+(1/\eps+4.5)q-6$, we have
    \begin{align*}
        \rad(K_{t^*}) &\leq \rad(K_{t_0})+\frac{\log(q/p)}{\log(1+\eps/2-\eps\delta)}+1\text{, and}\\
        |\partial K_{t^*}| &\geq q.\qedhere
    \end{align*}
\end{proof}

\begin{proof}[Proof of Lemma~\ref{lem:reach_m/log^2m}]
    Let $c=(e\cdot 5\cdot10\cdot16\cdot160)^{-1}$, and for $i\in\{1,2,3,4\}$, let $\delta_i,b_i,p_i,q_i$ be given by the following table.
    \[
    \begin{array}{clllllll}
    \toprule
    i & {\delta_i}   &  {b_i} & {p_i} & {q_i}\\
    \midrule
    1 &  {1}      & {5}   & {5}  & {35}\\
    2 &  {1/7}     & {5}   & {35}  & {5\cdot 7\cdot 18^2}\\ 
    3 &  {1/7}     & {5\cdot 18^2}   & {5\cdot 7\cdot 18^2}  & {5\lceil m^{1/3}\rceil}\\ 
    4 &  {1/7}     & {\lceil5(12\log m)^2\rceil}   & {5\lceil m^{1/3}\rceil}  & {\lceil cm/\log^2m\rceil}\\ 
    \bottomrule
    \end{array}
    \]

    Let $\cA_0$ consist of all $A\in\cA_\dseq$ such that $|\partial K_t(A,v)|\geq 5$ for some $t\leq T_0=19$. For $i\in\{1,2,3,4\}$, set $T_i=T_{i-1}+10q_i-6$ and let $\cA_i$ consist of all $A\in\cA_\dseq$ such that for all $t<T_i$ with $p_i\leq |\partial K_t(A,v)|<q_i$, either $|B_{K(A)}(K_t(A,v),3)|\geq m/2$ or $|\partial K_{t+s}(A,v)|\geq |\partial K_t(A,v)|+b_i/5$ for some $s\in[b_i]$. It suffices to show that for all $A\in\bigcap_{i=0}^4\cA_i$, there exists $t$ such that $\rad(K_t(A,v))\leq 15\log m$ and $|\partial K_t(A,v)|\geq cm/\log^2 m$, and that $\p{\bA\not\in\cA_i}=O(1/m^2)$ for all $i\in\{0,1,2,3,4\}$.  

    Fix $A\in\cA_\dseq$. We will recursively find, for each $i\in\{1,2,3,4\}$, a time $t_i\leq T_i$ such that
    \begin{align*}
        \rad(K_{t_i}(A,v)) &\leq T_1+4(i-1)+15\log(q_i)-15\log(q_1)\text{, and}\\
        |\partial K_{t_i}(A,v)| &\geq q_i.
    \end{align*}
    Let $\eps=1/5$ and note that for all $i\in\{1,2,3,4\}$, we have $q_i\geq 2p_i$, $b_i\leq\delta_i p_i$ and $2q_i(\min(\log(q_i)/\log(1+\eps/2-\eps\delta_i)+1,t_0+(1/\eps+1/2)q_i)+4)\leq m$. Furthermore for $i\neq 1$, we have $1/\log(1+\eps/2-\eps\delta_i)=1/\log(1+1/10-1/35)\leq 15$. 
    
    Since $A\in\cA_0$, there exists $t_0\leq T_0$ such that $|\partial K_{t_0}(A,v)|\geq p_1$. Now $T_1=T_0+10q_1-6\geq t_0+(1/\eps+1)q_1$, so since $A\in\cA_1$, for all $t<t_0+(1/\eps+1)q_1$ such that $p_1\leq |\partial K_t(A,v)|<q_1$, either $|B_{K(A)}(K_t(A,v),3)|\geq m/2$ or $|\partial K_{t+s}(A,v)|\geq |\partial K_t(A,v)|+\eps b_1$. By Lemma~\ref{lem:from_p_to_q}, there exists $t_1\leq t_0+(1/\eps+4.5)q_1-6\leq T_1$ such that $|\partial K_{t_1}(A,v)|\geq q_1$. We also have $\rad(K_{t_1}(A,v))\leq t_1\leq T_1$. This establishes the base case.

    Now fix $i\in\{1,2,3\}$ and suppose that we have found the desired $t_i\leq T_i$. Then $T_{i+1}=T_i+10q_{i+1}-6\geq t_i+(1/\eps+1)q_{i+1}$, so since $A\in\cA_{i+1}$, for all $t<t_i+(1/\eps+1)q_{i+1}$ such that $p_{i+1}\leq |\partial K_t(A,v)|<q_{i+1}$, either $|B_{K(A)}(K_t(A,v),3)|\geq m/2$ or $|\partial K_{t+s}(A,v)|\geq |\partial K_t(A,v)|+\eps b_{i+1}$ for some $s\in[b_{i+1}]$. By construction, $|\partial K_{t_i}(A,v)|\geq q_i=p_{i+1}$. By Lemma~\ref{lem:from_p_to_q}, there exists $t_{i+1}\leq t_i+(1/\eps+4.5)q_{i+1}-6\leq T_{i+1}$ such that $|\partial K_{t_{i+1}}|\geq q_{i+1}$ and 
    \begin{align*}
        \rad(K_{t_{i+1}}(A,v)) &\leq \rad(K_{t_i}(A,v))+\frac{\log(q_{i+1}/p_{i+1})}{\log(1+\eps/2-\eps\delta)}+4\\
        &\leq \rad(K_{t_i}(A,v))+4+15\log(q_{i+1})-15\log(q_i)\\
        &\leq T_1+4i+15\log(q_{i+1})-15\log(q_1).
    \end{align*}
    This completes the construction of $t_i\leq T_i$ for $i\in\{1,2,3,4\}$. With $t=t_3$, we have $|\partial K_t(A,v)|\geq q_3\geq cm/\log^2m$ and 
    \begin{align*}
        \rad(K_t(A,v)) &\leq T_1+12+15\log(q_3)-15\log(q_1)\\
        &\leq 363+12-15\log(35)+15\log(m)-15\log(\log^2(m))\\
        &\leq 15\log(m),
    \end{align*}
    as required.

    To bound $\p{\bA\not\in\cA_0}$, observe that for $A\in\cA_\dseq$, if there exists $t_0^-\leq 5$ such that either $|\partial K_{t_0^-}(A)|\geq 5$ or $|B_{K(A)}(K_{t_0^-}(A,v),2)|\geq m/2$, then $A\in\cA_0$. If $|\partial K_{t_0^-}(A)|\geq 5$ then taking $t_0=t_0^-\leq 19$ shows that $A\in\cA_0$. Otherwise $|B_{K(A)}(K_{t_0^-}(A,v),2)|\geq m/2$, so since $\rad(K_{t_0^-})\leq t_0^-\leq 5$ and $m>(5+3)(5-1)$, by Observation~\ref{obs:big_ball} there exists $t_0\leq 5+4\cdot 5-6=19$ such that $|\partial K_{t_0}(A,v)|\geq 5$. Hence, by Lemma~\ref{lem:reach_5},
    \begin{align*}
        \p{\bA\not\in\cA_0}\leq \p{|B_\bK(\bK_t,2)|\leq m/2,\text{$|\partial\bK_t|<5$ for all $t\in[5]$}}\leq \frac{416}{m^2}.
    \end{align*}

    For $i\in\{1,2,3,4\}$, we have $b_i-4\lceil\eps b_i\rceil=b_i-4b_i/5=b_i/5$, so
    \begin{align*}
        &\p{\bA\not\in\cA_i}\\
        &\leq \sum_{i=0}^{T_i-1}\p{\text{$|B_\bK(\bK_t,3)|\leq m/2$, $|\partial\bK_{t+s}|<|\partial\bK_t|+b_i/5$ for all $s\in[b_i]$}\mid p_i\leq |\partial\bK_t|<q_i}\\
        &\leq T_i\frac{b_i}{40}\left(\frac{160b_iq_i}{m}\right)^{\max\left(2,\frac{1}{4}\sqrt{\frac{b_i}{5}}\right)}
    \end{align*}
    For $i\in\{1,2\}$, we have $T_i,b_i,q_i=O(1)$ and so $\p{\bA\not\in\cA_i}=O(1/m^2)$. For $i=3$, note that $(\sqrt{b_3/5})/4=9/2$, $q_3=O(m^{1/3})$, and $T_3=T_2+10q_3-6=O(m^{1/3})\leq m$, so
    \begin{align*}
        \p{\bA\not\in\cA_3} \leq m\frac{5\cdot 18^2}{40}O(1/m^{2/3})^{9/2}=O(1/m^2).
    \end{align*}
    Finally, for $i=4$, we have $160b_4q_4\leq 1/e$ and $(\sqrt{b_4/5})/4\geq 3\log m$, and $T_4b_4= O(m)$, so
    \begin{equation*}
        \p{\bA\not\in\cA_4} =O(m(1/e)^{3\log m})= O(1/m^2).
    \end{equation*}
    We conclude that
    \begin{align*}
        &\p{\text{$\rad(\bK_t)\leq 15\log m$, $|\partial\bK_t|\geq cm/\log^2m$ for some $t$}}\\
        &\geq 1-\sum_{i=0}^{4}\p{\bA\not\in\cA_i}\\
        &= 1-O(1/m^2),
    \end{align*}
    as required.
\end{proof}

\subsection{Connecting the explorations}

In this subsection we prove Theorem~\ref{thm:ker_diam}. We have already shown that the breadth-first kernel exploration started from a given vertex reaches queue size of order $m/\log^2m$ within order $\log m$ radius with high probability. It remains to show that these explorations are well connected to each other. This is where we use the flexibility in Lemma~\ref{lem:switch} that allows us to only count ``good" switchings, and we start by providing the definition of good switching that we will use in most of the proofs.

\begin{dfn}
    Fix $A\in\cA_\dseq$, an integer $s\geq0$, and subgraphs $L,M\subset K(A)$. Say that a valid pair $(e,f)\in\vec{K}(A)^2$ is $(L,M,s)$-good in $A$ if $e\not\in L$ but $e$ has an endpoint in $V(L)$, $f\not\in M$ but $f$ has an endpoint in $V(M)$, and at least one of the following holds:
    \begin{enumerate}
        \item Either $e^-\in V(L)$ and $f^-\in V(M)$ or $e^+\in V(L)$ and $f^+\in V(M)$.
        \item There exists $u\in\{e^-,e^+\}\setminus V(L)$ joined to $V(L)$ by at least $2s+2$ edges in $K(B)$, and there exists $v\in\{f^-,f^+\}\setminus V(M)$ joined to $V(M)$ by at least $2s+2$ edges in $K(B)$.
    \end{enumerate}
\end{dfn}

Given $A\in\cA_\dseq$ and a half-edge $ui$, write $\len_A(ui)$ as shorthand for $\len_A(\{ui,vj\})$, where $vj$ is the unique half-edge such that $\{ui,vj\}\in K(A)$.

\begin{lem}\label{lem:num_good}
    Fix $A,B\in\cA_\dseq$ such that $B$ can be obtained from $A$ by a sequence of $s-1$ switchings. Let $L,M\subset K(A)$ and suppose that $\dist_{K(A)}(L,M)\geq 4$, $|\partial L|\geq p$, and for all $S\subset[n]$ with $|S|\leq 4s$, $|(\partial M)\setminus \{vi:v\in S, i\in[d_v],\len_A(vi)=1\}|\geq q$. Then the number of pairs $(e,f)\in\vec{K}(B)^2$ that are $(L,M,s)$-good in $B$ is at least $(p/2-8(2s^2+1))(q/2-8(2s^2+1))$.
\end{lem}

\begin{proof}
    Let $P=\{\{ui,vj\}\in K(A):ui\in\partial L\}$ and let $Q=\{\{ui,vj\}\in K(A):ui\in\partial M\}$. Note that $|K(A)\setminus K(B)|,|K(B)\setminus K(A)|\leq 2(s-1)$. Letting $P'=P\cap K(B)$, it follows that $|P'|\geq |P|-2(s-1)$. Now obtain $P''$ from $P'$ by removing all edges that are non-subdivided and have both endpoints incident to an edge in $K(B)\setminus K(A)$. The non-subdivided edges span a simple subgraph of $K(B)$, which has at most $4(s-1)$ vertices because $|K(B)\setminus K(A)|\leq 2(s-1)$, so $|P''|\geq|P'|-{4(s-1)\choose 2}$. Obtain $P^-$ from $P''$ by removing all edges $e$ with an endpoint $u\not\in V(L)$ that is incident to at most $2s+1$ edges joining $u$ to $V(L)$ in $K(B)$. Then $|P^-|\geq |P''|-4s(2s+1)$. Now $2(s-1)+{4(s-1)\choose 2}+4s(2s+1)=16s^2-12s+8\leq 8(2s^2+1)$, so combining the inequalities, $|P^-|\geq |P|-8(2s^2+1)\geq p/2-8(2s^2+1)$. Let $Q'=Q\cap K(B)$, so $|Q'|\geq |Q|-2(s-1)$, and obtain $Q''$ from $Q'$ by removing all non-subdivided edges with at least one endpoint incident to an edge in $K(B)\setminus K(A)$. Letting $S$ consist of all $v\in V(M)$ incident to a non-subdivided edge in $K(B)\setminus K(A)$, we have $|S|\leq 4s$, so $|(\partial M)\setminus \{vi:v\in S,i\in[d_v],\len_A(vi)=1\}|\geq q$. It follows that $|Q''|\geq q/2-2(s-1)$. Finally, obtain $Q^-$ by removing from $Q''$ all edges $e$ with an endpoint $u\not\in V(M)$ incident to at most $2s+1$ edges joining $u$ to $V(L)$. Then $|Q^-|\geq |Q''|-4s(2s+1)$. Since $2(s-1)+4s(2s+1)=8s^2+6s-2\leq 8(2s^2+1)$, we have $|Q^-|\geq q/2-8(2s^2+1)$.

    We claim that for each pair $(e,f)\in\vec{P}^-\times\vec{Q}^-$, at least one of $(e,f),(e,\overline{f}),(\overline{e},f),(\overline{e},\overline{f})$ are good. Since $P^-\subset P'$ and $Q^-\subset Q'$, we have $e,f\in \vec{K}(B)$. Since $P^-,Q^-\subset K(A)$ and $\dist_{K(A)}(L,M)\geq 4$, $e$ and $f$ are vertex-disjoint. Since $e\in P''$ and $f\in Q''$, both $e$ and $f$ are either subdivided or have an endpoint not incident to any edge in $K(B)\setminus K(A)$. Any edge joining an endpoint of $e$ to an endpoint of $f$ would have to lie in $K(B)\setminus K(A)$ since $\dist_{K(A)}(L,M)\geq 4$. Therefore both $e$ and $f$ are either subdivided or have an endpoint not adjacent to either endpoint of the other. By Observation~\ref{obs:suff_for_two}, at least two of $(e,f),(e,\overline{f}),(\overline{e},f),(\overline{e},\overline{f})$ are valid in $B$. Without loss of generality $e^-\in V(L)$ and $f^-\in V(M)$, so if $(e,f)$ or $(\overline{e},\overline{f})$ is valid, then it is good. Because $f\in Q^-$, it is either subdivided or $f^-$ is not adjacent to $e^-$ in $K(B)$. Hence if there is no edge between $e^+$ and $f^+$ in $K(B)$, then $(\overline{e},\overline{f})$ is valid and hence good. So suppose that there is an edge between $e^+$ and $f^+$ in $K(B)$. This edge must lie in $K(B)\setminus K(A)$ since $\dist_{K(A)}(L,M)\geq 4$. We claim that any valid $(g,h)\in\{(e,f),(e,\overline{f}),(\overline{e},f),(\overline{e},\overline{f})\}$ is good. This holds if $(g,h)\in\{(e,f),(\overline{e},\overline{f})\}$ as mentioned above. This also holds if $e^+\in V(L)$ or if $f^+\in V(M)$ for the same reason. So suppose that $e^+\not\in V(L)$ and $f^+\not\in V(M)$ and that $(g,h)\in\{(e,\overline{f}),(\overline{e},f)\}$ is valid. Since $e^+\not\in V(L)$ but $e\in\vec{P}^-$ and $e^+$ is incident to an edge in $K(B)\setminus K(A)$ (namely the edge joining $e^+$ to $f^+$), $e^+$ must be joined to $V(L)$ by at least $2s+2$ edges in $K(B)$. Similarly, $f^+$ must be joined to $V(M)$ by at least $2s+2$ edges in $K(B)$. Because $e^-\in V(L)$ and $f^-\in V(M)$, we conclude that if $(g,h)\in\{(e,\overline{f}),(\overline{e},f)\}$ is valid, then it is good. We have shown that either $(e,f)$ and $(\overline{e},\overline{f})$ are both good, or, any valid $(g,h)\in\{(e,f),(e,\overline{f}),(\overline{e},f),(\overline{e},\overline{f})\}$ is good. Since at least two of $(e,f),(e,\overline{f}),(\overline{e},f),(\overline{e},\overline{f})$ are valid, this implies that at least two must be good as claimed.

    The preceding argument shows that at least one quarter of the pairs $(e,f)\in\vec{P}^-\times\vec{Q}^-$ are good. Therefore 
    \begin{align*}
        &|\{(e,f)\in \vec{P}\times\vec{Q}:\text{$(e,f)$ is $(L,M,s)$-good in $B$}\}|\\
        &\geq |\{(e,f)\in \vec{P}^-\times\vec{Q}^-:\text{$(e,f)$ is $(L,M,s)$-good in $B$}\}|\\
        &\geq |\vec{P}^-\times\vec{Q}^-|/4\\
        &= |P^-||Q^-|\\
        &\geq (p/2-8(2s^2+1))(q/2-8(2s^2+1)).\qedhere
    \end{align*}
\end{proof}

\begin{lem}\label{lem:few_reverse}
    Let $A,B\in\cA_\dseq$ be such that $B$ can be obtained from $A$ by a sequence of $s$ switchings. If $L,M\subset K(A)$ and $\dist_{K(A)}(L,M)\geq 4$, then the number of $(L,M,s)$-good switchings from $\cA_\dseq$ to $B$ is at most $8s(s+m)$.
\end{lem}

\begin{proof}
    Since $B$ can be obtained from $A$ by a sequence of $s$ switchings, $|K(B)\setminus K(A)|\leq 2s$. We proceed by characterizing the pairs $(e,f)\in\vec{K}(B)^2$ which can be the reversal of some $(L,M,s)$-good switching $(e',f')$ to $B$. Fix such a switching $(e',f')$.

    First suppose that $(e',f')$ is $(L,M,s)$-good of type (1), so either $e^-=e'^-\in V(L)$ and $e^+=f'^-\in V(M)$ or $f^-=e'^+\in V(L)$ and $f^+=f'^+\in V(M)$. Because $\dist_{K(A)}(L,M)\geq 4$ and $|K(B)\setminus K(A)|\leq 2s$, there are at most $2s$ edges joining $V(L)$ and $V(M)$ in $K(B)$. Hence there are at most $2s$ oriented edges in $\vec{K}(B)$ from $V(L)$ to $V(M)$. Either $e$ or $f$ is one of these at most $2s$ edges. It follows that the number of pairs $(e,f)\in K(B)^2$ which are the reversal of some $(L,M,s)$-good switching is at most $2(2s)(2m)=8sm$ (first choose which of $e$ or $f$ goes from $V(L)$ to $V(M)$, then choose a specific edge from $V(L)$ to $V(M)$, and then choose any other oriented edge).

    Now suppose that $(e',f')$ is $(L,M,s)$-good of type (2) but not of type (1). Then either $e^-=e'^-\in V(L)$, $f^-=e'^+\not\in V(L)$, $e^+=f'^-\not\in V(M)$, $f^+=f'^+\in V(M)$ or $e^-=e'^-\not\in V(L)$, $f^-=e'^+\in V(L)$, $e^+=f'^-\in V(M)$, $f^+=f'^+\not\in V(M)$. Suppose at first that the former is true. Then in $K(B)$, there are at least $2s+1$ edges joining $f^-$ to $V(L)$ and at least $2s+1$ edges joining $e^+$ to $V(M)$. We claim that there are at most $2s$ such oriented edges $e\in\vec{K}(B)$ with $e^-\in V(L)$ such that $e^+$ is joined to $V(M)$ by at least $2s+1$ edges in $K(B)$. Otherwise, since $|K(B)\setminus K(A)|\leq 2s$, some such $e$ would be present in $\vec{K}(A)$. If $e\in\vec{K}(B)\cap\vec{K}(A)$ was such that $e^-\in V(L)$ and $e^+$ was joined to $V(M)$ by at least $2s+1$ edges in $K(B)$, then at least one of these edges would also be present in $K(A)$. This edge, together with $e$, would form a path of length 2 between $L$ and $M$ in $K(A)$, contradicting that $\dist_{K(A)}(L,M)\geq 4$. Thus $e$ belongs to the set of at most $2s$ oriented edges in $K(B)$ with $e^-\in V(L)$ and $e^+$ joined to $V(M)$ by at least $2s+1$ edges. Similarly $f$ belongs to the set of at most $2s$ oriented edges in $K(B)$ with $f^+\in V(M)$ and $f^-$ joined to $V(L)$ by at least $2s+1$ edges. A similar statement holds in the other case, so there are at most $2(2s)^2=8s^2$ pairs $(e,f)\in\vec{K}(B)$ which can be the reversal of some $(L,M,s)$-good switching of type (2) but not of type (1).

    In total, there are at most $8sm+8s^2=8s(s+m)$ pairs $(e,f)\in\vec{K}(B)^2$ which are the reversal of some $(L,M,s)$-good switching. It follows that there are at most $8s(s+m)$ $(L,M,s)$-good switchings to $B$.
\end{proof}

\begin{lem}\label{lem:connection}
    Fix partial matchings $L$ and $M$ and a positive integer $s$. Suppose $p,q\geq 96s^2$, $|\partial L|\geq p$, and for all $S\subset [n]$ with $|S|\leq 4s$, $|(\partial M)\setminus \{vi:v\in S,i\in[d_v]\}|\geq q$. Then 
    \begin{equation*}
        \p{\dist_\bK(L,M)\geq 4\mid L,M\subset\bK} \leq \left(\frac{131sm}{pq}\right)^s.
    \end{equation*}
\end{lem}

\begin{proof}
    Let $\cA_0=\{A\in\cA_\dseq:L,M\subset K(A),\dist_{K(A)}(L,M)\geq 4\}$, and for $i\in[s]$, let $\cA_i$ consist of all $B\in\cA_\dseq$ which can be obtained from some $A\in\cA_{i-1}$ by an $(L,M,i)$-good switching. Because $(L,M,i)$-good switchings do not involve edges of $L$ or $M$, each $B\in\cA_i$ has $L,M\subset K(B)$. By construction, for each $A\in\cA_{i-1}$, there is some $A_0\in\cA_0$ such that $A$ can be obtained from $A_0$ by a sequence of $i-1$ switchings. By assumption, for all $S\subset [n]$ with $|S|\leq 4i$, $|(\partial M)\setminus \{vi:v\in S,i\in[d_v]\}|\geq q$. Since $L,M\subset K(A_0)$ and $\dist_{K(A_0)}(L,M)\geq 4$, by Lemma~\ref{lem:num_good}, for all $A\in\cA_{i-1}$, there are at least $(p/2-8(2i^2+1))(q/2-8(2i^2+1))$ $(L,M,i)$-good switchings in $\vec{K}(A)^2$, and each of these switchings results in some $B\in\cA_i$. By construction, for each $B\in\cA_i$, there is some $A_0\in\cA_0$ such that $B$ can be obtained from $A_0$ by a sequence of $i$ switchings. Since $L,M\subset K(A_0)$ and $\dist_{K(A_0)}(L,M)\geq 4$, the number of $(L,M,i)$-good switchings from some $A\in\cA_{i-1}$ to $B$ is at most $8i(i+m)$ by Lemma~\ref{lem:few_reverse}. Hence by Lemma~\ref{lem:switch},
    \begin{equation*}
        \p{\bA\in\cA_{i-1}\mid L,M\subset\bK} \leq \frac{8i(i+m)}{(p/2-8(2i^2+1))(q/2-8(2i^2+1))}\p{\bA\in\cA_i\mid L,M\subset\bK}.
    \end{equation*}
    Increasing $i$ to $s$ only weakens this bound, so by induction,
    \begin{align*}
        &\p{\dist_\bK(L,M)\geq 4\mid L,M\subset\bK}\\
        &= \p{\bA\in\cA_0\mid L,M\subset\bK}\\
        &\leq \left(\frac{8s(s+m)}{(p/2-8(2s^2+1))(q/2-8(2s^2+1))}\right)^s\p{\bA\in\cA_s\mid L,M\subset\bK}\\
        &\leq \left(\frac{131sm}{pq}\right)^s,
    \end{align*}
    where the last inequality holds because $p,q\geq 96s^2$ and $\p{\bA\in\cA_s\mid L,M\subset\bK}\leq 1$.
\end{proof}

\begin{lem}\label{lem:two_vtcs}
    Fix vertices $u,v\in[n]$ and a positive integer $s$ such that $d_u,d_v>12s$. Let $\cA$ consist of all $A\in\cA_\dseq$ such that $|\{ui:i\in[d_u],\len_{A}(ui)\geq 2\}|\geq d_u/3$ and $|\{vi:i\in[d_v],\len_{A}(vi)\geq 2\}|\geq d_v/3$. Then 
    \begin{equation*}
        \p{\dist_\bK(u,v)\geq 3,\bA\in\cA} \leq \left(\frac{48sm}{d_ud_v}\right)^s.
    \end{equation*}
\end{lem}

\begin{proof}
    For $A\in\cA_\dseq$, call a switching $(e,f)\in\vec{K}(A)^2$ good in $A$ if $e^-=u$, $f^-=v$, $e^+\neq f^+$, and both $e$ and $f$ are subdivided in $A$. Note that any good switching in $A$ is valid. Set $\cA_0=\{A\in\cA_\dseq:\dist_{K(A)}(u,v)\geq 3\}\cap\cA$, and for $i\in[s]$, let $\cA_i$ consist of all $B\in\cA_\dseq$ which can be obtained from some $A\in\cA_{i-1}$ by a good switching. 
    
    Fix $A\in\cA_{i-1}$ and let $A_0\in\cA_0$ be such that $A$ can be obtained from $A_0$ by a sequence of $i-1$ good switchings. Then $|K(A)\setminus K(A_0)|\leq 2(i-1)$, and each edge not switched on in obtaining $A$ from $A_0$ is still present in $K(A)$ and moreover has the same length in $A$ as it does in $A_0$. Since $A_0\in\cA$, it follows that there are at least $d_u/3-2(i-1)$ oriented edges $e\in\vec{K}(A)\cap \vec{K}(A_0)$ that are subdivided and have $e^-=u$. Similarly there are at least $d_v/3-2(i-1)$ oriented edges $f\in\vec{K}(A)\cap \vec{K}(A_0)$ that are subdivided and have $f^-=v$. Each of the at least $(d_u/3-2(i-1))(d_v/3-2(i-1))$ such pairs $(e,f)$ are good because $e,f\in\vec{K}(A_0)$ and $\dist_{K(A_0)}(u,v)\geq 3$ implies $e^+\neq f^+$. This shows that for each $A\in\cA_{i-1}$ there are at least $(d_u/3-2(i-1))(d_v/3-2(i-1))$ good switchings from $A$ to $\cA_i$.

    Now fix $B\in\cA_i$ and let $A_0\in\cA_0$ be such that $B$ can be obtained from $A_0$ by a sequence of $i$ good switchings. Then $|K(B)\setminus K(A_0)|\leq 2i$, so there are most $2i$ edges joining $u$ and $v$ in $K(B)$. If $(e,f)$ is the reversal of a good switching $(e',f')$ to $B$, then $e^-=e'^-=u$ and $e^+=f'^-=v$. There are at most $2i$ oriented edges in $\vec{K}(B)$ from $u$ to $v$, so there are at most $(2i)(2m)=4im$ pairs $(e,f)\in\vec{K}(B)^2$ which can be the reversal of some good switching to $B$. This shows that for each $B\in\cA_i$, there are at most $4im$ good switchings to $B$ from $\cA_{i-1}$.

    Applying Lemma~\ref{lem:switch}, we obtain
    \begin{equation*}
        \p{\bA\in\cA_{i-1}}\leq \frac{4im}{(d_u/3-2(i-1))(d_v/3-2(i-1))}\p{\bA\in\cA_i}.
    \end{equation*}
    By induction then, and since increasing $i$ to $s$ only weakens this bound,
    \begin{align*}
        \p{\dist_\bK(u,v)\geq 3,\bA\in\cA}
        &= \p{\bA\in\cA_0}\\
        &\leq \left(\frac{4sm}{(d_u/3-2(s-1))(d_v/3-2(s-1))}\right)^s\p{\bA\in\cA_s}\\
        &\leq \left(\frac{48sm}{d_ud_v}\right)^s,
    \end{align*}
    where the last inequality holds since $d_u,d_v\geq 12s$ and $\p{\bA\in\cA_s}\leq 1$.
\end{proof}

\begin{cor}\label{cor:two_vtcs}
    For vertices $u,v\in[n]$ and a positive integer $s$ such that $d_u,d_v\geq 96s^2$,
    \begin{equation*}
        \p{\dist_\bK(u,v)\geq 4} \leq 3\left(\frac{393sm}{d_ud_v}\right)^s.
    \end{equation*}
\end{cor}

\begin{proof}
    For $A\in\cA_\dseq$, let $K=K(A)$, and let $N_K(u)$ denote the set of all edges in $K$ incident to $u$. We claim that either $|\{ui:i\in[d_u],\len_K(ui)\geq 2\}|\geq d_u/3$ or, for any set $S\subset[n]$ with $|S|\leq 4s$, $|(\partial N_K(u))\setminus \{xi:x\in S,i\in[d_x]\}|\geq d_u/3$. If $|V(N_K(u))|<2d_u/3$, then since all but at most 1 edge in $K$ from $u$ to a given $x\in V(N_K(u))$ can have length 1, $|\{ui:i\in[d_u],\len_K(ui)\geq 2\}|\geq d_u/3$. Suppose then that $|V(N_K(u))|\geq 2d_u/3$. Let $a$ be the number of vertices $x\in V(N_K(u))$ with at least 2 edges to $u$ in $K$, and let $b$ be the number of vertices $x\in V(N_K(u))$ with exactly one edge to $u$ in $K$. Then
    \begin{equation*}
        d_u\geq 2a+b =2(a+b) - b \geq 2(2d_u/3) -b,
    \end{equation*}
    so $b\geq d_u/3$. Each of the at least $b-4s$ vertices $x\not\in S\cup\{u\}$ with exactly one edge to $u$ in $K$ contributes at least $d_x-1\geq 2$ half-edges to $\partial N_K(u)$, and none of these are removed when we delete the half-edges incident to vertices in $S$. Hence $|(\partial N_K(u))\setminus \{xi:x\in S,i\in[d_x]\}|\geq 2(b-4s)\geq d_u/3$, where the second inequality holds since $b\geq d_u/3\geq 24s/3$. This proves the claim, and the analogous claim holds of $v$.

    Let $\cA$ be the set of all $A\in\cA_\dseq$ such that $|\{ui:i\in[d_u],\len_{A}(ui)\geq 2\}|\geq d_u/3$ and $|\{vi:i\in[d_v],\len_{A}(vi)\geq 2\}|\geq d_v/3$. We have shown that for $A\in\cA_\dseq$, either $A\in\cA$ or, for any $S\subset[n]$ with $|S|\leq 4s$ we have $|(\partial N_{K(A)}(u))\setminus\{xi:x\in S,i\in[d_x]\}|\geq d_u/3$ or, for any $S\subset[n]$ with $|S|\leq 4s$ we have $|(\partial N_{K(A)}(v))\setminus\{yi:y\in S,i\in[d_y]\}|\geq d_v/3$. Hence
    \begin{align*}
        \p{\dist_\bK(u,v)\geq 4}
        &\leq \p{\dist_\bK(u,v)\geq 4,\bA\in\cA}+\sum_{L}\p{\dist_\bK(u,v)\geq 4,N_\bK(u)=L}\\
        &+\sum_{M}\p{\dist_\bK(u,v)\geq 4,N_\bK(v)=M},
    \end{align*}
    where the sums are over all $L$ such that $\p{N_\bK(u)=L}>0$ and $|(\partial L)\setminus \{xi:x\in S,i\in[d_x]\}|\geq d_u/3$ for all $S\subset[n]$ with $|S|\leq 4s$ and over all $M$ such that $\p{N_\bK(v)=M}>0$ and $|(\partial M)\setminus \{yi:x\in S,i\in[d_y]\}|\geq d_v/3$ for all $S\subset[n]$ with $|S|\leq 4s$. For any such $M$, $N_\bK(v)=M$ if and only if $\{u\},M\subset \bK$. Letting $L=\{u\}$ so $|\partial L|=d_u$, we get
    \begin{align*}
        &\sum_{M}\p{\dist_\bK(u,v)\geq 4,N_\bK(v)=M}\\
        &\leq \sum_{M}\p{\dist_\bK(L,M)\geq 3\mid L,M\subset\bK}\p{N_\bK(v)=M}\\
        &\leq \sum_{M}\left(\frac{131sm}{d_ud_v/3}\right)^s\p{N_\bK(v)=M}\\
        &\leq \left(\frac{393sm}{d_ud_v}\right)^s\\
    \end{align*}
    where the second inequality follows from Lemma~\ref{lem:connection} and the final inequality holds because $\sum_M\p{N_\bK(v)=M}\leq 1$. A similar calculation shows that the same bound holds of the sum over $L$, and by Lemma~\ref{lem:two_vtcs},
    \begin{equation*}
        \p{\dist_\bK(u,v)\geq 4,\bA\in\cA}\leq \left(\frac{48sm}{d_ud_v}\right)^s\leq \left(\frac{393sm}{d_ud_v}\right)^s.
    \end{equation*}
    Combining the three bounds yields the desired inequality.
\end{proof}

\begin{proof}[Proof of Theorem~\ref{thm:ker_diam}]
    By Observation~\ref{obs:A=C}, we can take $\bA\in_u\cA_\dseq$, let $\bK=K(\bA)$, and bound the tail of $\diam^+(\bK)$ rather than that of $\diam^+(K(\bC))$ where $\bC\in_u\cG_\dseq^-$. For $u\in[n]$, let $(\bK_t^u,t\geq 0)$ be the breadth-first kernel exploration of $\bA$ started from $u$, and let
    \begin{equation*}
        \bt_u=\inf\{t:\rad(\bK_t^u)\leq 15\log m,|\partial\bK_t^u|\geq cm/\log^2m\}.
    \end{equation*}
    By Lemma~\ref{lem:reach_m/log^2m}, $\p{\bt_u=\infty}=O(1/m^2)$ for all $u\in[n]$.
    
    Consider a pair $u,v\in[n]$ and partial matchings $L,M$ with $\p{\bK_{\bt_u}^u=L,\bK_{\bt_v}^v=M}>0$, so $|\partial L|,|\partial M|\geq cm/\log^2m$. The absolute constant $s\geq 1$ will be chosen later. First suppose that $V(M)$ contains no vertex of degree at least $cm/(8s\log^2m)$. Then for any $S\subset[n]$ with $|S|\leq 4s$, $|(\partial M)\setminus\{xi:x\in S,i\in[d_x]\}|\geq |\partial M|-4scm/(8s\log^2m)\geq cm/(2\log^2m)$. Since $\bK_{\bt_u}^u=L$ and $\bK_{\bt_v}^v=M$ if and only if $L,M\subset\bK$, by Lemma~\ref{lem:connection},
    \begin{align*}
        \p{\dist_\bK(L,M)\geq 4\mid \bK_{\bt_u}^u=L,\bK_{\bt_v}^v=M}
        &= \p{\dist_\bK(L,M)\geq 4\mid L,M\subset\bK}\\
        &\leq \left(\frac{131sm}{(cm/\log^2m)(cm/(2\log^2m))}\right)^s\\
        &= O\left(\frac{\log^{4s}m}{m^s}\right).
    \end{align*}
    The same bound holds by similar reasoning if $V(L)$ contains no vertex of degree at least $cm/(8s\log^2m)$. Writing $E_{uv}$ for the event that $\bt_u,\bt_v<\infty$ and either $V(\bK_{\bt_u}^u)$ or $V(\bK_{\bt_v}^v)$ contains no vertex of degree at least $cm/(8s\log^2m)$, it follows that 
    \begin{equation*}
        \p{\dist_\bK(\bK_{\bt_u}^u,\bK_{\bt_v}^v)\geq 4, E_{uv}}\leq \p{\dist_\bK(\bK_{\bt_u}^u,\bK_{\bt_v}^v)\geq 4\mid E_{uv}} = O\left(\frac{\log^{4s}m}{m^s}\right).
    \end{equation*}
    
    Now for any two vertices $x,y\in[n]$ with $d_x,d_y\geq cm/(8s\log^2m)$, by Corollary~\ref{cor:two_vtcs},
    \begin{equation*}
        \p{\dist_\bK(x,y)\geq 4} \leq 3\left(\frac{393sm}{(cm/(8s\log^2m))^2}\right)^s = O\left(\frac{\log^{4s}m}{m^s}\right).
    \end{equation*}

    Now since $\log m\geq 4$, if $\bt_u,\bt_v<\infty$ but $\dist_\bK(u,v)\geq 31\log m$ then $\dist_\bK(\bK_{\bt_u}^u,\bK_{\bt_v}^v)\geq 4$. If $\bt_u,\bt_v<\infty$ and $\dist_\bK(\bK_{\bt_u}^u,\bK_{\bt_v}^v)\geq 4$ then either $E_{uv}$ occurs and $\dist_\bK(\bK_{\bt_u}^u,\bK_{\bt_v}^v)\geq 4$ or there exist $x\in V(\bK_{\bt_u}^u)$ and $y\in V(\bK_{\bt_v}^v)$ with $d_x,d_y\geq cm/(8s\log^2m)$ such that $\dist_\bK(x,y)\geq 4$. Hence,
    \begin{align*}
        \p{\diam^+(\bK)\geq 31\log m} &\leq \sum_{u\in[n]}\p{\bt_u=\infty}
        +\sum_{u,v\in[n]}\p{\dist_\bK(\bK_{\bt_u}^u,\bK_{\bt_v}^v)\geq 4, E_{uv}}\\
        &+\sum_{\substack{x,y\in[n]\\d_x,d_y\geq cm/(8s\log^2m)}}\p{\dist_\bK(x,y)\geq 4}\\
        &\leq O\left(\frac{n}{m^2}\right)+O\left(\frac{n^2\log^{4s}m}{m^s}\right)+O\left(\frac{n^2\log^{4s}m}{m^s}\right).
    \end{align*}
    Since $2m=\sum_{v\in[n]}d_v\geq3n$, the above is $O(1/m)$ if we take $s=4$. 
\end{proof}

\section{Combining the bounds, and dealing with degree two vertices}

Before reading this section, the reader may wish to reacquaint themselves with the definition of the simple kernel from Section~\ref{sec:simple_kernel}. Recall that $\grd^-$ is the set of graphs in $\grd$ with no cycle components, and that $\grd^=$ is the set of graphs in $\grd$ with no cycle components and no tree components.

The {\em simple homeomorphic reduction} of a graph $G$ is the graph $H=H(G)$ obtained from $G$ by deleting all cycle components and replacing each maximal path all of whose internal vertices lie outside $S=S(G)$ and have degree 2 by a single edge. Thus, the degree sequence of $H$ restricted to $V(H)\setminus V(S)$ is 2-free. Every immutable edge $e\in E(S)\setminus M(S)$ in $S$ is present in $H$, and we call $e$ {\em immutable in} $H$ as well. 
All other edges in $H$ are {\em mutable}. Write $M(H)=E(H)\setminus (E(S)\setminus M(S))$ and $m(H)=|M(H)|$ for the set and number of mutable edges in $H$. Then $e(H)/3\leq m(H)\leq e(H)$ because $m(S)\geq e(S)/3$. 

For each mutable edge $e\in M(H)$, let $G(e)$ denote the path in $G$ that is replaced by $e$ in constructing $H$, and let $|G(e)|$ be the number of internal vertices on $G(e)$. Because all leaves in $G$ are present in $H$, we have $S(H)=S$. In particular, we can recover the set of mutable edges $M(H)$ without making reference to $G$. For a degree sequence $\dseq=(d_1,\dots,d_n)$, we write $\grd(H)=\{G\in\grd:H(G)=H\}$, $\gmd(H)=\{G\in\gmd:H(G)=H\}$, and $\connd(H)=\{G\in\connd:H(G)=H\}$. If $H$ is connected then then $\gmd(H)=\connd(H)$.

Fix a degree sequence $\dseq=(d_1,\ldots,d_n)$ and let $H=H(G)$ for some $G\in\grd$. The data present in $G$ but not $H$ can be encoded as follows. Write $h=n-v(H)$ and $m=m(H)$ and, for convenience, suppose that $V(H)=[n]\setminus[h]$. 
Let $C$ be the the union of all cycle components of $G$. Let $P=(P_1,\dots,P_m)\in\cP_{m,h-v(C)}$ be given by $P_i=|G(e_i)|$ for $i\in[m]$, where $e_1,\dots,e_m$ are the mutable edges of $H$ listed in lexicographic order. 
For $i\in[m]$, let $\sigma^i=(\sigma_1^i,\dots,\sigma_{P_i}^i)$ be the sequence of internal vertices on the path $G(e_i)$ from $e_i^-$ to $e_i^+$. Finally, let $\sigma:[h-v(C)]\to[h]\setminus V(C)$ be the bijection where, for all $i\in[m]$, $\sigma(i)$ is equal to the $i$th entry in the concatenation of $\sigma^1,\dots,\sigma^m$. 

Let $\cS_{[h]\setminus V}$ denote the set of all bijections $[h-|V|]\to [h]\setminus V$ and let $\cR_V$ denote the set of all 2-regular graphs on $V$. Then setting
\begin{equation*}
    \cY_\dseq(H)=\bigcup_{V\subset[h]}\cR_V\times \cS_{[h]\setminus V} \times \cP_{m,h-|V|},
\end{equation*}
we have $(C,\sigma,P)\in\cY_\dseq(H)$. This encoding therefore defines a map $\phi:\grd(H)\to\cY_\dseq(H)$ which we claim is a bijection. To justify this, we describe its inverse. Given $(C,\sigma,P)\in\cY_\dseq(H)$, we can recover $\sigma^i=(\sigma_1^i,\dots,\sigma_{P_i}^i)$ for $i\in[m]$ as $(\sigma(j+1),\dots,\sigma(j+P_i))$, where $j=P_1+\cdots+P_{i-1}$. Now replace each mutable edge $e_i$ with a path from $e_i^-$ to $e_i^+$ whose sequence of internal vertices is $\sigma^i$, then take the union of the resulting graph with $C$. We obtain the unique graph $G\in\grd(H)$ which maps to $(C,\sigma,P)$ under $\phi$. 

Writing $\cyc(G)$ for the number of vertices in cycle components of $G$, if $G$ and $(C,\sigma,P)$ correspond under $\phi$, then $\cyc(G)=v(C)$ and $|G(e_i)|=P_i$ for all $i\in[m]$. In particular, $G$ has no cycle components if and only if $C=\varnothing$. Defining $\cY_\dseq^-(H)=\cS_{[h]}\times\cP_{m,h}$, the map $\phi:\grd(H)\to\cY_\dseq(H)$ therefore restricts to a bijection $\phi^-:\gmd(H)\to\cY_\dseq^-(H)$. Now observe that $|\cY_\dseq(H)|$ and $|\cY_\dseq^-(H)|$ only depend on $H$ through $h$ and $m$, which are in turn both determined by $V(H)$ and $S(H)$. Indeed $h=n-|V(H)|$ and $m=e(H)-(e(S(H))-m(S(H)))=\sum_{v\in V(H)}d_v/2-e(S(H))+m(S(H))$. If $\bG\in_u\cG_\dseq$ or $\bG\in_u\cG_\dseq^-$ then letting $\bH=H(\bG)$, it follows that for any $V\subset[n]$ and simple kernel $S$ with $\p{V(\bH)=V,S(\bH)=S}$, conditional on $V(\bH)=V$ and $S(\bH)=S$, we have $\bH\in_u\cG_{\dseq_{V}}^-(S)$. This establishes the following distributional identities.

\begin{prop}\label{prop:cycle_comps}
    Fix a degree sequence $\dseq=(d_1,\dots,d_n)$, and let $H$ be the simple homeomorphic reduction of a graph in $\grd$ with $m\geq 1$ mutable edges $e_1,\dots,e_m$ listed in lexicographic order. Let $\bG\in_u\grd(H)$, $\bG^-\in_u\gmd(H)$, $(\bC,\boldsigma,\bP)\in_u\cY_\dseq(H)$, $(\boldsigma^-,\bP^-)\in_u\cY_\dseq^-(H)$. Then
    \begin{align*}
        (\cyc(\bG),|\bG(e_1)|,\dots,|\bG(e_m)|) &\stackrel{d}{=}(v(\bC),\bP_1,\dots,\bP_m)\text{, and}\\
        (|\bG^-(e_1)|,\dots,|\bG^-(e_m)|) &\stackrel{d}{=}(\bP_1^-,\dots,\bP_m^-).
    \end{align*}
\end{prop}

When $n_2(\dseq)<(1-\eps)n$ and $\bG\in_u\grd$, Proposition~\ref{prop:cycle_comps} implies that $\cyc(\bG)$ is of order $\log(1/\eps)/\eps$ with exponential tails on scale $1/\eps$.

\begin{lem}\label{lem:cycle_tail}
    There exists $C>0$ such that for all $\eps>0$, the following holds. Let $\dseq=(d_1,\dots,d_n)$ be a degree sequence such that $n_2(\dseq)<(1-\eps)n$, and let $\bG\in_u\grd$. Then for all $x\geq 0$,
    \begin{equation*}
        \p{\cyc(\bG)\geq \frac{1}{\eps}\left(10\log\left(\frac{1}{\eps}\right)+x\right)} < \exp\left(-\frac{x}{10}+C\right).
    \end{equation*}
\end{lem}

\begin{proof}
    It suffices to prove the result with $\bG\in_u\cG_\dseq(H)$ for an arbitrary graph $H$ such that $\p{H(\bG)=H}>0$. Once this is done, averaging over $H$ recovers the lemma. Fix such a graph $H$, let $h=n-v(H)$ and let $m=m(H)$. Then $h\leq n_2(\dseq)<(1-\eps)n$ so $v(H)>\eps n$.

    Let $c_k=|\cR_{[k]}|$ be the number of $2$-regular graphs with vertex set $[k]$. It is shown in \cite[Example VI.2]{MR2483235} that $c_k=\frac{e^{-3/4}}{\sqrt{\pi k}} k!(1-5/8k+O(1/k^2))$ as $k\to\infty$, and it follows that there exists $k_0$ such that for all $k\geq k_0$, we have
    \begin{equation}\label{eq:2reg}
        \frac{c_{k+1}}{c_k} \leq \frac{k+1}{1-6/8k} \leq (1+1/k)(k+1).
    \end{equation}
    Recalling that $m\geq e(H)/3$, we have
    \begin{equation}\label{eq:6m>=epsh}
        6m \geq 2e(H) = \sum_{v\in V(H)}d_v
        \geq v(H) > \eps n\geq \eps h.
    \end{equation}
    By Proposition~\ref{prop:cycle_comps} and by definition of $\cY_\dseq(H)$, 
    \begin{equation*}
        \p{\cyc(\bG)=k} \propto \left|\bigcup_{V\subset [h],|V|=k}\cR_V\times S_{[h]\setminus V}\times\cP_{m,h-k}\right|={h\choose k}c_k(h-k)!{m-1+h-k\choose m-1}.
    \end{equation*}
    A quick calculation shows that if $k\geq 50/\eps+4$, then $(k-3)/((1+\eps/6)k-4)<(1+\eps/50)(1+\eps/6)$. If also $k\geq k_0$, then $c_{k+1}/c_k\leq (1+\eps/50)(k+1)$ by (\ref{eq:2reg}). Setting $k_1=\max(k_0,50/\eps+4)$, then for all $k\in[k_1,h)$, we have 
    \begin{align*}
       \frac{\p{\cyc(\bG)=k+1}}{\p{\cyc(\bG)=k}} &= \frac{h-k}{k+1}\cdot\frac{c_{k+1}}{c_k}\cdot\frac{1}{h-k}\cdot \frac{h-k}{m-1+h-k}\\
       &\leq \frac{c_{k+1}/c_k}{k+1}\cdot\frac{h-3}{m+h-4}\\
       &\leq \frac{c_{k+1}/c_k}{k+1}\cdot \frac{h-3}{(1+\eps/6)h-4}\\
       &\leq \frac{(1+\eps/50)^2}{1+\eps/6},
    \end{align*}
    where the second inequality comes from (\ref{eq:6m>=epsh}). Since $\eps\in(0,1]$, we have $(1+\eps/50)^2/(1+\eps/6)<1-\eps/10< \exp(-\eps/10)$. Thus, for all $k\geq k_1$,
    \begin{equation*}
        \p{\cyc(\bG)=k+1} < \exp\left(-\frac{\eps}{10}\right)\p{\cyc(\bG)=k}.
    \end{equation*}
    Applying this inequality repeatedly gives $\p{\cyc(\bG)=k}\leq \exp(-\eps(k-k_1)/10)$ for all $k\geq k_1$. This bound also holds for $k<k_1$, so for all $y\geq 0$,
    \begin{align*}
        \p{\cyc(\bG)\geq y} &= \sum_{k\geq y}\p{\cyc(\bG)=k}\\
        &< \sum_{k\geq y}\exp(-\eps(k-k_1)/10)\\
        &< \frac{\exp(\eps k_1/10)}{1-\exp(-\eps/10)}\exp\left(-\frac{\eps y}{10}\right).
    \end{align*}
    As $k_1=O(1/\eps)$, the coefficient is $O(1)/(1-\exp(-\eps /10))=O(1/\eps)$, so
    \begin{equation*}
        \p{\cyc(\bG)\geq y} \leq \exp\left(-\frac{\eps y}{10}+\log\left(\frac{1}{\eps}\right)+C\right)
    \end{equation*}
    for an absolute constant $C$. Setting $y=(10\log(1/\eps)+x)/\eps$ completes the proof.
\end{proof}

For a random graph $\bG$ as in Theorem~\ref{thm:main} or as in Theorem~\ref{thm:disco}, our next lemma, will allow us to control the diameter of $\bG$ given $H(\bG)$, and also to control the diameter of $C(\bG)$ given $S(\bG)$. 

\begin{lem}\label{lem:homeo_diam}
    Fix a graph $H$ and a set of $m\geq e(H)/3$ edges $\{e_1,\dots,e_m\}\subset E(H)$. Let $\bG$ be a random graph obtained by replacing each $e_i$ by a path $\bG(e_i)$. Suppose that $(|\bG(e_1)|,\dots,|\bG(e_m)|)\in_u\cP_{m,h}$ for some $h\geq0$. Then for all $x\geq0$,
    \begin{equation*}
        \p{\diam(\bG)\geq \left(2+\frac{3h}{m}\right)(\diam(H)+8\log m+x)} \leq 8\exp\left(-\frac{x}{4}\right).
    \end{equation*}
\end{lem}

\begin{proof}
    Notice that if $m\leq 2$ then $(2+3h/m)\diam(H)>\diam(H)+h$. But subdividing a graph $h$ times can at most increase its diameter by $h$, so $\diam(\bG)\leq\diam(H)+h$. Therefore if $m\leq 2$, then $\diam(\bG)<(2+3h/m)\diam(H)$ deterministically and there is nothing to show. For the rest of the proof, we can therefore assume that $m\geq3$.
    
    For $e\in E(H)$ write $\len_\bG(e)$ for the length of the corresponding path in $\bG$. Then $\len_\bG(e_i)=|\bG(e_i)|+1$ for $i\in[m]$, and $\len_\bG(e)=1$ for $e\in E(H)\setminus\{e_1,\dots,e_m\}$. Extend $\len_\bG$ to subsets $E\subset E(H)$ by defining $\len_\bG(E)=\sum_{e\in E}\len_\bG(e)$. Now fix $M\subset \{e_1,\dots,e_m\}$ and let $k=|M|$. With $\bP\in_u\cP_{m,h}$, the entries $(\bP_1,\dots,\bP_m)$ are exchangeable, so $\len_\bG(M) \stackrel{d}{=} \sum_{i=1}^k(\bP_i+1)$. It therefore follows (for example, from the stars and bars encoding of $\cP_{m,h}$) that $\len_\bG(M)$ has the same distribution as that of the random variable $\bX_k$ described by the following experiment. Consider an urn with $m-1$ white balls and $h$ black balls. Repeatedly sample balls from the urn without replacement until none remain. For $i\in[m-1+h]$, let $\bW_i$ be the number of white balls observed after $i$ balls in total have been drawn. Also set $\bW_0=0$ and $\bW_{m+h}=m$. Let $\bX_k=\min(i\in[0,m+h]:\bW_i=k)$. 
    
    It follows from \cite[Proposition 20.6]{MR883646} that, letting $\bW_i'$ be a $\operatorname{Binomial}(i,(m-1)/(m-1+h))$ random variable, then $\E{\phi(\bW_i)}\leq\E{\phi(\bW_i')}$ for any convex function $\phi:\mathbb{R}\to[0,\infty]$. Hence, tail bounds for $\bW_i'$ that are proved by bounding exponential moments $\E{\exp(\lambda\bW_i')}$ can also be applied to $\bW_i$. Now fix $y\geq 2$ and let $z=yk(1+h/(m-1))$. Then $\E{\bW_{\lceil z\rceil}}\geq yk\geq 2k$ and thus a Chernoff bound with $\lambda=1/2$ gives
    \begin{align}\label{eq:X_k}
    \begin{split}
        \p{\bX_k\geq yk\left(1+\frac{h}{m-1}\right)}
        &\leq \p{\bW_{\lceil z\rceil}\leq k}\\
        &\leq \p{\bW_{\lceil z\rceil}\leq \frac{1}{2}\E{\bW_{\lceil z\rceil }}}\\
        &\leq \exp\left(-\frac{yk}{8}\right).
    \end{split}
    \end{align}
    If $k=|M|\leq 8\log m+x$, then taking $y=2(8\log m+x)/k\geq2$ in (\ref{eq:X_k}) yields
    \begin{equation*}
        \p{\len_\bG(M)\geq 2(8\log m+x)\left(1+\frac{h}{m-1}\right)} \leq \exp\left(-\frac{x}{4}-2\log m\right),
    \end{equation*}
    and if $k=|M|>8\log m+x$, then taking $y=2$ in (\ref{eq:X_k}) yields
    \begin{align*}
        \p{\len_\bG(M)\geq 2|M|\left(1+\frac{h}{m-1}\right)} \leq \exp\left(-\frac{|M|}{4}\right)\leq \exp\left(-\frac{x}{4}-2\log m\right).
    \end{align*}
    In either case,
    \begin{equation}\label{eq:len(E*)}
        \p{\len_\bG(M)\geq 2\left(1+\frac{h}{m-1}\right)(|M|+8\log m+x)}\leq \exp\left(-\frac{x}{4}-2\log m\right).
    \end{equation}
    
    Now for each pair of edges $e,f\in E(H)$ in the same component of $H$, let $E_{ef}$ consist of $e$, $f$, and the edges from a shortest path joining $e^-$ to $f^-$ in $H$. Then $|E_{ef}|\leq \diam(H)+2$. Write $M_{ef}=E_{ef}\cap \{e_1,\dots,e_m\}$ and note that if $\len_\bG(M_{ef}) <  2(1+h/(m-1))(|M_{ef}|+8\log m+x-2)$, then
    \begin{align*}
        \len_\bG(E_{ef}) &= |E_{ef}| - |M_{ef}|+\len_\bG(M_{ef})\\
        &< |E_{ef}|-|M_{ef}|+2\left(1+\frac{h}{m-1}\right)(|M_{ef}|+8\log m+x-2)\\
        &\leq 2\left(1+\frac{h}{m-1}\right)(|E_{ef}|+8\log m+x-2)\\
        &\leq 2\left(1+\frac{h}{m-1}\right)(\diam(H)+8\log m+x).
    \end{align*}
    
    Each $v\in V(\bG)$ lies on $\bG(e)$ for some $e\in E(H)$, so for each pair $u,v\in V(\bG)$ in the same component of $\bG$, there exist $e,f\in E(H)$ in the same component of $H$ such that $\dist_\bG(u,v)\leq\len_\bG(E_{ef})$. A union bound and (\ref{eq:len(E*)}) then gives
    \begin{align*}
        &\p{\diam(\bG)\geq 2\left(1+\frac{h}{m-1}\right)(\diam(H)+8\log m+x)}\\
        &\leq \sum_{e,f \in E(H)}\p{\len_\bG(E_{ef})\geq 2\left(1+\frac{h}{m-1}\right)(\diam(H)+8\log m+x)}\\
        &\leq \sum_{e,f\in E(H)}\p{\len_{\bG}(M_{ef})\geq 2\left(1+\frac{h}{m-1}\right)(|M_{ef}|+8\log m+x-2)}\\
        &\leq {e(H)\choose 2}\exp\left(-\frac{x-2}{4}-2\log m\right)\\
        &\leq \frac{(3m)^2}{2}\cdot\frac{e^{1/2}}{m^2}\cdot\exp\left(-\frac{x}{4}\right)\\
        &\leq 8\exp\left(-\frac{x}{4}\right),
    \end{align*}
    where in the second to last inequality we have used that $m\geq e(H)/3$. Now $2h/(m-1)\leq 3h/m$ because $m\geq 3$, so this completes the proof.
\end{proof}

Combining Proposition~\ref{prop:bij} and Theorem~\ref{thm:tail_bound} gives a tail bound on $v(\bC)-v(S)$. Then conditional on $v(\bC)-v(S)=h$, Proposition~\ref{prop:bij} implies that $(|\bC(e_1)|,\dots,|\bC(e_m)|)\in_u\cP_{m,h}$, where $\{e_1,\dots,e_m\}=M(S)$. Thus conditional on $v(\bC)-v(S)=h$, Lemma~\ref{lem:homeo_diam} can be applied to $\bC$. This reasoning leads to the following tail bound on $\diam(\bC)$.

\begin{lem}\label{lem:core_diam}
    Fix a degree sequence $\dseq=(d_1,\dots,d_n)$ and a simple kernel $S$ with $m\geq 1$ mutable edges. Suppose that $\dseq_{[n]\setminus V(S)}$ is 2-free. Let $\bG\in_u\grd^-(S)$ and let $\bC=C(\bG)$. Then for all $x\geq 0$,
    \begin{equation*}
        \p{\diam(\bC)\geq (2+14\sqrt{n/m})(\diam(S)+8\log m+x)} \leq \exp\left(-\frac{x}{4}+3\right).
    \end{equation*}
\end{lem}

\begin{proof}
    The graph $\bG^-$ obtained from $\bG$ by removing all tree components has $C(\bG^-)=\bC$ and $S(\bG^-)=S$. In fact, for any $V\subset [n]$ with $\p{V(\bG^-)=V}>0$, conditionally given that $V(\bG^-)=V$, we have $\bG^-\in_u\cG_{\dseq_{V}}^=(S)$. The claimed bound only becomes stronger if $n$ is replaced by $|V|$, so we will carry out the proof with $\bG\in_u\grd^=(S)$ instead of $\bG\in_u\grd^-(S)$. 

    Let $(\bF,\bT,\bP)\in_u\cX_\dseq(S)$, so $v(\bC)-v(S)\eqdist\h_\bT(0)$ by Proposition~\ref{prop:bij}. Proposition~\ref{prop:bij} also implies that for any $h\geq0$, conditional on $v(\bC)-v(S)=h$, we have $(|\bC(e_1)|,\dots,|\bC(e_m)|)\in_u\cP_{m,h}$, where $e_1,\dots,e_m$ are the mutable edges of $S$ in lexicographic order. As mentioned in Section~\ref{sec:simple_kernel}, for any $0\in B\subset[0,n]$ with $\p{V(\bT)=B}>0$, there is a 1-free child sequence $\cseq=(c_v,v\in B)$ with $c_0=0$ such that conditional on $V(\bT)=B$, we have $(\bT,\bP)\in_u\{T\in\cT_\cseq,P\in\cP_{m,\h_T(0)}\}$. Now $(2\sqrt{m}+\sqrt{3y})^2/3\geq y+3$ for $y\geq 1$, so if $|B|-1\geq 64m$, then Theorem~\ref{thm:tail_bound} applied with $x=2\sqrt{m}+\sqrt{3y}$ yields
    \begin{equation}\label{eq:withB}
        \p{\h_\bT(0)\geq 4\sqrt{m(|B|-1)}+\sqrt{3y(|B|-1)}\mid V(\bT)=B} \leq \exp\left(-y+1\right).
    \end{equation}
    But if $y<1$ then this bound is greater than 1, and if $|B|-1<64m$, then $4\sqrt{m(|B|-1)}>(|B|-1)/2$, which since $\cseq$ is 1-free, is an upper bound on the height of any tree in $\cT_\cseq$. Thus (\ref{eq:withB}) holds for all $y$ and $B$. Because $|B|-1\leq n$ and $v(\bC)-v(S)\eqdist\h_\bT(0)$, averaging over $B$ yields that for all $y\geq0$,
    \begin{equation}\label{eq:C-S}
        \p{v(\bC)-v(S)\geq 4\sqrt{mn}+\sqrt{3yn}}\leq \exp\left(-y+1\right).
    \end{equation}
    On the other hand $m\geq e(S)/3$, and for any $h\geq0$, given that $v(\bC)-v(S)=h$, we have $(|\bC(e_1)|,\dots,|\bC(e_m)|)\in_u\cP_{m,h}$. Using Lemma~\ref{lem:homeo_diam} applied with $x=4y$ and averaging over $h$ yields that for all $y\geq0$, 
    \begin{equation}\label{eq:CgivenV(C)}
        \p{\diam(\bC)\geq\left(2+\frac{3(v(\bC)-v(S))}{m}\right)(\diam(S)+8\log m+4y)}\leq 8\exp\left(-y\right).
    \end{equation}

    Now for $y\leq 4m/27$, if $v(\bC)-v(S)<4\sqrt{mn}+\sqrt{3yn}$ and $\diam(\bC)<(2+3(v(\bC)-v(S))/m)(\diam(S)+8\log m+4y)$, then
    \begin{align*}
        \diam(\bC) &< (2+12\sqrt{n/m}+3\sqrt{3yn}/m)(\diam(S)+8\log m+4y)\\
        &\leq (2+14\sqrt{n/m})(\diam(S)+8\log m+4y).
    \end{align*}
    For $y>4m/27$, we have $4\sqrt{mn}+\sqrt{3yn}\leq 2\sqrt{27yn}+\sqrt{3yn}\leq 13\sqrt{(27y/(4m))yn}\leq (14\sqrt{n/m})(4y)$. Since $\diam(\bC)-\diam(S)\leq v(\bC)-v(S)$, in this case if $v(\bC)-v(S)<4\sqrt{mn}+\sqrt{3yn}$, then
    \begin{align*}
        \diam(\bC) &\leq \diam(S)+v(\bC)-v(S)\\
        &< \diam(S)+4\sqrt{mn}+\sqrt{3yn}\\
        &\leq (2+14\sqrt{n/m})(\diam(S)+8\log m+4y).
    \end{align*}
    Thus for all $y$, if $\diam(\bC)\geq (2+14\sqrt{n/m})(\diam(S)+8\log m+4y)$, then one of the events whose probability is bounded in (\ref{eq:C-S}) or (\ref{eq:CgivenV(C)}) must occur, so
    \begin{align*}
        &\p{\diam(\bC)\geq (2+14\sqrt{n/m})(\diam(S)+8\log m+4y)} \leq \exp(-y+1)+8\exp(-y)\\
        &\leq \exp(-y+3).
    \end{align*}
    Taking $y=x/4$ completes the proof.
\end{proof}

Before proving Theorem~\ref{thm:main}, we note that the number of edges in the kernel and the number of mutable edges in the simple kernel can only vary by absolute constant factors across connected graphs with a given degree sequence. More precisely, fix a degree sequence $\dseq=(d_1,\dots,d_n)$ with $n_2(\dseq)<n$. Let $G\in\connd$ and let $S=S(G)$ and $K=K(G)$. If $K$ is empty then $e(K)=m(S)=0$ and if $K$ is a loop then $e(K)=m(S)=1$. Otherwise, $K$ has minimum degree at least 3 so $s(K)\leq e(K)\leq 3s(K)$. Recall that $e(K)/2\leq m(S)\leq e(K)$. Since $n_2(\dseq)<n$, $G$ is not a cycle, so $s(K)=s(G)=s(\dseq)$. Summarizing, we have $s(\dseq)\leq e(K)\leq 3s(\dseq)$ and $s(\dseq)/2\leq m(S)\leq 3s(\dseq)$.

\begin{proof}[Proof of Theorem~\ref{thm:main}]
    Fix $\eps\in(0,1]$ and a degree sequence $\dseq=(d_1,\dots,d_n)$ such that $n_2(\dseq)<(1-\eps)n$, and let $\bG\in_u\connd$. Let $\bH=H(\bG)$, $\bF=F(\bH)$, $\bC=C(\bH)$, $\bS=S(\bH)=S(\bG)$, and $\bK=K(\bH)=K(\bG)$. Note that $v(\bH)\geq \eps n$ because all vertices of degree not equal to 2 lie in $\bH$. Let $s=s(\dseq)$, so $s\leq e(\bK)\leq 3s$ and $s/2\leq m(\bS)\leq 3s$ deterministically. Our proof breaks into four pieces. First, we use Theorem~\ref{thm:ker_diam} to show that $\E{\diam(\bS)}=O(\log(s+1)+1)$. Second, we use Lemma~\ref{lem:core_diam} to show that $\E{\diam(\bC)\mid\bS,v(\bH)}=O((1+\sqrt{v(\bH)/(s+1)})(\diam(\bS)+\log(s+1)+1))$. Third, we use Theorem~\ref{thm:ad24} to show that $\E{\h(\bF)\mid \bS,v(\bH)}=O(\sqrt{v(\bH)})$. Fourth, we use Lemma~\ref{lem:homeo_diam} to show that $\E{\diam(\bG)\mid\bH}=O((n/v(\bH))(\diam(\bH)+\log v(\bH)))$. Combining these four bounds and using that $v(\bH)\geq\eps n$ allows us to conclude.\\
    \\
    \textbf{Step 1.} Let $c$ be the implicit absolute constant from Theorem~\ref{thm:ker_diam}, and suppose for now that $s\geq 2c$. Consider a degree sequence $\dseq'$ with $\p{\dseq(C(\bG))=\dseq'}>0$, and let $\bC'\in_u\cG_{\dseq'}^-$. Then $\dseq'$ is 1-free and letting $2m$ be the sum of degrees in $\dseq'$ not equal to 2, if $C'\in\cG_{\dseq'}^-$, then $K(C')$ has $m$ edges. Moreover $\p{e(\bK)=m}>0$, so $m\geq s\geq 2c$. Conditional on $\bC'$ being connected, $\bC'\in_u\cC_{\dseq'}$, and conditional on $\dseq(C(\bG))=\dseq'$, we have $C(\bG)\in_u\cC_{\dseq'}$. Hence $(C(\bG)\mid \dseq(C(\bG))=\dseq')\eqdist(\bC'\mid\text{$\bC'$ is connected})$. Now letting $\bK'=K(\bC')$, we have $\p{\diam^+(\bK')\geq 31\log m}\leq c/m$ by Theorem~\ref{thm:ker_diam}. Since $\bK'$ is connected if and only if $\bC'$ is connected, this in particular implies $\p{\text{$\bC'$ is connected}}\geq 1-c/m\geq 1/2$. Now $K(C(\bG))=\bK$ because $s(\dseq)\geq 2c\geq 2$, so we obtain
    \begin{align*}
        \p{\diam(\bK)\geq 31\log m\mid \dseq(C(\bG))=\dseq'} &= \p{\diam(\bK')\geq 31\log m\mid\text{$\bC'$ is connected}}\\
        &\leq \frac{\p{\diam^+(\bK')\geq 31\log m}}{\p{\text{$\bC'$ is connected}}}\\
        &\leq \frac{2c}{m}
    \end{align*}
    Summing the tail of $(\diam(\bK)\mid \dseq(C(\bG))=\dseq')$ then gives
    \begin{equation*}
        \E{\diam(\bK)\mid \dseq(C(\bG))=\dseq'}\leq 31\log m+2c\leq 31\log(3s)+2c,
    \end{equation*}
    the second inequality holding because $\p{e(\bK)=m'}>0$ and $e(\bK)\leq 3s$. Averaging over $\dseq'$ and recalling that $\diam(\bS)\leq 2\diam(\bK)+2$, we obtain $\E{\diam(\bS)}\leq 2(31\log (3s)+2c)+2$. If $s<2c$, then $\diam(\bK)\leq e(\bK)\leq 3s<6c$, and hence $\diam(\bS)\leq 2(6c)+2$ deterministically. Because $c$ is an absolute constant, we conclude that $\E{\diam(\bS)}=O(\log(s+1)+1)$ regardless of whether $s\geq 2c$ or $s<2c$.\\
    \\
    \textbf{Step 2.} For now suppose that $s\geq 1$. Fix $V\subset[n]$ and a simple kernel $S$ with $\p{V(\bH)=V,\bS=S}>0$, and let $m'=m(S)$. Then $m'\geq \lceil s/2\rceil\geq 1$ and $\dseq_{V\setminus V(S)}$ is 2-free. Moreover, conditionally given that $V(\bH)=V$ and $\bS=S$, we have $\bH\in_u\cC_{\dseq_{V}}(S)$. But since $S$ is connected, $\cC_{\dseq_{V}}(S)=\cG_{\dseq_{V}}^-(S)$, so by Lemma~\ref{lem:core_diam}, for all $x\geq0$,
    \begin{align*}
        &\p{\diam(\bC)\geq (2+14\sqrt{|V|/m'})(\diam(S)+8\log m'+x)\mid V(\bH)=V,\bS=S}\\
        &\leq \exp\left(-\frac{x}{4}+3\right).
    \end{align*}
    Now $(s+1)/3\leq\lceil s/2\rceil\leq m'\leq 3s+1$, so this implies that for all $x\geq0$,
    \begin{align*}
        &\p{\diam(\bC)\geq (2+14\sqrt{3v(\bH)/(s+1)})(\diam(\bS)+8\log(3s+1)+x)\mid\bS,v(\bH)}\\
        &\leq \exp\left(-\frac{x}{4}+3\right).
    \end{align*}
    If $s=0$, then $\bC$ is empty so the above bound holds trivially. It follows immediately that $\E{\diam(\bC)\mid\bS,v(\bH)} = O((1+\sqrt{v(\bH)/(s+1)})(\diam(\bS)+\log(s+1)+1)$.\\
    \\
    \textbf{Step 3.} For a fixed $V\subset[n]$ and simple kernel $S$ with $\p{V(\bH)=V,\bS=S}>0$, conditional on $V(\bH)=V$ and $\bS=S$, we have $\bH\in_u\cG_{\dseq_V}^-(S)$. For $U\subset V$ with $\p{V(\bF)=U,V(\bH)=V,\bS=S}>0$, it then follows from the same argument as in \ref{sec:forest height} that conditional on $V(\bF)=U$, $V(\bH)=V$, and $\bS=S$, we have $\bF\in_u\cF_\cseq$, where $\cseq=(c_v,v\in U)$ is the 1-free child sequence given by $c_v=d_v-1$. By Theorem~\ref{thm:ad24}, for all $x\geq 0$, 
    \begin{equation*}
        \p{\h(\bF)\geq x\sqrt{|U|}\mid V(\bF)=U,V(\bH)=V,\bS=S}\leq 4\exp(-x^2/2^8).
    \end{equation*}
    Since $|U|\leq v(\bH)$, averaging gives $\p{\h(\bF)\geq x\sqrt{v(\bH)}\mid \bS,v(\bH)}\leq 4\exp(-x^2/2^8)$ for all $x\geq 0$. This implies $\E{\h(\bF)\mid \bS,v(\bH)}=O(\sqrt{v(\bH)})$.
    \\
    \\
    \textbf{Step 4.} Fix a graph $H$ with $\p{\bH=H}>0$, so $H$ is connected and $\cG_\dseq^-(H)=\connd(H)$. Let $m''=m(H)$ and let $h=n-v(H)$. Then $m''\geq 1$ because $n_2(\dseq)<n$. Conditionally given that $\bH=H$, we have $\bG\in_u\connd(H)=\cG_\dseq^-(H)$, and hence $(|\bG(e_1)|,\dots,|\bG(e_{m''})|)\in_u\cP_{m'',h}$ by Proposition~\ref{prop:cycle_comps}. Also $m''\geq e(H)/3$, so by Lemma~\ref{lem:homeo_diam}, for all $x\geq 0$,
    \begin{equation*}
        \p{\diam(\bG)\geq \left(2+\frac{3h}{m''}\right)(\diam(H)+8\log m''+x)\mid\bH=H} \leq 8\exp\left(-\frac{x}{4}\right).
    \end{equation*}
    Taking $\eps'=v(H)/n$, we have $v(H)\geq \eps'n$, so exactly as in (\ref{eq:6m>=epsh}) from the proof of Lemma~\ref{lem:cycle_tail}, $6m''\geq \eps'h=v(H)h/n$, and therefore $3h/m''\leq 18n/v(H)$. Also $m''\leq {v(H)\choose 2}$ so $\log m''\leq 2\log v(H)$. The preceding bound therefore gives, for all $x\geq 0$,
    \begin{equation*}
        \p{\diam(\bG)\geq \left(2+\frac{18n}{v(\bH)}\right)(\diam(\bH)+16\log v(\bH)+x)\mid\bH}\leq 8\exp\left(-\frac{x}{4}\right).
    \end{equation*}
    Because $n\geq v(\bH)\geq 2$, it follows that $\E{\diam(\bG)\mid \bH}=O((n/v(\bH))(\diam(\bH)+\log v(\bH)))$.\\
    \\
    Combining the result of Step 1 and of Step 2 with the bound $\diam(\bH)\leq 2(\h(\bF)+1)+\diam(\bC)$ gives
    \begin{equation*}
        \E{\diam(\bH)\mid\bS,v(\bH)} = O\left(\sqrt{v(\bH)}+\left(1+\sqrt{\frac{v(\bH)}{s+1}}\right)(\diam(\bS)+\log(s+1)+1)\right).
    \end{equation*}
    Now $\bH$ determines $\bS$ and $v(\bH)$, so $\E{\diam(\bG)\mid\bS,v(\bH)}=\E{\E{\diam(\bG)\mid\bH}\mid\bS,v(\bH)}$. Combining the result of Step 4 with the preceding bound therefore yields
    \begin{align*}
        &\E{\diam(\bG)\mid\bS,v(\bH)}\\
        &= O\left(\frac{n}{\sqrt{v(\bH)}}+\left(\frac{n}{v(\bH)}+\frac{n}{\sqrt{(s+1)v(\bH)}}\right)(\diam(\bS)+\log(s+1)+1)+\frac{n\log v(\bH)}{v(\bH)}\right)\\
        &= O\left(\sqrt{\frac{n}{\eps}}+\left(\frac{1}{\eps}+\sqrt{\frac{n}{\eps(s+1)}}\right)(\diam(\bS)+\log(s+1)+1)\right),
    \end{align*}
    where we have used that $\log v(\bH)=O(\sqrt{v(\bH)})$ and $v(\bH)\geq \eps n$ deterministically. Averaging over $\bS$ and $v(\bH)$ and using the result of Step 1, we obtain
    \begin{align*}
        \E{\diam(\bG)} &= O\left(\sqrt{\frac{n}{\eps}}+\left(\frac{1}{\eps}+\sqrt{\frac{n}{\eps(s+1)}}\right)(\log(s+1)+1)\right)\\
        &= O\left(\sqrt{\frac{n}{\eps}}+\frac{\log(s+1)+1}{\eps}\right)\\
        &= O\left(\sqrt{n/\eps}+\log(n)/\eps\right),
    \end{align*}
    where in the final line we have used that $s\leq {n\choose 2}$ so $\log(s+1)\leq 2\log n$.
\end{proof}

\begin{proof}[Proof of Theorem~\ref{thm:disco}]
    Fix $\eps\in(0,1]$ and a degree sequence $\dseq=(d_1,\dots,d_n)$ such that $n_2(\dseq)<(1-\eps)n$, and let $\bG\in_u\grd$. Let $\bH=H(\bG)$, $\bF=F(\bH)$, $\bC=C(\bH)$, $\bS=S(\bH)=S(\bG)$, and $\bK=K(\bH)=K(\bG)$. We proceed much like in the proof of Theorem~\ref{thm:main}, except that we work conditional on $e(\bK)$, which plays the role in this argument that surplus does in the proof of Theorem~\ref{thm:main}. We therefore fix, for the rest of the proof, some $m\geq 0$ with $\p{e(\bK)=m}>0$. Another small difference is that we need to bound the diameter of the cycle components using Lemma~\ref{lem:cycle_tail}.\\
    \\
    \textbf{Step 1.} Let $\bC'$ be obtained from $C(\bG)$ by removing all cycle components. Then for any degree sequence $\dseq'$ with $\p{\dseq(\bC')=\dseq'}>0$, conditionally given that $\dseq(\bC')=\dseq'$, we have $\bC'\in_u\cG_{\dseq'}^-$. Let $2m'$ be the sum of degrees not equal to 2 in $\dseq'$ and suppose for now that $m'\geq 1$. Then letting $\bK'=K(\bC')$, by Theorem~\ref{thm:ker_diam}, 
    \begin{equation*}
        \p{\diam^+(\bK')\geq 31\log m'\mid\dseq(\bC')=\dseq'} \leq \frac{c}{m'},
    \end{equation*}
    where $c$ is an absolute constant. Since $\diam(\bK')\leq \diam^+(\bK')$ and $\diam(\bK')\leq m'$, it follows that
    \begin{equation*}
        \E{\diam(\bK')\mid\dseq(\bC')=\dseq'} \leq 31\log (m'+1)+c.
    \end{equation*}
    The above holds trivially if $m'=0$ since then $\bK'$ is empty. The components of $\bK$ not in $\bK'$ are loops which have diameter 0, so $\diam(\bK)=\diam(\bK')$. Then since $e(\bK')\leq e(\bK)$, averaging over $\dseq$ and the loop components of $\bK$ yields $\E{\diam(\bK)\mid e(\bK)=m}=O(\log(m+1)+1)$. Since $\diam(\bS)\leq 2\diam(\bK)+2$, it follows that $\E{\diam(\bS)\mid e(\bK)=m}=O(\log(m+1)+1)$.\\
    \\
    \textbf{Step 2.} Fix $V\subset[n]$ and a simple kernel $S$ with $m''\geq 1$ mutable edges such that $\p{V(\bH)=V,\bS=S,e(\bK)=m}>0$. Then $m/2\leq m''\leq m$, $\dseq_{V\setminus V(S)}$ is 2-free, and conditionally given that $V(\bH)=V$ and $\bS=S$, we have $\bH\in_u\cG_{\dseq_V}^-(S)$. Therefore by Lemma~\ref{lem:core_diam}, for all $x\geq0$,
    \begin{align*}
        &\p{\diam(\bC)\geq (2+14\sqrt{|V|/m''})(\diam(S)+8\log m''+x)\mid V(\bH)=V,\bS=S}\\
        &\leq \exp\left(-\frac{x}{4}+3\right).
    \end{align*}
    Since $(m+1)/3\leq\lceil m/2\rceil\leq m''\leq m+1$, this implies
    \begin{align*}
        &\p{\diam(\bC)\geq (2+14\sqrt{3|V|/(m+1)})(\diam(S)+8\log(m+1)+x)\mid V(\bH)=V,\bS=S}\\
        &\leq \exp\left(-\frac{x}{4}+3\right).
    \end{align*}
    This bound also holds if $m''=0$ though, since then $\bC=\varnothing$. As $\bS$ determines $\bK$, we obtain
    \begin{align*}
        &\p{\diam(\bC)\geq (2+14\sqrt{3v(\bH)/(m+1)})(\diam(\bS)+8\log(m\!+\!1)+x)\!\mid\!\bS,v(\bH),e(\bK)=m}\\
        &\leq \exp\left(-\frac{x}{4}+3\right).
    \end{align*}
    It follows that
    \begin{equation*}
        \E{\diam(\bC)\mid\bS,v(\bH),e(\bK)=m}=O((1+\sqrt{v(\bH)/(m+1)})(\diam(\bS)+\log(m+1)+1)),
    \end{equation*}
    as needed.\\
    \\
    \textbf{Step 3.} Precisely the same argument as in Step 3 of the proof of Theorem~\ref{thm:main} shows that $\E{\h(\bF)\mid \bS,v(\bH),e(\bK)=m}=O(\sqrt{v(\bH)})$.\\
    \\
    \textbf{Step 4.} Obtain $\bG^-$ from $\bG$ by removing the cycle components. Then for any graph $H$ with $\p{\bH=H,e(\bK)=m}>0$, conditionally given that $\bH=H$, we have $\bG^-\in_u\grd^-(H)$. Because $\bH$ determines $e(\bK)$, precisely the same reasoning as in Step (4) in the proof of Theorem~\ref{thm:main} shows that $\E{\diam(\bG^-)\mid\bH,e(\bK)=m}=O((n/v(\bH))(\diam(\bH)+\log v(\bH)))$.\\
    \\
    Combining the results so far as in the proof of Theorem~\ref{thm:main} but with $s$ replaced by $m$ yields $\E{\diam(\bG^-)\mid e(\bK)=m}=O(\sqrt{n/\eps}+\log(n)/\eps)$. Averaging over $m$ then gives $\E{\diam(\bG^-)}=O(\sqrt{n/\eps}+\log(n)/\eps)$. By Lemma~\ref{lem:cycle_tail}, there is an absolute constant $c$ such that for all $x\geq 0$, $\p{\cyc(\bG)\geq (10\log(1/\eps)+x)/\eps}<\exp(-x/10+c)$. This gives $\E{\cyc(\bG)}=O(\log(1/\eps)/\eps)$. But $n_2(\dseq)<(1-\eps)n$ implies that $n_2(\dseq)\leq n-1 <(1-2/n)n$, so we can assume that $\eps\geq 2/n$. Then $\log(1/\eps)=O(\log n)$, so $\E{\cyc(\bG)}=O(\log(n)/\eps)$. Because $\diam(\bG)\leq \max(\diam(\bG^-),\cyc(\bG))$, we conclude that $\E{\diam(\bG)}=O(\sqrt{n/\eps}+\log(n)/\eps)$.
\end{proof}

\begin{lem}\label{lem:subgaussian}
    Fix $n>0$ and $n_1,\dots,n_k\geq 0$ such that $n_1+\cdots+n_k\leq n$. Also fix $\alpha\geq0$ and $\beta\geq 1$. Let $(\bX_1,\dots,\bX_k)$ be random variables such that for all $i\in[k]$ and for all $x\geq 0$, $\p{\bX_i\geq x\sqrt{n_i}} \leq \exp(-\alpha x^2+\beta)$. Then for all $x\geq 0$,
    \begin{equation*}
        \p{\max(\bX_1,\dots,\bX_k)\geq x\sqrt{n}} \leq \exp(-\alpha x^2+\beta).
    \end{equation*}
\end{lem}

\begin{proof}
    By a union bound,
    \vspace{-0.1cm}\begin{align*}
        \p{\max(\bX_1,\dots,\bX_k)\geq x\sqrt{n}} \leq \sum_{i=1}^k\p{\bX_i\geq \left(x\sqrt{n/n_i}\right)\sqrt{n_i}}
        \leq e^\beta\sum_{i=1}^k\exp(-\frac{\alpha x^2n}{n_i}).
    \end{align*}
    It is a calculus exercise to show that for any real numbers $a>0$ and $a_1,\dots,a_k\geq 0$ with $a_1+\cdots+a_k\leq a$,
    \vspace{-0.1cm}\begin{equation}\label{eq:sum_of_exp}
        \sum_{i=1}^k\exp\left(-\frac{a}{a_i}\right) \leq \exp\left(-\frac{a}{a_1+\cdots+a_k}\right).
    \end{equation}
    We may assume that $\alpha x^2\geq 1$ because $\beta\geq1$, so otherwise the claimed bound is greater than 1. Now apply (\ref{eq:sum_of_exp}) with $a=\alpha x^2n\geq n_1+\cdots+n_k$ and $(a_1,\dots,a_k)=(n_1,\dots,n_k)$.
\end{proof}

\vspace{-0.3cm}
\bibliographystyle{plain}
\bibliography{biblio}
\end{document}